\documentclass[letterpaper, 11pt, thm-restate]{article}

\usepackage{import}
\usepackage{amsthm}
\usepackage{thm-restate}
\newtheorem{theorem}{Theorem}[section]
\newtheorem{lemma}[theorem]{Lemma}
\newtheorem{proposition}[theorem]{Proposition}

\theoremstyle{definition}

\theoremstyle{definition}
\newtheorem{example}[theorem]{Example}
\theoremstyle{definition}

\usepackage{amsmath}
\usepackage{amssymb}
\usepackage{amsfonts}
\usepackage{mathtools}
\usepackage{leftindex}

\usepackage[margin=1in]{geometry}
\usepackage{changepage}

\usepackage{todonotes}
\usepackage{footmisc}
\usepackage[colorlinks = true]{hyperref}
\hypersetup{
    linkcolor=black,
    citecolor=blue,
    urlcolor=blue}
    
\usepackage{enumitem}
\usepackage{graphicx}
\usepackage{xcolor}
\usepackage[title]{appendix}
\usepackage{array}
\usepackage{extarrows}
\usepackage[ruled]{algorithm}

\usepackage{titlesec}
\titleformat{\subsection}[runin]
       {\normalfont\bfseries}
       {\thesubsection}
       {0.5em}
       {}
       [.]

\allowdisplaybreaks
\newcommand{\Z}{\mathbb{Z}}
\newcommand{\N}{\mathbb{N}}

\newcommand{\Sym}{\operatorname{Sym}}
\newcommand{\FSym}{\operatorname{FSym}}

\newcommand{\FAlt}{\operatorname{FAlt}}

\newcommand{\SL}{\mathsf{SL}}
\DeclareMathOperator{\supp}{supp}
\DeclareMathOperator{\sgn}{sgn}
\DeclareMathOperator{\dist}{dist}
\DeclareMathOperator{\diam}{diam}
\DeclareMathOperator{\id}{id}

\newcommand{\la}{\left\langle}
\newcommand{\ra}{\right\rangle}
\newcommand{\gen}[1]{\left\langle {#1} \right\rangle}
\newcommand{\card}[1]{\left\lvert #1 \right\rvert}
\newcommand{\abs}[1]{\left\lvert #1 \right\rvert}

\newcommand{\mF}{\mathcal{F}}

\newcommand{\mA}{\mathcal{A}}

\newcommand{\mI}{\mathcal{I}}
\newcommand{\mO}{\mathcal{O}}

\DeclareMathOperator{\fin}{fin}

\newcommand{\say}[1]{``#1''}
\newcommand{\figureexampleone}{

\tikzset{every picture/.style={line width=0.75pt}} 

\begin{tikzpicture}[x=0.75pt,y=0.75pt,yscale=-1,xscale=1]

\draw  [color={rgb, 255:red, 0; green, 0; blue, 0 }  ,draw opacity=1 ][fill={rgb, 255:red, 155; green, 155; blue, 155 }  ,fill opacity=0.4 ] (250,50) .. controls (250,44.48) and (254.48,40) .. (260,40) .. controls (265.52,40) and (270,44.48) .. (270,50) .. controls (270,55.52) and (265.52,60) .. (260,60) .. controls (254.48,60) and (250,55.52) .. (250,50) -- cycle ;
\draw  [color={rgb, 255:red, 0; green, 0; blue, 0 }  ,draw opacity=1 ][fill={rgb, 255:red, 155; green, 155; blue, 155 }  ,fill opacity=0.4 ] (280,50) .. controls (280,44.48) and (284.48,40) .. (290,40) .. controls (295.52,40) and (300,44.48) .. (300,50) .. controls (300,55.52) and (295.52,60) .. (290,60) .. controls (284.48,60) and (280,55.52) .. (280,50) -- cycle ;
\draw  [color={rgb, 255:red, 0; green, 0; blue, 0 }  ,draw opacity=1 ][fill={rgb, 255:red, 155; green, 155; blue, 155 }  ,fill opacity=0.4 ] (310,50) .. controls (310,44.48) and (314.48,40) .. (320,40) .. controls (325.52,40) and (330,44.48) .. (330,50) .. controls (330,55.52) and (325.52,60) .. (320,60) .. controls (314.48,60) and (310,55.52) .. (310,50) -- cycle ;
\draw [color={rgb, 255:red, 0; green, 0; blue, 0 }  ,draw opacity=1 ]   (260,90) -- (288,90) ;
\draw [shift={(290,90)}, rotate = 180] [color={rgb, 255:red, 0; green, 0; blue, 0 }  ,draw opacity=1 ][line width=0.75]    (7.65,-2.3) .. controls (4.86,-0.97) and (2.31,-0.21) .. (0,0) .. controls (2.31,0.21) and (4.86,0.98) .. (7.65,2.3)   ;
\draw [color={rgb, 255:red, 208; green, 2; blue, 27 }  ,draw opacity=1 ]   (290,50) .. controls (280.96,60.04) and (271.63,60.69) .. (262.01,51.96) ;
\draw [shift={(260,50)}, rotate = 46.09] [fill={rgb, 255:red, 208; green, 2; blue, 27 }  ,fill opacity=1 ][line width=0.08]  [draw opacity=0] (5.36,-2.57) -- (0,0) -- (5.36,2.57) -- cycle    ;
\draw  [color={rgb, 255:red, 0; green, 0; blue, 0 }  ,draw opacity=1 ][fill={rgb, 255:red, 155; green, 155; blue, 155 }  ,fill opacity=0.4 ] (340,50) .. controls (340,44.48) and (344.48,40) .. (350,40) .. controls (355.52,40) and (360,44.48) .. (360,50) .. controls (360,55.52) and (355.52,60) .. (350,60) .. controls (344.48,60) and (340,55.52) .. (340,50) -- cycle ;
\draw  [color={rgb, 255:red, 0; green, 0; blue, 0 }  ,draw opacity=1 ][fill={rgb, 255:red, 155; green, 155; blue, 155 }  ,fill opacity=0.4 ] (370,50) .. controls (370,44.48) and (374.48,40) .. (380,40) .. controls (385.52,40) and (390,44.48) .. (390,50) .. controls (390,55.52) and (385.52,60) .. (380,60) .. controls (374.48,60) and (370,55.52) .. (370,50) -- cycle ;
\draw  [color={rgb, 255:red, 0; green, 0; blue, 0 }  ,draw opacity=1 ][fill={rgb, 255:red, 155; green, 155; blue, 155 }  ,fill opacity=0.4 ] (400,50) .. controls (400,44.48) and (404.48,40) .. (410,40) .. controls (415.52,40) and (420,44.48) .. (420,50) .. controls (420,55.52) and (415.52,60) .. (410,60) .. controls (404.48,60) and (400,55.52) .. (400,50) -- cycle ;
\draw [color={rgb, 255:red, 208; green, 2; blue, 27 }  ,draw opacity=1 ]   (260,50) .. controls (269.25,40.31) and (278.79,40.33) .. (287.81,47.99) ;
\draw [shift={(290,50)}, rotate = 224.8] [fill={rgb, 255:red, 208; green, 2; blue, 27 }  ,fill opacity=1 ][line width=0.08]  [draw opacity=0] (5.36,-2.57) -- (0,0) -- (5.36,2.57) -- cycle    ;
\draw [color={rgb, 255:red, 208; green, 2; blue, 27 }  ,draw opacity=0.5 ] [dash pattern={on 3pt off 1.5pt}]  (320,50) .. controls (310.96,60.04) and (301.63,60.69) .. (292.01,51.96) ;
\draw [shift={(290,50)}, rotate = 46.09] [fill={rgb, 255:red, 208; green, 2; blue, 27 }  ,fill opacity=0.5 ][line width=0.08]  [draw opacity=0] (5.36,-2.57) -- (0,0) -- (5.36,2.57) -- cycle    ;
\draw [color={rgb, 255:red, 208; green, 2; blue, 27 }  ,draw opacity=0.5 ] [dash pattern={on 3pt off 1.5pt}]  (290,50) .. controls (299.25,40.31) and (308.79,40.33) .. (317.81,47.99) ;
\draw [shift={(320,50)}, rotate = 224.8] [fill={rgb, 255:red, 208; green, 2; blue, 27 }  ,fill opacity=0.5 ][line width=0.08]  [draw opacity=0] (5.36,-2.57) -- (0,0) -- (5.36,2.57) -- cycle    ;
\draw [color={rgb, 255:red, 208; green, 2; blue, 27 }  ,draw opacity=0.5 ] [dash pattern={on 3pt off 1.5pt}]  (350,50) .. controls (359.25,40.31) and (368.79,40.33) .. (377.81,47.99) ;
\draw [shift={(380,50)}, rotate = 224.8] [fill={rgb, 255:red, 208; green, 2; blue, 27 }  ,fill opacity=0.5 ][line width=0.08]  [draw opacity=0] (5.36,-2.57) -- (0,0) -- (5.36,2.57) -- cycle    ;
\draw [color={rgb, 255:red, 208; green, 2; blue, 27 }  ,draw opacity=0.5 ] [dash pattern={on 3pt off 1.5pt}]  (380,50) .. controls (370.96,60.04) and (361.63,60.69) .. (352.01,51.96) ;
\draw [shift={(350,50)}, rotate = 46.09] [fill={rgb, 255:red, 208; green, 2; blue, 27 }  ,fill opacity=0.5 ][line width=0.08]  [draw opacity=0] (5.36,-2.57) -- (0,0) -- (5.36,2.57) -- cycle    ;
\draw [color={rgb, 255:red, 208; green, 2; blue, 27 }  ,draw opacity=0.5 ]   (320,50) .. controls (329.25,40.31) and (338.79,40.33) .. (347.81,47.99) ;
\draw [shift={(350,50)}, rotate = 224.8] [fill={rgb, 255:red, 208; green, 2; blue, 27 }  ,fill opacity=0.5 ][line width=0.08]  [draw opacity=0] (5.36,-2.57) -- (0,0) -- (5.36,2.57) -- cycle    ;
\draw [color={rgb, 255:red, 208; green, 2; blue, 27 }  ,draw opacity=0.5 ]   (350,50) .. controls (340.96,60.04) and (331.63,60.69) .. (322.01,51.96) ;
\draw [shift={(320,50)}, rotate = 46.09] [fill={rgb, 255:red, 208; green, 2; blue, 27 }  ,fill opacity=0.5 ][line width=0.08]  [draw opacity=0] (5.36,-2.57) -- (0,0) -- (5.36,2.57) -- cycle    ;
\draw [color={rgb, 255:red, 208; green, 2; blue, 27 }  ,draw opacity=0.5 ]   (380,50) .. controls (389.25,40.31) and (398.79,40.33) .. (407.81,47.99) ;
\draw [shift={(410,50)}, rotate = 224.8] [fill={rgb, 255:red, 208; green, 2; blue, 27 }  ,fill opacity=0.5 ][line width=0.08]  [draw opacity=0] (5.36,-2.57) -- (0,0) -- (5.36,2.57) -- cycle    ;
\draw [color={rgb, 255:red, 208; green, 2; blue, 27 }  ,draw opacity=0.5 ]   (410,50) .. controls (400.96,60.04) and (391.63,60.69) .. (382.01,51.96) ;
\draw [shift={(380,50)}, rotate = 46.09] [fill={rgb, 255:red, 208; green, 2; blue, 27 }  ,fill opacity=0.5 ][line width=0.08]  [draw opacity=0] (5.36,-2.57) -- (0,0) -- (5.36,2.57) -- cycle    ;

\draw (267,92.4) node [anchor=north west][inner sep=0.75pt]  [color={rgb, 255:red, 0; green, 0; blue, 0 }  ,opacity=1 ]  {$\rho $};
\draw (268,27.4) node [anchor=north west][inner sep=0.75pt]  [color={rgb, 255:red, 208; green, 2; blue, 27 }  ,opacity=1 ]  {$\tau $};
\draw (287,58.4) node [anchor=north west][inner sep=0.75pt]  [font=\footnotesize,color={rgb, 255:red, 208; green, 2; blue, 27 }  ,opacity=0.5 ]  {$\rho \tau \rho ^{-1}$};
\draw (257,73.4) node [anchor=north west][inner sep=0.75pt]  [color={rgb, 255:red, 0; green, 0; blue, 0 }  ,opacity=1 ]  {$0$};
\draw (287,73.4) node [anchor=north west][inner sep=0.75pt]  [color={rgb, 255:red, 0; green, 0; blue, 0 }  ,opacity=1 ]  {$1$};
\draw (317,73.4) node [anchor=north west][inner sep=0.75pt]  [color={rgb, 255:red, 0; green, 0; blue, 0 }  ,opacity=1 ]  {$2$};
\draw (348,73.4) node [anchor=north west][inner sep=0.75pt]  [color={rgb, 255:red, 0; green, 0; blue, 0 }  ,opacity=1 ]  {$3$};
\draw (377,73.4) node [anchor=north west][inner sep=0.75pt]  [color={rgb, 255:red, 0; green, 0; blue, 0 }  ,opacity=1 ]  {$4$};
\draw (407,73.4) node [anchor=north west][inner sep=0.75pt]  [color={rgb, 255:red, 0; green, 0; blue, 0 }  ,opacity=1 ]  {$5$};
\draw (317,22.4) node [anchor=north west][inner sep=0.75pt]  [font=\footnotesize,color={rgb, 255:red, 208; green, 2; blue, 27 }  ,opacity=0.5 ]  {$\rho ^{2} \tau \rho ^{-2}$};
\draw (346,58.4) node [anchor=north west][inner sep=0.75pt]  [font=\footnotesize,color={rgb, 255:red, 208; green, 2; blue, 27 }  ,opacity=0.5 ]  {$\rho ^{3} \tau \rho ^{-3}$};
\draw (377,23.4) node [anchor=north west][inner sep=0.75pt]  [font=\footnotesize,color={rgb, 255:red, 208; green, 2; blue, 27 }  ,opacity=0.5 ]  {$\rho ^{4} \tau \rho ^{-4}$};

\end{tikzpicture}

}
\newcommand{\figureexampletwo}{

\tikzset{every picture/.style={line width=0.75pt}} 

\begin{tikzpicture}[x=0.7pt,y=0.7pt,yscale=-1,xscale=1]

\draw  [color={rgb, 255:red, 0; green, 0; blue, 0 }  ,draw opacity=1 ][fill={rgb, 255:red, 155; green, 155; blue, 155 }  ,fill opacity=0.4 ] (250,40) .. controls (250,34.48) and (254.48,30) .. (260,30) .. controls (265.52,30) and (270,34.48) .. (270,40) .. controls (270,45.52) and (265.52,50) .. (260,50) .. controls (254.48,50) and (250,45.52) .. (250,40) -- cycle ;
\draw  [color={rgb, 255:red, 0; green, 0; blue, 0 }  ,draw opacity=1 ][fill={rgb, 255:red, 155; green, 155; blue, 155 }  ,fill opacity=0.4 ] (280,70) .. controls (280,64.48) and (284.48,60) .. (290,60) .. controls (295.52,60) and (300,64.48) .. (300,70) .. controls (300,75.52) and (295.52,80) .. (290,80) .. controls (284.48,80) and (280,75.52) .. (280,70) -- cycle ;
\draw [color={rgb, 255:red, 0; green, 0; blue, 0 }  ,draw opacity=1 ]   (260,105) -- (318,105) ;
\draw [shift={(320,105)}, rotate = 180] [color={rgb, 255:red, 0; green, 0; blue, 0 }  ,draw opacity=1 ][line width=0.75]    (7.65,-2.3) .. controls (4.86,-0.97) and (2.31,-0.21) .. (0,0) .. controls (2.31,0.21) and (4.86,0.98) .. (7.65,2.3)   ;
\draw [color={rgb, 255:red, 208; green, 2; blue, 27 }  ,draw opacity=1 ]   (290,70) .. controls (271.28,70) and (261.14,61.61) .. (260.09,42.74) ;
\draw [shift={(260,40)}, rotate = 89.45] [fill={rgb, 255:red, 208; green, 2; blue, 27 }  ,fill opacity=1 ][line width=0.08]  [draw opacity=0] (5.36,-2.57) -- (0,0) -- (5.36,2.57) -- cycle    ;
\draw [color={rgb, 255:red, 208; green, 2; blue, 27 }  ,draw opacity=1 ]   (260,40) .. controls (279,40.19) and (288.98,49.04) .. (289.93,67.08) ;
\draw [shift={(290,70)}, rotate = 270] [fill={rgb, 255:red, 208; green, 2; blue, 27 }  ,fill opacity=1 ][line width=0.08]  [draw opacity=0] (5.36,-2.57) -- (0,0) -- (5.36,2.57) -- cycle    ;
\draw  [color={rgb, 255:red, 0; green, 0; blue, 0 }  ,draw opacity=1 ][fill={rgb, 255:red, 155; green, 155; blue, 155 }  ,fill opacity=0.4 ] (310,40) .. controls (310,34.48) and (314.48,30) .. (320,30) .. controls (325.52,30) and (330,34.48) .. (330,40) .. controls (330,45.52) and (325.52,50) .. (320,50) .. controls (314.48,50) and (310,45.52) .. (310,40) -- cycle ;
\draw  [color={rgb, 255:red, 0; green, 0; blue, 0 }  ,draw opacity=1 ][fill={rgb, 255:red, 155; green, 155; blue, 155 }  ,fill opacity=0.4 ] (340,70) .. controls (340,64.48) and (344.48,60) .. (350,60) .. controls (355.52,60) and (360,64.48) .. (360,70) .. controls (360,75.52) and (355.52,80) .. (350,80) .. controls (344.48,80) and (340,75.52) .. (340,70) -- cycle ;
\draw [color={rgb, 255:red, 208; green, 2; blue, 27 }  ,draw opacity=0.5 ]   (350,70) .. controls (331.28,70) and (321.14,61.61) .. (320.09,42.74) ;
\draw [shift={(320,40)}, rotate = 89.45] [fill={rgb, 255:red, 208; green, 2; blue, 27 }  ,fill opacity=0.5 ][line width=0.08]  [draw opacity=0] (5.36,-2.57) -- (0,0) -- (5.36,2.57) -- cycle    ;
\draw [color={rgb, 255:red, 208; green, 2; blue, 27 }  ,draw opacity=0.5 ]   (320,40) .. controls (339,40.19) and (348.98,49.04) .. (349.93,67.08) ;
\draw [shift={(350,70)}, rotate = 270] [fill={rgb, 255:red, 208; green, 2; blue, 27 }  ,fill opacity=0.5 ][line width=0.08]  [draw opacity=0] (5.36,-2.57) -- (0,0) -- (5.36,2.57) -- cycle    ;
\draw  [color={rgb, 255:red, 0; green, 0; blue, 0 }  ,draw opacity=1 ][fill={rgb, 255:red, 155; green, 155; blue, 155 }  ,fill opacity=0.4 ] (370,40) .. controls (370,34.48) and (374.48,30) .. (380,30) .. controls (385.52,30) and (390,34.48) .. (390,40) .. controls (390,45.52) and (385.52,50) .. (380,50) .. controls (374.48,50) and (370,45.52) .. (370,40) -- cycle ;
\draw  [color={rgb, 255:red, 0; green, 0; blue, 0 }  ,draw opacity=1 ][fill={rgb, 255:red, 155; green, 155; blue, 155 }  ,fill opacity=0.4 ] (400,70) .. controls (400,64.48) and (404.48,60) .. (410,60) .. controls (415.52,60) and (420,64.48) .. (420,70) .. controls (420,75.52) and (415.52,80) .. (410,80) .. controls (404.48,80) and (400,75.52) .. (400,70) -- cycle ;
\draw [color={rgb, 255:red, 208; green, 2; blue, 27 }  ,draw opacity=0.5 ]   (410,70) .. controls (391.28,70) and (381.14,61.61) .. (380.09,42.74) ;
\draw [shift={(380,40)}, rotate = 89.45] [fill={rgb, 255:red, 208; green, 2; blue, 27 }  ,fill opacity=0.5 ][line width=0.08]  [draw opacity=0] (5.36,-2.57) -- (0,0) -- (5.36,2.57) -- cycle    ;
\draw [color={rgb, 255:red, 208; green, 2; blue, 27 }  ,draw opacity=0.5 ]   (380,40) .. controls (399,40.19) and (408.98,49.04) .. (409.93,67.08) ;
\draw [shift={(410,70)}, rotate = 270] [fill={rgb, 255:red, 208; green, 2; blue, 27 }  ,fill opacity=0.5 ][line width=0.08]  [draw opacity=0] (5.36,-2.57) -- (0,0) -- (5.36,2.57) -- cycle    ;

\draw (286,107.4) node [anchor=north west][inner sep=0.75pt]  [color={rgb, 255:red, 0; green, 0; blue, 0 }  ,opacity=1 ]  {$\rho $};
\draw (281,38.4) node [anchor=north west][inner sep=0.75pt]  [color={rgb, 255:red, 208; green, 2; blue, 27 }  ,opacity=1 ]  {$\tau $};
\draw (257,88.4) node [anchor=north west][inner sep=0.75pt]  [color={rgb, 255:red, 0; green, 0; blue, 0 }  ,opacity=1 ]  {$0$};
\draw (287,88.4) node [anchor=north west][inner sep=0.75pt]  [color={rgb, 255:red, 0; green, 0; blue, 0 }  ,opacity=1 ]  {$1$};
\draw (317,88.4) node [anchor=north west][inner sep=0.75pt]  [color={rgb, 255:red, 0; green, 0; blue, 0 }  ,opacity=1 ]  {$2$};
\draw (348,88.4) node [anchor=north west][inner sep=0.75pt]  [color={rgb, 255:red, 0; green, 0; blue, 0 }  ,opacity=1 ]  {$3$};
\draw (377,88.4) node [anchor=north west][inner sep=0.75pt]  [color={rgb, 255:red, 0; green, 0; blue, 0 }  ,opacity=1 ]  {$4$};
\draw (407,88.4) node [anchor=north west][inner sep=0.75pt]  [color={rgb, 255:red, 0; green, 0; blue, 0 }  ,opacity=1 ]  {$5$};
\draw (331,32.4) node [anchor=north west][inner sep=0.75pt]  [font=\footnotesize,color={rgb, 255:red, 208; green, 2; blue, 27 }  ,opacity=0.5 ]  {$\rho \tau \rho ^{-1}$};
\draw (391,32.4) node [anchor=north west][inner sep=0.75pt]  [font=\footnotesize,color={rgb, 255:red, 208; green, 2; blue, 27 }  ,opacity=0.5 ]  {$\rho ^{2} \tau \rho ^{-2}$};
\draw (230,32.4) node [anchor=north west][inner sep=0.75pt]  [color={rgb, 255:red, 0; green, 0; blue, 0 }  ,opacity=1 ]  {$\mathcal{I}_{1}{}$};
\draw (230,60.4) node [anchor=north west][inner sep=0.75pt]  [color={rgb, 255:red, 0; green, 0; blue, 0 }  ,opacity=1 ]  {$\mathcal{I}_{2}{}$};

\end{tikzpicture}

}
\newcommand{\figureexamplethree}{

\tikzset{every picture/.style={line width=0.75pt}} 

\begin{tikzpicture}[x=0.7pt,y=0.7pt,yscale=-1,xscale=1]

\draw  [color={rgb, 255:red, 0; green, 0; blue, 0 }  ,draw opacity=1 ][fill={rgb, 255:red, 155; green, 155; blue, 155 }  ,fill opacity=0.4 ] (250,40) .. controls (250,34.48) and (254.48,30) .. (260,30) .. controls (265.52,30) and (270,34.48) .. (270,40) .. controls (270,45.52) and (265.52,50) .. (260,50) .. controls (254.48,50) and (250,45.52) .. (250,40) -- cycle ;
\draw  [color={rgb, 255:red, 0; green, 0; blue, 0 }  ,draw opacity=1 ][fill={rgb, 255:red, 155; green, 155; blue, 155 }  ,fill opacity=0.4 ] (280,70) .. controls (280,64.48) and (284.48,60) .. (290,60) .. controls (295.52,60) and (300,64.48) .. (300,70) .. controls (300,75.52) and (295.52,80) .. (290,80) .. controls (284.48,80) and (280,75.52) .. (280,70) -- cycle ;
\draw [color={rgb, 255:red, 0; green, 0; blue, 0 }  ,draw opacity=1 ]   (260,105) -- (318,105) ;
\draw [shift={(320,105)}, rotate = 180] [color={rgb, 255:red, 0; green, 0; blue, 0 }  ,draw opacity=1 ][line width=0.75]    (7.65,-2.3) .. controls (4.86,-0.97) and (2.31,-0.21) .. (0,0) .. controls (2.31,0.21) and (4.86,0.98) .. (7.65,2.3)   ;
\draw [color={rgb, 255:red, 208; green, 2; blue, 27 }  ,draw opacity=1 ]   (290,70) .. controls (271.28,70) and (261.14,61.61) .. (260.09,42.74) ;
\draw [shift={(260,40)}, rotate = 89.45] [fill={rgb, 255:red, 208; green, 2; blue, 27 }  ,fill opacity=1 ][line width=0.08]  [draw opacity=0] (5.36,-2.57) -- (0,0) -- (5.36,2.57) -- cycle    ;
\draw [color={rgb, 255:red, 208; green, 2; blue, 27 }  ,draw opacity=1 ]   (260,40) .. controls (279,40.19) and (288.98,49.04) .. (289.93,67.08) ;
\draw [shift={(290,70)}, rotate = 270] [fill={rgb, 255:red, 208; green, 2; blue, 27 }  ,fill opacity=1 ][line width=0.08]  [draw opacity=0] (5.36,-2.57) -- (0,0) -- (5.36,2.57) -- cycle    ;
\draw  [color={rgb, 255:red, 0; green, 0; blue, 0 }  ,draw opacity=1 ][fill={rgb, 255:red, 155; green, 155; blue, 155 }  ,fill opacity=0.4 ] (310,40) .. controls (310,34.48) and (314.48,30) .. (320,30) .. controls (325.52,30) and (330,34.48) .. (330,40) .. controls (330,45.52) and (325.52,50) .. (320,50) .. controls (314.48,50) and (310,45.52) .. (310,40) -- cycle ;
\draw  [color={rgb, 255:red, 0; green, 0; blue, 0 }  ,draw opacity=1 ][fill={rgb, 255:red, 155; green, 155; blue, 155 }  ,fill opacity=0.4 ] (340,70) .. controls (340,64.48) and (344.48,60) .. (350,60) .. controls (355.52,60) and (360,64.48) .. (360,70) .. controls (360,75.52) and (355.52,80) .. (350,80) .. controls (344.48,80) and (340,75.52) .. (340,70) -- cycle ;
\draw [color={rgb, 255:red, 208; green, 2; blue, 27 }  ,draw opacity=0.5 ]   (350,70) .. controls (331.28,70) and (321.14,61.61) .. (320.09,42.74) ;
\draw [shift={(320,40)}, rotate = 89.45] [fill={rgb, 255:red, 208; green, 2; blue, 27 }  ,fill opacity=0.5 ][line width=0.08]  [draw opacity=0] (5.36,-2.57) -- (0,0) -- (5.36,2.57) -- cycle    ;
\draw [color={rgb, 255:red, 208; green, 2; blue, 27 }  ,draw opacity=0.5 ]   (320,40) .. controls (339,40.19) and (348.98,49.04) .. (349.93,67.08) ;
\draw [shift={(350,70)}, rotate = 270] [fill={rgb, 255:red, 208; green, 2; blue, 27 }  ,fill opacity=0.5 ][line width=0.08]  [draw opacity=0] (5.36,-2.57) -- (0,0) -- (5.36,2.57) -- cycle    ;
\draw  [color={rgb, 255:red, 0; green, 0; blue, 0 }  ,draw opacity=1 ][fill={rgb, 255:red, 155; green, 155; blue, 155 }  ,fill opacity=0.4 ] (370,40) .. controls (370,34.48) and (374.48,30) .. (380,30) .. controls (385.52,30) and (390,34.48) .. (390,40) .. controls (390,45.52) and (385.52,50) .. (380,50) .. controls (374.48,50) and (370,45.52) .. (370,40) -- cycle ;
\draw  [color={rgb, 255:red, 0; green, 0; blue, 0 }  ,draw opacity=1 ][fill={rgb, 255:red, 155; green, 155; blue, 155 }  ,fill opacity=0.4 ] (400,70) .. controls (400,64.48) and (404.48,60) .. (410,60) .. controls (415.52,60) and (420,64.48) .. (420,70) .. controls (420,75.52) and (415.52,80) .. (410,80) .. controls (404.48,80) and (400,75.52) .. (400,70) -- cycle ;
\draw [color={rgb, 255:red, 208; green, 2; blue, 27 }  ,draw opacity=0.5 ]   (410,70) .. controls (391.28,70) and (381.14,61.61) .. (380.09,42.74) ;
\draw [shift={(380,40)}, rotate = 89.45] [fill={rgb, 255:red, 208; green, 2; blue, 27 }  ,fill opacity=0.5 ][line width=0.08]  [draw opacity=0] (5.36,-2.57) -- (0,0) -- (5.36,2.57) -- cycle    ;
\draw [color={rgb, 255:red, 208; green, 2; blue, 27 }  ,draw opacity=0.5 ]   (380,40) .. controls (399,40.19) and (408.98,49.04) .. (409.93,67.08) ;
\draw [shift={(410,70)}, rotate = 270] [fill={rgb, 255:red, 208; green, 2; blue, 27 }  ,fill opacity=0.5 ][line width=0.08]  [draw opacity=0] (5.36,-2.57) -- (0,0) -- (5.36,2.57) -- cycle    ;
\draw  [color={rgb, 255:red, 74; green, 144; blue, 226 }  ,draw opacity=1 ] (248.5,28.78) -- (301.5,28.78) -- (301.5,81.28) -- (248.5,81.28) -- cycle ;
\draw  [color={rgb, 255:red, 74; green, 144; blue, 226 }  ,draw opacity=1 ] (308.75,28.78) -- (361.75,28.78) -- (361.75,81.28) -- (308.75,81.28) -- cycle ;
\draw  [color={rgb, 255:red, 74; green, 144; blue, 226 }  ,draw opacity=0.5 ] (368.75,28.78) -- (421.75,28.78) -- (421.75,81.28) -- (368.75,81.28) -- cycle ;
\draw [color={rgb, 255:red, 74; green, 144; blue, 226 }  ,draw opacity=1 ]   (280.5,28.78) .. controls (291.78,18.2) and (314.77,17.57) .. (327.66,26.25) ;
\draw [shift={(330,28.03)}, rotate = 220.6] [fill={rgb, 255:red, 74; green, 144; blue, 226 }  ,fill opacity=1 ][line width=0.08]  [draw opacity=0] (5.36,-2.57) -- (0,0) -- (5.36,2.57) -- cycle    ;
\draw [color={rgb, 255:red, 74; green, 144; blue, 226 }  ,draw opacity=1 ]   (282.1,83.64) .. controls (295.12,92.79) and (319.15,92.54) .. (329.25,81.03) ;
\draw [shift={(279.75,81.78)}, rotate = 41.92] [fill={rgb, 255:red, 74; green, 144; blue, 226 }  ,fill opacity=1 ][line width=0.08]  [draw opacity=0] (5.36,-2.57) -- (0,0) -- (5.36,2.57) -- cycle    ;
\draw [color={rgb, 255:red, 74; green, 144; blue, 226 }  ,draw opacity=0.5 ] [dash pattern={on 3pt off 1.5pt}]  (340.25,28.78) .. controls (351.53,18.2) and (374.52,17.57) .. (387.41,26.25) ;
\draw [shift={(389.75,28.03)}, rotate = 220.6] [fill={rgb, 255:red, 74; green, 144; blue, 226 }  ,fill opacity=0.5 ][line width=0.08]  [draw opacity=0] (5.36,-2.57) -- (0,0) -- (5.36,2.57) -- cycle    ;
\draw [color={rgb, 255:red, 74; green, 144; blue, 226 }  ,draw opacity=0.5 ] [dash pattern={on 3pt off 1.5pt}]  (342.85,83.89) .. controls (355.87,93.04) and (379.9,92.79) .. (390,81.28) ;
\draw [shift={(340.5,82.03)}, rotate = 41.92] [fill={rgb, 255:red, 74; green, 144; blue, 226 }  ,fill opacity=0.5 ][line width=0.08]  [draw opacity=0] (5.36,-2.57) -- (0,0) -- (5.36,2.57) -- cycle    ;

\draw (286,107.4) node [anchor=north west][inner sep=0.75pt]  [color={rgb, 255:red, 0; green, 0; blue, 0 }  ,opacity=1 ]  {$\rho $};
\draw (281,38.4) node [anchor=north west][inner sep=0.75pt]  [color={rgb, 255:red, 208; green, 2; blue, 27 }  ,opacity=1 ]  {$\tau $};
\draw (257,88.4) node [anchor=north west][inner sep=0.75pt]  [color={rgb, 255:red, 0; green, 0; blue, 0 }  ,opacity=1 ]  {$0$};
\draw (287,88.4) node [anchor=north west][inner sep=0.75pt]  [color={rgb, 255:red, 0; green, 0; blue, 0 }  ,opacity=1 ]  {$1$};
\draw (317,88.4) node [anchor=north west][inner sep=0.75pt]  [color={rgb, 255:red, 0; green, 0; blue, 0 }  ,opacity=1 ]  {$2$};
\draw (348,88.4) node [anchor=north west][inner sep=0.75pt]  [color={rgb, 255:red, 0; green, 0; blue, 0 }  ,opacity=1 ]  {$3$};
\draw (377,88.4) node [anchor=north west][inner sep=0.75pt]  [color={rgb, 255:red, 0; green, 0; blue, 0 }  ,opacity=1 ]  {$4$};
\draw (407,88.4) node [anchor=north west][inner sep=0.75pt]  [color={rgb, 255:red, 0; green, 0; blue, 0 }  ,opacity=1 ]  {$5$};
\draw (331,32.4) node [anchor=north west][inner sep=0.75pt]  [font=\footnotesize,color={rgb, 255:red, 208; green, 2; blue, 27 }  ,opacity=0.5 ]  {$\rho \tau \rho ^{-1}$};
\draw (391,32.4) node [anchor=north west][inner sep=0.75pt]  [font=\footnotesize,color={rgb, 255:red, 208; green, 2; blue, 27 }  ,opacity=0.5 ]  {$\rho ^{2} \tau \rho ^{-2}$};
\draw (230,32.4) node [anchor=north west][inner sep=0.75pt]  [color={rgb, 255:red, 0; green, 0; blue, 0 }  ,opacity=1 ]  {$\mathcal{I}_{1}{}$};
\draw (230,60.4) node [anchor=north west][inner sep=0.75pt]  [color={rgb, 255:red, 0; green, 0; blue, 0 }  ,opacity=1 ]  {$\mathcal{I}_{2}{}$};
\draw (297,3.25) node [anchor=north west][inner sep=0.75pt]  [color={rgb, 255:red, 74; green, 144; blue, 226 }  ,opacity=1 ]  {$\tau '$};
\draw (344.67,1.4) node [anchor=north west][inner sep=0.75pt]  [color={rgb, 255:red, 74; green, 144; blue, 226 }  ,opacity=0.5 ]  {$\rho \tau '\rho ^{-1}$};

\end{tikzpicture}

}
\newcommand{\figureexamplefour}{

\tikzset{every picture/.style={line width=0.75pt}} 

\begin{tikzpicture}[x=0.7pt,y=0.7pt,yscale=-1,xscale=1]

\draw  [color={rgb, 255:red, 0; green, 0; blue, 0 }  ,draw opacity=1 ][fill={rgb, 255:red, 155; green, 155; blue, 155 }  ,fill opacity=0.4 ] (189.33,40.33) .. controls (189.33,34.81) and (193.81,30.33) .. (199.33,30.33) .. controls (204.86,30.33) and (209.33,34.81) .. (209.33,40.33) .. controls (209.33,45.86) and (204.86,50.33) .. (199.33,50.33) .. controls (193.81,50.33) and (189.33,45.86) .. (189.33,40.33) -- cycle ;
\draw  [color={rgb, 255:red, 0; green, 0; blue, 0 }  ,draw opacity=1 ][fill={rgb, 255:red, 155; green, 155; blue, 155 }  ,fill opacity=0.4 ] (219.33,70.33) .. controls (219.33,64.81) and (223.81,60.33) .. (229.33,60.33) .. controls (234.86,60.33) and (239.33,64.81) .. (239.33,70.33) .. controls (239.33,75.86) and (234.86,80.33) .. (229.33,80.33) .. controls (223.81,80.33) and (219.33,75.86) .. (219.33,70.33) -- cycle ;
\draw [color={rgb, 255:red, 0; green, 0; blue, 0 }  ,draw opacity=1 ]   (197.33,167) -- (319.67,167.24) ;
\draw [shift={(321.67,167.24)}, rotate = 180.11] [color={rgb, 255:red, 0; green, 0; blue, 0 }  ,draw opacity=1 ][line width=0.75]    (7.65,-2.3) .. controls (4.86,-0.97) and (2.31,-0.21) .. (0,0) .. controls (2.31,0.21) and (4.86,0.98) .. (7.65,2.3)   ;
\draw [color={rgb, 255:red, 208; green, 2; blue, 27 }  ,draw opacity=1 ]   (229.33,70.33) .. controls (210.62,70.33) and (200.47,61.94) .. (199.42,43.07) ;
\draw [shift={(199.33,40.33)}, rotate = 89.45] [fill={rgb, 255:red, 208; green, 2; blue, 27 }  ,fill opacity=1 ][line width=0.08]  [draw opacity=0] (5.36,-2.57) -- (0,0) -- (5.36,2.57) -- cycle    ;
\draw [color={rgb, 255:red, 208; green, 2; blue, 27 }  ,draw opacity=1 ]   (199.33,40.33) .. controls (218.33,40.52) and (228.31,49.38) .. (229.26,67.41) ;
\draw [shift={(229.33,70.33)}, rotate = 270] [fill={rgb, 255:red, 208; green, 2; blue, 27 }  ,fill opacity=1 ][line width=0.08]  [draw opacity=0] (5.36,-2.57) -- (0,0) -- (5.36,2.57) -- cycle    ;
\draw  [color={rgb, 255:red, 0; green, 0; blue, 0 }  ,draw opacity=1 ][fill={rgb, 255:red, 155; green, 155; blue, 155 }  ,fill opacity=0.4 ] (309.67,40.33) .. controls (309.67,34.81) and (314.14,30.33) .. (319.67,30.33) .. controls (325.19,30.33) and (329.67,34.81) .. (329.67,40.33) .. controls (329.67,45.86) and (325.19,50.33) .. (319.67,50.33) .. controls (314.14,50.33) and (309.67,45.86) .. (309.67,40.33) -- cycle ;
\draw  [color={rgb, 255:red, 0; green, 0; blue, 0 }  ,draw opacity=1 ][fill={rgb, 255:red, 155; green, 155; blue, 155 }  ,fill opacity=0.4 ] (339.67,70.33) .. controls (339.67,64.81) and (344.14,60.33) .. (349.67,60.33) .. controls (355.19,60.33) and (359.67,64.81) .. (359.67,70.33) .. controls (359.67,75.86) and (355.19,80.33) .. (349.67,80.33) .. controls (344.14,80.33) and (339.67,75.86) .. (339.67,70.33) -- cycle ;
\draw [color={rgb, 255:red, 208; green, 2; blue, 27 }  ,draw opacity=0.5 ]   (349.67,70.33) .. controls (330.95,70.33) and (320.8,61.94) .. (319.75,43.07) ;
\draw [shift={(319.67,40.33)}, rotate = 89.45] [fill={rgb, 255:red, 208; green, 2; blue, 27 }  ,fill opacity=0.5 ][line width=0.08]  [draw opacity=0] (5.36,-2.57) -- (0,0) -- (5.36,2.57) -- cycle    ;
\draw [color={rgb, 255:red, 208; green, 2; blue, 27 }  ,draw opacity=0.5 ]   (319.67,40.33) .. controls (338.67,40.52) and (348.64,49.38) .. (349.59,67.41) ;
\draw [shift={(349.67,70.33)}, rotate = 270] [fill={rgb, 255:red, 208; green, 2; blue, 27 }  ,fill opacity=0.5 ][line width=0.08]  [draw opacity=0] (5.36,-2.57) -- (0,0) -- (5.36,2.57) -- cycle    ;
\draw [color={rgb, 255:red, 74; green, 144; blue, 226 }  ,draw opacity=1 ]   (301,110.41) .. controls (317.32,98.85) and (347.16,98.42) .. (365.74,109.32) ;
\draw [shift={(368,110.74)}, rotate = 213.93] [fill={rgb, 255:red, 74; green, 144; blue, 226 }  ,fill opacity=1 ][line width=0.08]  [draw opacity=0] (5.36,-2.57) -- (0,0) -- (5.36,2.57) -- cycle    ;
\draw [color={rgb, 255:red, 74; green, 144; blue, 226 }  ,draw opacity=1 ]   (303.04,121.66) .. controls (322.07,132.14) and (350.83,132.5) .. (368.33,120.41) ;
\draw [shift={(300.42,120.11)}, rotate = 32.31] [fill={rgb, 255:red, 74; green, 144; blue, 226 }  ,fill opacity=1 ][line width=0.08]  [draw opacity=0] (5.36,-2.57) -- (0,0) -- (5.36,2.57) -- cycle    ;
\draw  [color={rgb, 255:red, 0; green, 0; blue, 0 }  ,draw opacity=1 ][fill={rgb, 255:red, 155; green, 155; blue, 155 }  ,fill opacity=0.4 ] (249.67,100.67) .. controls (249.67,95.14) and (254.14,90.67) .. (259.67,90.67) .. controls (265.19,90.67) and (269.67,95.14) .. (269.67,100.67) .. controls (269.67,106.19) and (265.19,110.67) .. (259.67,110.67) .. controls (254.14,110.67) and (249.67,106.19) .. (249.67,100.67) -- cycle ;
\draw  [color={rgb, 255:red, 0; green, 0; blue, 0 }  ,draw opacity=1 ][fill={rgb, 255:red, 155; green, 155; blue, 155 }  ,fill opacity=0.4 ] (279.67,130.67) .. controls (279.67,125.14) and (284.14,120.67) .. (289.67,120.67) .. controls (295.19,120.67) and (299.67,125.14) .. (299.67,130.67) .. controls (299.67,136.19) and (295.19,140.67) .. (289.67,140.67) .. controls (284.14,140.67) and (279.67,136.19) .. (279.67,130.67) -- cycle ;
\draw [color={rgb, 255:red, 208; green, 2; blue, 27 }  ,draw opacity=1 ]   (289.67,130.67) .. controls (270.95,130.67) and (260.8,122.28) .. (259.75,103.4) ;
\draw [shift={(259.67,100.67)}, rotate = 89.45] [fill={rgb, 255:red, 208; green, 2; blue, 27 }  ,fill opacity=1 ][line width=0.08]  [draw opacity=0] (5.36,-2.57) -- (0,0) -- (5.36,2.57) -- cycle    ;
\draw [color={rgb, 255:red, 208; green, 2; blue, 27 }  ,draw opacity=1 ]   (259.67,100.67) .. controls (278.67,100.86) and (288.64,109.71) .. (289.59,127.74) ;
\draw [shift={(289.67,130.67)}, rotate = 270] [fill={rgb, 255:red, 208; green, 2; blue, 27 }  ,fill opacity=1 ][line width=0.08]  [draw opacity=0] (5.36,-2.57) -- (0,0) -- (5.36,2.57) -- cycle    ;
\draw  [color={rgb, 255:red, 74; green, 144; blue, 226 }  ,draw opacity=1 ] (248.17,89.44) -- (301.17,89.44) -- (301.17,141.94) -- (248.17,141.94) -- cycle ;
\draw  [color={rgb, 255:red, 0; green, 0; blue, 0 }  ,draw opacity=1 ][fill={rgb, 255:red, 155; green, 155; blue, 155 }  ,fill opacity=0.4 ] (369.33,100.33) .. controls (369.33,94.81) and (373.81,90.33) .. (379.33,90.33) .. controls (384.86,90.33) and (389.33,94.81) .. (389.33,100.33) .. controls (389.33,105.86) and (384.86,110.33) .. (379.33,110.33) .. controls (373.81,110.33) and (369.33,105.86) .. (369.33,100.33) -- cycle ;
\draw  [color={rgb, 255:red, 0; green, 0; blue, 0 }  ,draw opacity=1 ][fill={rgb, 255:red, 155; green, 155; blue, 155 }  ,fill opacity=0.4 ] (399.33,130.33) .. controls (399.33,124.81) and (403.81,120.33) .. (409.33,120.33) .. controls (414.86,120.33) and (419.33,124.81) .. (419.33,130.33) .. controls (419.33,135.86) and (414.86,140.33) .. (409.33,140.33) .. controls (403.81,140.33) and (399.33,135.86) .. (399.33,130.33) -- cycle ;
\draw [color={rgb, 255:red, 208; green, 2; blue, 27 }  ,draw opacity=0.5 ]   (409.33,130.33) .. controls (390.62,130.33) and (380.47,121.94) .. (379.42,103.07) ;
\draw [shift={(379.33,100.33)}, rotate = 89.45] [fill={rgb, 255:red, 208; green, 2; blue, 27 }  ,fill opacity=0.5 ][line width=0.08]  [draw opacity=0] (5.36,-2.57) -- (0,0) -- (5.36,2.57) -- cycle    ;
\draw [color={rgb, 255:red, 208; green, 2; blue, 27 }  ,draw opacity=0.5 ]   (379.33,100.33) .. controls (398.33,100.52) and (408.31,109.38) .. (409.26,127.41) ;
\draw [shift={(409.33,130.33)}, rotate = 270] [fill={rgb, 255:red, 208; green, 2; blue, 27 }  ,fill opacity=0.5 ][line width=0.08]  [draw opacity=0] (5.36,-2.57) -- (0,0) -- (5.36,2.57) -- cycle    ;
\draw  [color={rgb, 255:red, 74; green, 144; blue, 226 }  ,draw opacity=1 ] (368.08,89.11) -- (421.08,89.11) -- (421.08,141.61) -- (368.08,141.61) -- cycle ;
\draw [color={rgb, 255:red, 74; green, 144; blue, 226 }  ,draw opacity=0.5 ] [dash pattern={on 3pt off 1.5pt}]  (421.5,109.91) .. controls (437.82,98.35) and (467.66,97.92) .. (486.24,108.82) ;
\draw [shift={(488.5,110.24)}, rotate = 213.93] [fill={rgb, 255:red, 74; green, 144; blue, 226 }  ,fill opacity=0.5 ][line width=0.08]  [draw opacity=0] (5.36,-2.57) -- (0,0) -- (5.36,2.57) -- cycle    ;
\draw [color={rgb, 255:red, 74; green, 144; blue, 226 }  ,draw opacity=0.5 ] [dash pattern={on 3pt off 1.5pt}]  (423.54,121.16) .. controls (442.57,131.64) and (471.33,132) .. (488.83,119.91) ;
\draw [shift={(420.92,119.61)}, rotate = 32.31] [fill={rgb, 255:red, 74; green, 144; blue, 226 }  ,fill opacity=0.5 ][line width=0.08]  [draw opacity=0] (5.36,-2.57) -- (0,0) -- (5.36,2.57) -- cycle    ;
\draw  [color={rgb, 255:red, 0; green, 0; blue, 0 }  ,draw opacity=1 ][fill={rgb, 255:red, 155; green, 155; blue, 155 }  ,fill opacity=0.4 ] (429.77,40.53) .. controls (429.77,35.01) and (434.24,30.53) .. (439.77,30.53) .. controls (445.29,30.53) and (449.77,35.01) .. (449.77,40.53) .. controls (449.77,46.06) and (445.29,50.53) .. (439.77,50.53) .. controls (434.24,50.53) and (429.77,46.06) .. (429.77,40.53) -- cycle ;
\draw  [color={rgb, 255:red, 0; green, 0; blue, 0 }  ,draw opacity=1 ][fill={rgb, 255:red, 155; green, 155; blue, 155 }  ,fill opacity=0.4 ] (459.77,70.53) .. controls (459.77,65.01) and (464.24,60.53) .. (469.77,60.53) .. controls (475.29,60.53) and (479.77,65.01) .. (479.77,70.53) .. controls (479.77,76.06) and (475.29,80.53) .. (469.77,80.53) .. controls (464.24,80.53) and (459.77,76.06) .. (459.77,70.53) -- cycle ;
\draw [color={rgb, 255:red, 208; green, 2; blue, 27 }  ,draw opacity=0.5 ]   (469.77,70.53) .. controls (451.05,70.53) and (440.9,62.14) .. (439.85,43.27) ;
\draw [shift={(439.77,40.53)}, rotate = 89.45] [fill={rgb, 255:red, 208; green, 2; blue, 27 }  ,fill opacity=0.5 ][line width=0.08]  [draw opacity=0] (5.36,-2.57) -- (0,0) -- (5.36,2.57) -- cycle    ;
\draw [color={rgb, 255:red, 208; green, 2; blue, 27 }  ,draw opacity=0.5 ]   (439.77,40.53) .. controls (458.77,40.72) and (468.74,49.58) .. (469.69,67.61) ;
\draw [shift={(469.77,70.53)}, rotate = 270] [fill={rgb, 255:red, 208; green, 2; blue, 27 }  ,fill opacity=0.5 ][line width=0.08]  [draw opacity=0] (5.36,-2.57) -- (0,0) -- (5.36,2.57) -- cycle    ;
\draw  [color={rgb, 255:red, 0; green, 0; blue, 0 }  ,draw opacity=1 ][fill={rgb, 255:red, 155; green, 155; blue, 155 }  ,fill opacity=0.4 ] (489.33,100.73) .. controls (489.33,95.21) and (493.81,90.73) .. (499.33,90.73) .. controls (504.86,90.73) and (509.33,95.21) .. (509.33,100.73) .. controls (509.33,106.26) and (504.86,110.73) .. (499.33,110.73) .. controls (493.81,110.73) and (489.33,106.26) .. (489.33,100.73) -- cycle ;
\draw  [color={rgb, 255:red, 0; green, 0; blue, 0 }  ,draw opacity=1 ][fill={rgb, 255:red, 155; green, 155; blue, 155 }  ,fill opacity=0.4 ] (519.33,130.73) .. controls (519.33,125.21) and (523.81,120.73) .. (529.33,120.73) .. controls (534.86,120.73) and (539.33,125.21) .. (539.33,130.73) .. controls (539.33,136.26) and (534.86,140.73) .. (529.33,140.73) .. controls (523.81,140.73) and (519.33,136.26) .. (519.33,130.73) -- cycle ;
\draw [color={rgb, 255:red, 208; green, 2; blue, 27 }  ,draw opacity=0.5 ]   (529.33,130.73) .. controls (510.62,130.73) and (500.47,122.34) .. (499.42,103.47) ;
\draw [shift={(499.33,100.73)}, rotate = 89.45] [fill={rgb, 255:red, 208; green, 2; blue, 27 }  ,fill opacity=0.5 ][line width=0.08]  [draw opacity=0] (5.36,-2.57) -- (0,0) -- (5.36,2.57) -- cycle    ;
\draw [color={rgb, 255:red, 208; green, 2; blue, 27 }  ,draw opacity=0.5 ]   (499.33,100.73) .. controls (518.33,100.92) and (528.31,109.78) .. (529.26,127.81) ;
\draw [shift={(529.33,130.73)}, rotate = 270] [fill={rgb, 255:red, 208; green, 2; blue, 27 }  ,fill opacity=0.5 ][line width=0.08]  [draw opacity=0] (5.36,-2.57) -- (0,0) -- (5.36,2.57) -- cycle    ;
\draw  [color={rgb, 255:red, 74; green, 144; blue, 226 }  ,draw opacity=1 ] (488.08,89.51) -- (541.08,89.51) -- (541.08,142.01) -- (488.08,142.01) -- cycle ;

\draw (253.33,170) node [anchor=north west][inner sep=0.75pt]  [color={rgb, 255:red, 0; green, 0; blue, 0 }  ,opacity=1 ]  {$\rho $};
\draw (220.33,38) node [anchor=north west][inner sep=0.75pt]  [color={rgb, 255:red, 208; green, 2; blue, 27 }  ,opacity=1 ]  {$\tau $};
\draw (194.33,149.4) node [anchor=north west][inner sep=0.75pt]  [color={rgb, 255:red, 0; green, 0; blue, 0 }  ,opacity=1 ]  {$0$};
\draw (224.33,149.4) node [anchor=north west][inner sep=0.75pt]  [color={rgb, 255:red, 0; green, 0; blue, 0 }  ,opacity=1 ]  {$1$};
\draw (254.33,149.4) node [anchor=north west][inner sep=0.75pt]  [color={rgb, 255:red, 0; green, 0; blue, 0 }  ,opacity=1 ]  {$2$};
\draw (285.33,149.4) node [anchor=north west][inner sep=0.75pt]  [color={rgb, 255:red, 0; green, 0; blue, 0 }  ,opacity=1 ]  {$3$};
\draw (314.33,149.4) node [anchor=north west][inner sep=0.75pt]  [color={rgb, 255:red, 0; green, 0; blue, 0 }  ,opacity=1 ]  {$4$};
\draw (344.33,149.4) node [anchor=north west][inner sep=0.75pt]  [color={rgb, 255:red, 0; green, 0; blue, 0 }  ,opacity=1 ]  {$5$};
\draw (330.67,32.73) node [anchor=north west][inner sep=0.75pt]  [font=\footnotesize,color={rgb, 255:red, 208; green, 2; blue, 27 }  ,opacity=0.5 ]  {$\rho \tau \rho ^{-1}$};
\draw (160.33,32.73) node [anchor=north west][inner sep=0.75pt]  [color={rgb, 255:red, 0; green, 0; blue, 0 }  ,opacity=1 ]  {$\mathcal{I}_{1}{}$};
\draw (160.33,60.73) node [anchor=north west][inner sep=0.75pt]  [color={rgb, 255:red, 0; green, 0; blue, 0 }  ,opacity=1 ]  {$\mathcal{I}_{2}{}$};
\draw (328.5,110.25) node [anchor=north west][inner sep=0.75pt]  [color={rgb, 255:red, 74; green, 144; blue, 226 }  ,opacity=1 ]  {$\tau '$};
\draw (433.33,105.9) node [anchor=north west][inner sep=0.75pt]  [color={rgb, 255:red, 74; green, 144; blue, 226 }  ,opacity=0.5 ]  {$\rho \tau '\rho ^{-1}$};
\draw (374.67,149.73) node [anchor=north west][inner sep=0.75pt]  [color={rgb, 255:red, 0; green, 0; blue, 0 }  ,opacity=1 ]  {$6$};
\draw (404.67,149.4) node [anchor=north west][inner sep=0.75pt]  [color={rgb, 255:red, 0; green, 0; blue, 0 }  ,opacity=1 ]  {$7$};
\draw (435,149.73) node [anchor=north west][inner sep=0.75pt]  [color={rgb, 255:red, 0; green, 0; blue, 0 }  ,opacity=1 ]  {$8$};
\draw (280.67,98) node [anchor=north west][inner sep=0.75pt]  [color={rgb, 255:red, 208; green, 2; blue, 27 }  ,opacity=1 ]  {$\tau $};
\draw (390.33,92.73) node [anchor=north west][inner sep=0.75pt]  [font=\footnotesize,color={rgb, 255:red, 208; green, 2; blue, 27 }  ,opacity=0.5 ]  {$\rho \tau \rho ^{-1}$};
\draw (160.67,92.73) node [anchor=north west][inner sep=0.75pt]  [color={rgb, 255:red, 0; green, 0; blue, 0 }  ,opacity=1 ]  {$\mathcal{I}_{3}{}$};
\draw (160.67,120.73) node [anchor=north west][inner sep=0.75pt]  [color={rgb, 255:red, 0; green, 0; blue, 0 }  ,opacity=1 ]  {$\mathcal{I}_{4}{}$};
\draw (464.83,149.4) node [anchor=north west][inner sep=0.75pt]  [color={rgb, 255:red, 0; green, 0; blue, 0 }  ,opacity=1 ]  {$9$};
\draw (489.67,149.73) node [anchor=north west][inner sep=0.75pt]  [color={rgb, 255:red, 0; green, 0; blue, 0 }  ,opacity=1 ]  {$10$};
\draw (520.67,149.4) node [anchor=north west][inner sep=0.75pt]  [color={rgb, 255:red, 0; green, 0; blue, 0 }  ,opacity=1 ]  {$11$};
\draw (450.77,32.93) node [anchor=north west][inner sep=0.75pt]  [font=\footnotesize,color={rgb, 255:red, 208; green, 2; blue, 27 }  ,opacity=0.5 ]  {$\rho ^{2} \tau \rho ^{-2}$};
\draw (510.33,93.13) node [anchor=north west][inner sep=0.75pt]  [font=\footnotesize,color={rgb, 255:red, 208; green, 2; blue, 27 }  ,opacity=0.5 ]  {$\rho ^{2} \tau \rho ^{-2}$};

\end{tikzpicture}

}
\import{figures/}{orbits.tex}
\newcommand{\figureinfiniteorbits}{

\tikzset{every picture/.style={line width=0.75pt}} 

\begin{tikzpicture}[x=0.75pt,y=0.75pt,yscale=-1,xscale=1]

\draw  [color={rgb, 255:red, 208; green, 2; blue, 27 }  ,draw opacity=1 ][fill={rgb, 255:red, 208; green, 2; blue, 27 }  ,fill opacity=1 ] (190,90) .. controls (190,84.48) and (194.48,80) .. (200,80) .. controls (205.52,80) and (210,84.48) .. (210,90) .. controls (210,95.52) and (205.52,100) .. (200,100) .. controls (194.48,100) and (190,95.52) .. (190,90) -- cycle ;
\draw  [color={rgb, 255:red, 208; green, 2; blue, 27 }  ,draw opacity=1 ][fill={rgb, 255:red, 208; green, 2; blue, 27 }  ,fill opacity=1 ] (250,90) .. controls (250,84.48) and (254.48,80) .. (260,80) .. controls (265.52,80) and (270,84.48) .. (270,90) .. controls (270,95.52) and (265.52,100) .. (260,100) .. controls (254.48,100) and (250,95.52) .. (250,90) -- cycle ;
\draw  [color={rgb, 255:red, 208; green, 2; blue, 27 }  ,draw opacity=1 ][fill={rgb, 255:red, 208; green, 2; blue, 27 }  ,fill opacity=1 ] (310,90) .. controls (310,84.48) and (314.48,80) .. (320,80) .. controls (325.52,80) and (330,84.48) .. (330,90) .. controls (330,95.52) and (325.52,100) .. (320,100) .. controls (314.48,100) and (310,95.52) .. (310,90) -- cycle ;
\draw  [color={rgb, 255:red, 208; green, 2; blue, 27 }  ,draw opacity=1 ][fill={rgb, 255:red, 208; green, 2; blue, 27 }  ,fill opacity=1 ] (190,120) .. controls (190,114.48) and (194.48,110) .. (200,110) .. controls (205.52,110) and (210,114.48) .. (210,120) .. controls (210,125.52) and (205.52,130) .. (200,130) .. controls (194.48,130) and (190,125.52) .. (190,120) -- cycle ;
\draw  [color={rgb, 255:red, 208; green, 2; blue, 27 }  ,draw opacity=1 ][fill={rgb, 255:red, 208; green, 2; blue, 27 }  ,fill opacity=1 ] (250,120) .. controls (250,114.48) and (254.48,110) .. (260,110) .. controls (265.52,110) and (270,114.48) .. (270,120) .. controls (270,125.52) and (265.52,130) .. (260,130) .. controls (254.48,130) and (250,125.52) .. (250,120) -- cycle ;
\draw  [color={rgb, 255:red, 208; green, 2; blue, 27 }  ,draw opacity=1 ][fill={rgb, 255:red, 208; green, 2; blue, 27 }  ,fill opacity=1 ] (310,120) .. controls (310,114.48) and (314.48,110) .. (320,110) .. controls (325.52,110) and (330,114.48) .. (330,120) .. controls (330,125.52) and (325.52,130) .. (320,130) .. controls (314.48,130) and (310,125.52) .. (310,120) -- cycle ;
\draw  [color={rgb, 255:red, 208; green, 2; blue, 27 }  ,draw opacity=1 ][fill={rgb, 255:red, 208; green, 2; blue, 27 }  ,fill opacity=1 ] (190,30) .. controls (190,24.48) and (194.48,20) .. (200,20) .. controls (205.52,20) and (210,24.48) .. (210,30) .. controls (210,35.52) and (205.52,40) .. (200,40) .. controls (194.48,40) and (190,35.52) .. (190,30) -- cycle ;
\draw  [color={rgb, 255:red, 208; green, 2; blue, 27 }  ,draw opacity=1 ][fill={rgb, 255:red, 208; green, 2; blue, 27 }  ,fill opacity=1 ] (250,30) .. controls (250,24.48) and (254.48,20) .. (260,20) .. controls (265.52,20) and (270,24.48) .. (270,30) .. controls (270,35.52) and (265.52,40) .. (260,40) .. controls (254.48,40) and (250,35.52) .. (250,30) -- cycle ;
\draw  [color={rgb, 255:red, 208; green, 2; blue, 27 }  ,draw opacity=1 ][fill={rgb, 255:red, 208; green, 2; blue, 27 }  ,fill opacity=1 ] (310,30) .. controls (310,24.48) and (314.48,20) .. (320,20) .. controls (325.52,20) and (330,24.48) .. (330,30) .. controls (330,35.52) and (325.52,40) .. (320,40) .. controls (314.48,40) and (310,35.52) .. (310,30) -- cycle ;
\draw  [color={rgb, 255:red, 208; green, 2; blue, 27 }  ,draw opacity=1 ][fill={rgb, 255:red, 208; green, 2; blue, 27 }  ,fill opacity=1 ] (190,60) .. controls (190,54.48) and (194.48,50) .. (200,50) .. controls (205.52,50) and (210,54.48) .. (210,60) .. controls (210,65.52) and (205.52,70) .. (200,70) .. controls (194.48,70) and (190,65.52) .. (190,60) -- cycle ;
\draw  [color={rgb, 255:red, 208; green, 2; blue, 27 }  ,draw opacity=1 ][fill={rgb, 255:red, 208; green, 2; blue, 27 }  ,fill opacity=1 ] (250,60) .. controls (250,54.48) and (254.48,50) .. (260,50) .. controls (265.52,50) and (270,54.48) .. (270,60) .. controls (270,65.52) and (265.52,70) .. (260,70) .. controls (254.48,70) and (250,65.52) .. (250,60) -- cycle ;
\draw  [color={rgb, 255:red, 208; green, 2; blue, 27 }  ,draw opacity=1 ][fill={rgb, 255:red, 208; green, 2; blue, 27 }  ,fill opacity=1 ] (310,60) .. controls (310,54.48) and (314.48,50) .. (320,50) .. controls (325.52,50) and (330,54.48) .. (330,60) .. controls (330,65.52) and (325.52,70) .. (320,70) .. controls (314.48,70) and (310,65.52) .. (310,60) -- cycle ;
\draw  [color={rgb, 255:red, 208; green, 2; blue, 27 }  ,draw opacity=1 ][fill={rgb, 255:red, 208; green, 2; blue, 27 }  ,fill opacity=1 ] (490,90) .. controls (490,84.48) and (494.48,80) .. (500,80) .. controls (505.52,80) and (510,84.48) .. (510,90) .. controls (510,95.52) and (505.52,100) .. (500,100) .. controls (494.48,100) and (490,95.52) .. (490,90) -- cycle ;
\draw  [color={rgb, 255:red, 208; green, 2; blue, 27 }  ,draw opacity=1 ][fill={rgb, 255:red, 208; green, 2; blue, 27 }  ,fill opacity=1 ] (490,120) .. controls (490,114.48) and (494.48,110) .. (500,110) .. controls (505.52,110) and (510,114.48) .. (510,120) .. controls (510,125.52) and (505.52,130) .. (500,130) .. controls (494.48,130) and (490,125.52) .. (490,120) -- cycle ;
\draw  [color={rgb, 255:red, 208; green, 2; blue, 27 }  ,draw opacity=1 ][fill={rgb, 255:red, 208; green, 2; blue, 27 }  ,fill opacity=1 ] (490,30) .. controls (490,24.48) and (494.48,20) .. (500,20) .. controls (505.52,20) and (510,24.48) .. (510,30) .. controls (510,35.52) and (505.52,40) .. (500,40) .. controls (494.48,40) and (490,35.52) .. (490,30) -- cycle ;
\draw  [color={rgb, 255:red, 208; green, 2; blue, 27 }  ,draw opacity=1 ][fill={rgb, 255:red, 208; green, 2; blue, 27 }  ,fill opacity=1 ] (490,60) .. controls (490,54.48) and (494.48,50) .. (500,50) .. controls (505.52,50) and (510,54.48) .. (510,60) .. controls (510,65.52) and (505.52,70) .. (500,70) .. controls (494.48,70) and (490,65.52) .. (490,60) -- cycle ;
\draw [color={rgb, 255:red, 144; green, 19; blue, 254 }  ,draw opacity=1 ]   (320,30) .. controls (310.47,39.01) and (310.19,48.53) .. (318.08,57.89) ;
\draw [shift={(320,60)}, rotate = 225.67] [fill={rgb, 255:red, 144; green, 19; blue, 254 }  ,fill opacity=1 ][line width=0.08]  [draw opacity=0] (5.36,-2.57) -- (0,0) -- (5.36,2.57) -- cycle    ;
\draw [color={rgb, 255:red, 144; green, 19; blue, 254 }  ,draw opacity=1 ]   (320,60) .. controls (329.49,50.76) and (330.15,41.47) .. (321.99,32.11) ;
\draw [shift={(320,30)}, rotate = 44.62] [fill={rgb, 255:red, 144; green, 19; blue, 254 }  ,fill opacity=1 ][line width=0.08]  [draw opacity=0] (5.36,-2.57) -- (0,0) -- (5.36,2.57) -- cycle    ;
\draw  [color={rgb, 255:red, 128; green, 128; blue, 128 }  ,draw opacity=1 ] (188,18) -- (211.75,18) -- (211.75,72.03) -- (188,72.03) -- cycle ;
\draw  [color={rgb, 255:red, 128; green, 128; blue, 128 }  ,draw opacity=1 ] (188,78.75) -- (211.75,78.75) -- (211.75,132.78) -- (188,132.78) -- cycle ;
\draw  [color={rgb, 255:red, 128; green, 128; blue, 128 }  ,draw opacity=1 ] (248,18.08) -- (271.75,18.08) -- (271.75,72.1) -- (248,72.1) -- cycle ;
\draw  [color={rgb, 255:red, 128; green, 128; blue, 128 }  ,draw opacity=1 ] (248,78.83) -- (271.75,78.83) -- (271.75,132.85) -- (248,132.85) -- cycle ;
\draw  [color={rgb, 255:red, 208; green, 2; blue, 27 }  ,draw opacity=1 ][fill={rgb, 255:red, 208; green, 2; blue, 27 }  ,fill opacity=1 ] (370,90) .. controls (370,84.48) and (374.48,80) .. (380,80) .. controls (385.52,80) and (390,84.48) .. (390,90) .. controls (390,95.52) and (385.52,100) .. (380,100) .. controls (374.48,100) and (370,95.52) .. (370,90) -- cycle ;
\draw  [color={rgb, 255:red, 208; green, 2; blue, 27 }  ,draw opacity=1 ][fill={rgb, 255:red, 208; green, 2; blue, 27 }  ,fill opacity=1 ] (430,90) .. controls (430,84.48) and (434.48,80) .. (440,80) .. controls (445.52,80) and (450,84.48) .. (450,90) .. controls (450,95.52) and (445.52,100) .. (440,100) .. controls (434.48,100) and (430,95.52) .. (430,90) -- cycle ;
\draw  [color={rgb, 255:red, 208; green, 2; blue, 27 }  ,draw opacity=1 ][fill={rgb, 255:red, 208; green, 2; blue, 27 }  ,fill opacity=1 ] (370,120) .. controls (370,114.48) and (374.48,110) .. (380,110) .. controls (385.52,110) and (390,114.48) .. (390,120) .. controls (390,125.52) and (385.52,130) .. (380,130) .. controls (374.48,130) and (370,125.52) .. (370,120) -- cycle ;
\draw  [color={rgb, 255:red, 208; green, 2; blue, 27 }  ,draw opacity=1 ][fill={rgb, 255:red, 208; green, 2; blue, 27 }  ,fill opacity=1 ] (430,120) .. controls (430,114.48) and (434.48,110) .. (440,110) .. controls (445.52,110) and (450,114.48) .. (450,120) .. controls (450,125.52) and (445.52,130) .. (440,130) .. controls (434.48,130) and (430,125.52) .. (430,120) -- cycle ;
\draw  [color={rgb, 255:red, 208; green, 2; blue, 27 }  ,draw opacity=1 ][fill={rgb, 255:red, 208; green, 2; blue, 27 }  ,fill opacity=1 ] (370,30) .. controls (370,24.48) and (374.48,20) .. (380,20) .. controls (385.52,20) and (390,24.48) .. (390,30) .. controls (390,35.52) and (385.52,40) .. (380,40) .. controls (374.48,40) and (370,35.52) .. (370,30) -- cycle ;
\draw  [color={rgb, 255:red, 208; green, 2; blue, 27 }  ,draw opacity=1 ][fill={rgb, 255:red, 208; green, 2; blue, 27 }  ,fill opacity=1 ] (430,30) .. controls (430,24.48) and (434.48,20) .. (440,20) .. controls (445.52,20) and (450,24.48) .. (450,30) .. controls (450,35.52) and (445.52,40) .. (440,40) .. controls (434.48,40) and (430,35.52) .. (430,30) -- cycle ;
\draw  [color={rgb, 255:red, 208; green, 2; blue, 27 }  ,draw opacity=1 ][fill={rgb, 255:red, 208; green, 2; blue, 27 }  ,fill opacity=1 ] (370,60) .. controls (370,54.48) and (374.48,50) .. (380,50) .. controls (385.52,50) and (390,54.48) .. (390,60) .. controls (390,65.52) and (385.52,70) .. (380,70) .. controls (374.48,70) and (370,65.52) .. (370,60) -- cycle ;
\draw  [color={rgb, 255:red, 208; green, 2; blue, 27 }  ,draw opacity=1 ][fill={rgb, 255:red, 208; green, 2; blue, 27 }  ,fill opacity=1 ] (430,60) .. controls (430,54.48) and (434.48,50) .. (440,50) .. controls (445.52,50) and (450,54.48) .. (450,60) .. controls (450,65.52) and (445.52,70) .. (440,70) .. controls (434.48,70) and (430,65.52) .. (430,60) -- cycle ;
\draw    (381,31.1) .. controls (371.45,50.14) and (371.25,69.24) .. (379.74,88.39) ;
\draw [shift={(381,91.1)}, rotate = 244.08] [fill={rgb, 255:red, 0; green, 0; blue, 0 }  ][line width=0.08]  [draw opacity=0] (5.36,-2.57) -- (0,0) -- (5.36,2.57) -- cycle    ;
\draw    (380,90) .. controls (399.2,80.34) and (427.62,52.31) .. (438.71,32.44) ;
\draw [shift={(440,30)}, rotate = 116.63] [fill={rgb, 255:red, 0; green, 0; blue, 0 }  ][line width=0.08]  [draw opacity=0] (5.36,-2.57) -- (0,0) -- (5.36,2.57) -- cycle    ;
\draw    (381,61.1) .. controls (371.45,80.14) and (371.25,99.24) .. (379.74,118.39) ;
\draw [shift={(381,121.1)}, rotate = 244.08] [fill={rgb, 255:red, 0; green, 0; blue, 0 }  ][line width=0.08]  [draw opacity=0] (5.36,-2.57) -- (0,0) -- (5.36,2.57) -- cycle    ;
\draw    (440,60) .. controls (421.04,45.54) and (401.62,44.5) .. (382.4,58.21) ;
\draw [shift={(380,60)}, rotate = 322.12] [fill={rgb, 255:red, 0; green, 0; blue, 0 }  ][line width=0.08]  [draw opacity=0] (5.36,-2.57) -- (0,0) -- (5.36,2.57) -- cycle    ;
\draw    (440,30) .. controls (421.04,15.54) and (401.62,14.5) .. (382.4,28.21) ;
\draw [shift={(380,30)}, rotate = 322.12] [fill={rgb, 255:red, 0; green, 0; blue, 0 }  ][line width=0.08]  [draw opacity=0] (5.36,-2.57) -- (0,0) -- (5.36,2.57) -- cycle    ;
\draw    (380,120) .. controls (399.2,110.34) and (427.62,82.31) .. (438.71,62.44) ;
\draw [shift={(440,60)}, rotate = 116.63] [fill={rgb, 255:red, 0; green, 0; blue, 0 }  ][line width=0.08]  [draw opacity=0] (5.36,-2.57) -- (0,0) -- (5.36,2.57) -- cycle    ;
\draw  [color={rgb, 255:red, 128; green, 128; blue, 128 }  ,draw opacity=1 ] (308.25,18.22) -- (332,18.22) -- (332,72.25) -- (308.25,72.25) -- cycle ;
\draw  [color={rgb, 255:red, 128; green, 128; blue, 128 }  ,draw opacity=1 ] (308.25,78.97) -- (332,78.97) -- (332,133) -- (308.25,133) -- cycle ;
\draw  [color={rgb, 255:red, 128; green, 128; blue, 128 }  ,draw opacity=1 ] (368.25,17.58) -- (392,17.58) -- (392,71.6) -- (368.25,71.6) -- cycle ;
\draw  [color={rgb, 255:red, 128; green, 128; blue, 128 }  ,draw opacity=1 ] (368.25,78.33) -- (392,78.33) -- (392,132.35) -- (368.25,132.35) -- cycle ;
\draw  [color={rgb, 255:red, 128; green, 128; blue, 128 }  ,draw opacity=1 ] (428.25,17.58) -- (452,17.58) -- (452,71.6) -- (428.25,71.6) -- cycle ;
\draw  [color={rgb, 255:red, 128; green, 128; blue, 128 }  ,draw opacity=1 ] (428.25,78.33) -- (452,78.33) -- (452,132.35) -- (428.25,132.35) -- cycle ;
\draw  [color={rgb, 255:red, 128; green, 128; blue, 128 }  ,draw opacity=1 ] (488,17.58) -- (511.75,17.58) -- (511.75,71.6) -- (488,71.6) -- cycle ;
\draw  [color={rgb, 255:red, 128; green, 128; blue, 128 }  ,draw opacity=1 ] (488,78.33) -- (511.75,78.33) -- (511.75,132.35) -- (488,132.35) -- cycle ;
\draw    (200,150) -- (258,150) ;
\draw [shift={(260,150)}, rotate = 180] [color={rgb, 255:red, 0; green, 0; blue, 0 }  ][line width=0.75]    (7.65,-2.3) .. controls (4.86,-0.97) and (2.31,-0.21) .. (0,0) .. controls (2.31,0.21) and (4.86,0.98) .. (7.65,2.3)   ;

\draw (274,33.4) node [anchor=north west][inner sep=0.75pt]  [font=\small,color={rgb, 255:red, 144; green, 19; blue, 254 }  ,opacity=1 ]  {$\rho \tau_{2} \rho^{-1}$};
\draw (403,7) node [anchor=north west][inner sep=0.75pt]  [font=\small]  {$\tau_{1}$};
\draw (222,133) node [anchor=north west][inner sep=0.75pt]  [font=\small]  {$\rho^{2}$};
\draw (173,37) node [anchor=north west][inner sep=0.75pt]  [color={rgb, 255:red, 128; green, 128; blue, 128 }  ,opacity=1 ]  {$\Gamma$};

\end{tikzpicture}
}
\newcommand{\figuretwist}{

\tikzset{every picture/.style={line width=0.75pt}} 

\begin{tikzpicture}[x=0.75pt,y=0.75pt,yscale=-1,xscale=1]

\draw  [color={rgb, 255:red, 208; green, 2; blue, 27 }  ,draw opacity=1 ][fill={rgb, 255:red, 208; green, 2; blue, 27 }  ,fill opacity=1 ] (194.67,31.91) .. controls (194.67,26.39) and (199.14,21.91) .. (204.67,21.91) .. controls (210.19,21.91) and (214.67,26.39) .. (214.67,31.91) .. controls (214.67,37.43) and (210.19,41.91) .. (204.67,41.91) .. controls (199.14,41.91) and (194.67,37.43) .. (194.67,31.91) -- cycle ;
\draw  [color={rgb, 255:red, 208; green, 2; blue, 27 }  ,draw opacity=1 ][fill={rgb, 255:red, 208; green, 2; blue, 27 }  ,fill opacity=0 ] (234.42,31.76) .. controls (234.42,26.24) and (238.89,21.76) .. (244.42,21.76) .. controls (249.94,21.76) and (254.42,26.24) .. (254.42,31.76) .. controls (254.42,37.28) and (249.94,41.76) .. (244.42,41.76) .. controls (238.89,41.76) and (234.42,37.28) .. (234.42,31.76) -- cycle ;
\draw  [color={rgb, 255:red, 208; green, 2; blue, 27 }  ,draw opacity=1 ][fill={rgb, 255:red, 208; green, 2; blue, 27 }  ,fill opacity=0 ] (194.67,61.91) .. controls (194.67,56.39) and (199.14,51.91) .. (204.67,51.91) .. controls (210.19,51.91) and (214.67,56.39) .. (214.67,61.91) .. controls (214.67,67.43) and (210.19,71.91) .. (204.67,71.91) .. controls (199.14,71.91) and (194.67,67.43) .. (194.67,61.91) -- cycle ;
\draw  [color={rgb, 255:red, 208; green, 2; blue, 27 }  ,draw opacity=1 ][fill={rgb, 255:red, 208; green, 2; blue, 27 }  ,fill opacity=1 ] (234.42,61.76) .. controls (234.42,56.24) and (238.89,51.76) .. (244.42,51.76) .. controls (249.94,51.76) and (254.42,56.24) .. (254.42,61.76) .. controls (254.42,67.28) and (249.94,71.76) .. (244.42,71.76) .. controls (238.89,71.76) and (234.42,67.28) .. (234.42,61.76) -- cycle ;
\draw  [color={rgb, 255:red, 128; green, 128; blue, 128 }  ,draw opacity=1 ] (192.67,19.98) -- (216.42,19.98) -- (216.42,74.01) -- (192.67,74.01) -- cycle ;
\draw  [color={rgb, 255:red, 208; green, 2; blue, 27 }  ,draw opacity=1 ][fill={rgb, 255:red, 208; green, 2; blue, 27 }  ,fill opacity=1 ] (274.42,31.98) .. controls (274.42,26.46) and (278.89,21.98) .. (284.42,21.98) .. controls (289.94,21.98) and (294.42,26.46) .. (294.42,31.98) .. controls (294.42,37.51) and (289.94,41.98) .. (284.42,41.98) .. controls (278.89,41.98) and (274.42,37.51) .. (274.42,31.98) -- cycle ;
\draw  [color={rgb, 255:red, 208; green, 2; blue, 27 }  ,draw opacity=1 ] (154.42,32.42) .. controls (154.42,26.9) and (158.89,22.42) .. (164.42,22.42) .. controls (169.94,22.42) and (174.42,26.9) .. (174.42,32.42) .. controls (174.42,37.95) and (169.94,42.42) .. (164.42,42.42) .. controls (158.89,42.42) and (154.42,37.95) .. (154.42,32.42) -- cycle ;
\draw  [color={rgb, 255:red, 208; green, 2; blue, 27 }  ,draw opacity=1 ] (274.42,61.98) .. controls (274.42,56.46) and (278.89,51.98) .. (284.42,51.98) .. controls (289.94,51.98) and (294.42,56.46) .. (294.42,61.98) .. controls (294.42,67.51) and (289.94,71.98) .. (284.42,71.98) .. controls (278.89,71.98) and (274.42,67.51) .. (274.42,61.98) -- cycle ;
\draw  [color={rgb, 255:red, 208; green, 2; blue, 27 }  ,draw opacity=1 ][fill={rgb, 255:red, 208; green, 2; blue, 27 }  ,fill opacity=1 ] (154.42,62.42) .. controls (154.42,56.9) and (158.89,52.42) .. (164.42,52.42) .. controls (169.94,52.42) and (174.42,56.9) .. (174.42,62.42) .. controls (174.42,67.95) and (169.94,72.42) .. (164.42,72.42) .. controls (158.89,72.42) and (154.42,67.95) .. (154.42,62.42) -- cycle ;
\draw    (284.42,31.98) .. controls (264.29,10.61) and (225.3,9.59) .. (206.36,29.96) ;
\draw [shift={(204.67,31.91)}, rotate = 309.15] [fill={rgb, 255:red, 0; green, 0; blue, 0 }  ][line width=0.08]  [draw opacity=0] (5.36,-2.57) -- (0,0) -- (5.36,2.57) -- cycle    ;
\draw  [color={rgb, 255:red, 128; green, 128; blue, 128 }  ,draw opacity=1 ] (232.67,19.98) -- (256.42,19.98) -- (256.42,74.01) -- (232.67,74.01) -- cycle ;
\draw  [color={rgb, 255:red, 128; green, 128; blue, 128 }  ,draw opacity=1 ] (272.67,19.56) -- (296.42,19.56) -- (296.42,73.58) -- (272.67,73.58) -- cycle ;
\draw  [color={rgb, 255:red, 128; green, 128; blue, 128 }  ,draw opacity=1 ] (152.67,20) -- (176.42,20) -- (176.42,74.03) -- (152.67,74.03) -- cycle ;
\draw    (284.42,61.98) .. controls (265.55,82.41) and (226.15,83.5) .. (206.7,64.1) ;
\draw [shift={(204.67,61.91)}, rotate = 49.32] [fill={rgb, 255:red, 0; green, 0; blue, 0 }  ][line width=0.08]  [draw opacity=0] (5.36,-2.57) -- (0,0) -- (5.36,2.57) -- cycle    ;
\draw    (244.42,31.76) -- (282.02,60.18) ;
\draw [shift={(284.42,61.98)}, rotate = 217.08] [fill={rgb, 255:red, 0; green, 0; blue, 0 }  ][line width=0.08]  [draw opacity=0] (5.36,-2.57) -- (0,0) -- (5.36,2.57) -- cycle    ;
\draw    (204.67,31.91) -- (242.27,60.33) ;
\draw [shift={(244.67,62.13)}, rotate = 217.08] [fill={rgb, 255:red, 0; green, 0; blue, 0 }  ][line width=0.08]  [draw opacity=0] (5.36,-2.57) -- (0,0) -- (5.36,2.57) -- cycle    ;
\draw    (204.67,61.91) -- (242.03,33.57) ;
\draw [shift={(244.42,31.76)}, rotate = 142.82] [fill={rgb, 255:red, 0; green, 0; blue, 0 }  ][line width=0.08]  [draw opacity=0] (5.36,-2.57) -- (0,0) -- (5.36,2.57) -- cycle    ;
\draw    (244.67,62.13) -- (282.03,33.8) ;
\draw [shift={(284.42,31.98)}, rotate = 142.82] [fill={rgb, 255:red, 0; green, 0; blue, 0 }  ][line width=0.08]  [draw opacity=0] (5.36,-2.57) -- (0,0) -- (5.36,2.57) -- cycle    ;
\draw  [dash pattern={on 0.84pt off 2.51pt}]  (310.92,46.68) -- (382.92,47.06) ;
\draw [shift={(384.92,47.07)}, rotate = 180.3] [color={rgb, 255:red, 0; green, 0; blue, 0 }  ][line width=0.75]    (7.65,-2.3) .. controls (4.86,-0.97) and (2.31,-0.21) .. (0,0) .. controls (2.31,0.21) and (4.86,0.98) .. (7.65,2.3)   ;
\draw  [color={rgb, 255:red, 208; green, 2; blue, 27 }  ,draw opacity=1 ][fill={rgb, 255:red, 208; green, 2; blue, 27 }  ,fill opacity=1 ] (439.33,31.91) .. controls (439.33,26.39) and (443.81,21.91) .. (449.33,21.91) .. controls (454.86,21.91) and (459.33,26.39) .. (459.33,31.91) .. controls (459.33,37.43) and (454.86,41.91) .. (449.33,41.91) .. controls (443.81,41.91) and (439.33,37.43) .. (439.33,31.91) -- cycle ;
\draw  [color={rgb, 255:red, 208; green, 2; blue, 27 }  ,draw opacity=1 ][fill={rgb, 255:red, 208; green, 2; blue, 27 }  ,fill opacity=0 ] (479.33,62.13) .. controls (479.33,56.61) and (483.81,52.13) .. (489.33,52.13) .. controls (494.86,52.13) and (499.33,56.61) .. (499.33,62.13) .. controls (499.33,67.66) and (494.86,72.13) .. (489.33,72.13) .. controls (483.81,72.13) and (479.33,67.66) .. (479.33,62.13) -- cycle ;
\draw  [color={rgb, 255:red, 208; green, 2; blue, 27 }  ,draw opacity=1 ][fill={rgb, 255:red, 208; green, 2; blue, 27 }  ,fill opacity=0 ] (439.33,61.91) .. controls (439.33,56.39) and (443.81,51.91) .. (449.33,51.91) .. controls (454.86,51.91) and (459.33,56.39) .. (459.33,61.91) .. controls (459.33,67.43) and (454.86,71.91) .. (449.33,71.91) .. controls (443.81,71.91) and (439.33,67.43) .. (439.33,61.91) -- cycle ;
\draw  [color={rgb, 255:red, 208; green, 2; blue, 27 }  ,draw opacity=1 ][fill={rgb, 255:red, 208; green, 2; blue, 27 }  ,fill opacity=1 ] (479.08,31.76) .. controls (479.08,26.24) and (483.56,21.76) .. (489.08,21.76) .. controls (494.61,21.76) and (499.08,26.24) .. (499.08,31.76) .. controls (499.08,37.28) and (494.61,41.76) .. (489.08,41.76) .. controls (483.56,41.76) and (479.08,37.28) .. (479.08,31.76) -- cycle ;
\draw  [color={rgb, 255:red, 128; green, 128; blue, 128 }  ,draw opacity=1 ] (437.33,19.98) -- (461.08,19.98) -- (461.08,74.01) -- (437.33,74.01) -- cycle ;
\draw  [color={rgb, 255:red, 208; green, 2; blue, 27 }  ,draw opacity=1 ][fill={rgb, 255:red, 208; green, 2; blue, 27 }  ,fill opacity=1 ] (519.08,31.98) .. controls (519.08,26.46) and (523.56,21.98) .. (529.08,21.98) .. controls (534.61,21.98) and (539.08,26.46) .. (539.08,31.98) .. controls (539.08,37.51) and (534.61,41.98) .. (529.08,41.98) .. controls (523.56,41.98) and (519.08,37.51) .. (519.08,31.98) -- cycle ;
\draw  [color={rgb, 255:red, 208; green, 2; blue, 27 }  ,draw opacity=1 ] (399.08,62.42) .. controls (399.08,56.9) and (403.56,52.42) .. (409.08,52.42) .. controls (414.61,52.42) and (419.08,56.9) .. (419.08,62.42) .. controls (419.08,67.95) and (414.61,72.42) .. (409.08,72.42) .. controls (403.56,72.42) and (399.08,67.95) .. (399.08,62.42) -- cycle ;
\draw  [color={rgb, 255:red, 208; green, 2; blue, 27 }  ,draw opacity=1 ] (519.08,61.98) .. controls (519.08,56.46) and (523.56,51.98) .. (529.08,51.98) .. controls (534.61,51.98) and (539.08,56.46) .. (539.08,61.98) .. controls (539.08,67.51) and (534.61,71.98) .. (529.08,71.98) .. controls (523.56,71.98) and (519.08,67.51) .. (519.08,61.98) -- cycle ;
\draw  [color={rgb, 255:red, 208; green, 2; blue, 27 }  ,draw opacity=1 ][fill={rgb, 255:red, 208; green, 2; blue, 27 }  ,fill opacity=1 ] (399.08,32.42) .. controls (399.08,26.9) and (403.56,22.42) .. (409.08,22.42) .. controls (414.61,22.42) and (419.08,26.9) .. (419.08,32.42) .. controls (419.08,37.95) and (414.61,42.42) .. (409.08,42.42) .. controls (403.56,42.42) and (399.08,37.95) .. (399.08,32.42) -- cycle ;
\draw    (529.08,31.98) .. controls (508.96,10.61) and (469.96,9.59) .. (451.03,29.96) ;
\draw [shift={(449.33,31.91)}, rotate = 309.15] [fill={rgb, 255:red, 0; green, 0; blue, 0 }  ][line width=0.08]  [draw opacity=0] (5.36,-2.57) -- (0,0) -- (5.36,2.57) -- cycle    ;
\draw  [color={rgb, 255:red, 128; green, 128; blue, 128 }  ,draw opacity=1 ] (477.33,19.98) -- (501.08,19.98) -- (501.08,74.01) -- (477.33,74.01) -- cycle ;
\draw  [color={rgb, 255:red, 128; green, 128; blue, 128 }  ,draw opacity=1 ] (517.33,19.56) -- (541.08,19.56) -- (541.08,73.58) -- (517.33,73.58) -- cycle ;
\draw  [color={rgb, 255:red, 128; green, 128; blue, 128 }  ,draw opacity=1 ] (397.33,20) -- (421.08,20) -- (421.08,74.03) -- (397.33,74.03) -- cycle ;
\draw    (529.08,61.98) .. controls (510.22,82.41) and (470.82,83.5) .. (451.37,64.1) ;
\draw [shift={(449.33,61.91)}, rotate = 49.32] [fill={rgb, 255:red, 0; green, 0; blue, 0 }  ][line width=0.08]  [draw opacity=0] (5.36,-2.57) -- (0,0) -- (5.36,2.57) -- cycle    ;
\draw    (489.33,62.13) -- (526.08,62) ;
\draw [shift={(529.08,61.98)}, rotate = 179.78] [fill={rgb, 255:red, 0; green, 0; blue, 0 }  ][line width=0.08]  [draw opacity=0] (5.36,-2.57) -- (0,0) -- (5.36,2.57) -- cycle    ;
\draw    (449.33,31.91) -- (486.08,31.77) ;
\draw [shift={(489.08,31.76)}, rotate = 179.78] [fill={rgb, 255:red, 0; green, 0; blue, 0 }  ][line width=0.08]  [draw opacity=0] (5.36,-2.57) -- (0,0) -- (5.36,2.57) -- cycle    ;
\draw    (449.33,61.91) -- (486.33,62.12) ;
\draw [shift={(489.33,62.13)}, rotate = 180.32] [fill={rgb, 255:red, 0; green, 0; blue, 0 }  ][line width=0.08]  [draw opacity=0] (5.36,-2.57) -- (0,0) -- (5.36,2.57) -- cycle    ;
\draw    (489.08,31.76) -- (526.08,31.97) ;
\draw [shift={(529.08,31.98)}, rotate = 180.32] [fill={rgb, 255:red, 0; green, 0; blue, 0 }  ][line width=0.08]  [draw opacity=0] (5.36,-2.57) -- (0,0) -- (5.36,2.57) -- cycle    ;

\draw (190,85) node [anchor=north west][inner sep=0.75pt]  [color={rgb, 255:red, 128; green, 128; blue, 128 }  ,opacity=1 ]  {$\Gamma \times \Delta$};
\draw (340,28) node [anchor=north west][inner sep=0.75pt]    {$\phi$};
\draw (434,85) node [anchor=north west][inner sep=0.75pt]  [color={rgb, 255:red, 128; green, 128; blue, 128 }  ,opacity=1 ]  {$\Gamma \times \Delta$};

\end{tikzpicture}
}
\newcommand{\figuremerge}{

\tikzset{every picture/.style={line width=0.75pt}} 

\begin{tikzpicture}[x=0.75pt,y=0.75pt,yscale=-1,xscale=1]

\draw  [color={rgb, 255:red, 208; green, 2; blue, 27 }  ,draw opacity=1 ][fill={rgb, 255:red, 208; green, 2; blue, 27 }  ,fill opacity=1 ] (350.67,86.58) .. controls (350.67,81.05) and (355.14,76.58) .. (360.67,76.58) .. controls (366.19,76.58) and (370.67,81.05) .. (370.67,86.58) .. controls (370.67,92.1) and (366.19,96.58) .. (360.67,96.58) .. controls (355.14,96.58) and (350.67,92.1) .. (350.67,86.58) -- cycle ;
\draw  [color={rgb, 255:red, 208; green, 2; blue, 27 }  ,draw opacity=1 ][fill={rgb, 255:red, 208; green, 2; blue, 27 }  ,fill opacity=0 ] (390.67,116.8) .. controls (390.67,111.28) and (395.14,106.8) .. (400.67,106.8) .. controls (406.19,106.8) and (410.67,111.28) .. (410.67,116.8) .. controls (410.67,122.32) and (406.19,126.8) .. (400.67,126.8) .. controls (395.14,126.8) and (390.67,122.32) .. (390.67,116.8) -- cycle ;
\draw  [color={rgb, 255:red, 208; green, 2; blue, 27 }  ,draw opacity=1 ][fill={rgb, 255:red, 208; green, 2; blue, 27 }  ,fill opacity=0 ] (350.67,116.58) .. controls (350.67,111.05) and (355.14,106.58) .. (360.67,106.58) .. controls (366.19,106.58) and (370.67,111.05) .. (370.67,116.58) .. controls (370.67,122.1) and (366.19,126.58) .. (360.67,126.58) .. controls (355.14,126.58) and (350.67,122.1) .. (350.67,116.58) -- cycle ;
\draw  [color={rgb, 255:red, 208; green, 2; blue, 27 }  ,draw opacity=1 ][fill={rgb, 255:red, 208; green, 2; blue, 27 }  ,fill opacity=1 ] (390.42,86.43) .. controls (390.42,80.9) and (394.89,76.43) .. (400.42,76.43) .. controls (405.94,76.43) and (410.42,80.9) .. (410.42,86.43) .. controls (410.42,91.95) and (405.94,96.43) .. (400.42,96.43) .. controls (394.89,96.43) and (390.42,91.95) .. (390.42,86.43) -- cycle ;
\draw  [color={rgb, 255:red, 128; green, 128; blue, 128 }  ,draw opacity=1 ] (348.67,74.65) -- (372.42,74.65) -- (372.42,128.68) -- (348.67,128.68) -- cycle ;
\draw  [color={rgb, 255:red, 208; green, 2; blue, 27 }  ,draw opacity=1 ][fill={rgb, 255:red, 208; green, 2; blue, 27 }  ,fill opacity=1 ] (430.42,86.65) .. controls (430.42,81.13) and (434.89,76.65) .. (440.42,76.65) .. controls (445.94,76.65) and (450.42,81.13) .. (450.42,86.65) .. controls (450.42,92.17) and (445.94,96.65) .. (440.42,96.65) .. controls (434.89,96.65) and (430.42,92.17) .. (430.42,86.65) -- cycle ;
\draw  [color={rgb, 255:red, 208; green, 2; blue, 27 }  ,draw opacity=1 ] (310.42,117.09) .. controls (310.42,111.57) and (314.89,107.09) .. (320.42,107.09) .. controls (325.94,107.09) and (330.42,111.57) .. (330.42,117.09) .. controls (330.42,122.61) and (325.94,127.09) .. (320.42,127.09) .. controls (314.89,127.09) and (310.42,122.61) .. (310.42,117.09) -- cycle ;
\draw  [color={rgb, 255:red, 208; green, 2; blue, 27 }  ,draw opacity=1 ] (430.42,116.65) .. controls (430.42,111.13) and (434.89,106.65) .. (440.42,106.65) .. controls (445.94,106.65) and (450.42,111.13) .. (450.42,116.65) .. controls (450.42,122.17) and (445.94,126.65) .. (440.42,126.65) .. controls (434.89,126.65) and (430.42,122.17) .. (430.42,116.65) -- cycle ;
\draw  [color={rgb, 255:red, 208; green, 2; blue, 27 }  ,draw opacity=1 ][fill={rgb, 255:red, 208; green, 2; blue, 27 }  ,fill opacity=1 ] (310.42,87.09) .. controls (310.42,81.57) and (314.89,77.09) .. (320.42,77.09) .. controls (325.94,77.09) and (330.42,81.57) .. (330.42,87.09) .. controls (330.42,92.61) and (325.94,97.09) .. (320.42,97.09) .. controls (314.89,97.09) and (310.42,92.61) .. (310.42,87.09) -- cycle ;
\draw    (440.42,86.65) .. controls (420.29,65.28) and (381.3,64.26) .. (362.36,84.62) ;
\draw [shift={(360.67,86.58)}, rotate = 309.15] [fill={rgb, 255:red, 0; green, 0; blue, 0 }  ][line width=0.08]  [draw opacity=0] (5.36,-2.57) -- (0,0) -- (5.36,2.57) -- cycle    ;
\draw  [color={rgb, 255:red, 128; green, 128; blue, 128 }  ,draw opacity=1 ] (388.67,74.65) -- (412.42,74.65) -- (412.42,128.68) -- (388.67,128.68) -- cycle ;
\draw  [color={rgb, 255:red, 128; green, 128; blue, 128 }  ,draw opacity=1 ] (428.67,74.23) -- (452.42,74.23) -- (452.42,128.25) -- (428.67,128.25) -- cycle ;
\draw  [color={rgb, 255:red, 128; green, 128; blue, 128 }  ,draw opacity=1 ] (308.67,74.67) -- (332.42,74.67) -- (332.42,128.69) -- (308.67,128.69) -- cycle ;
\draw    (440.42,116.65) .. controls (421.55,137.08) and (382.15,138.16) .. (362.7,118.77) ;
\draw [shift={(360.67,116.58)}, rotate = 49.32] [fill={rgb, 255:red, 0; green, 0; blue, 0 }  ][line width=0.08]  [draw opacity=0] (5.36,-2.57) -- (0,0) -- (5.36,2.57) -- cycle    ;
\draw    (400.67,116.8) -- (437.42,116.66) ;
\draw [shift={(440.42,116.65)}, rotate = 179.78] [fill={rgb, 255:red, 0; green, 0; blue, 0 }  ][line width=0.08]  [draw opacity=0] (5.36,-2.57) -- (0,0) -- (5.36,2.57) -- cycle    ;
\draw    (360.67,86.58) -- (397.42,86.44) ;
\draw [shift={(400.42,86.43)}, rotate = 179.78] [fill={rgb, 255:red, 0; green, 0; blue, 0 }  ][line width=0.08]  [draw opacity=0] (5.36,-2.57) -- (0,0) -- (5.36,2.57) -- cycle    ;
\draw    (360.67,116.58) -- (397.67,116.78) ;
\draw [shift={(400.67,116.8)}, rotate = 180.32] [fill={rgb, 255:red, 0; green, 0; blue, 0 }  ][line width=0.08]  [draw opacity=0] (5.36,-2.57) -- (0,0) -- (5.36,2.57) -- cycle    ;
\draw    (400.42,86.43) -- (437.42,86.63) ;
\draw [shift={(440.42,86.65)}, rotate = 180.32] [fill={rgb, 255:red, 0; green, 0; blue, 0 }  ][line width=0.08]  [draw opacity=0] (5.36,-2.57) -- (0,0) -- (5.36,2.57) -- cycle    ;
\draw  [color={rgb, 255:red, 208; green, 2; blue, 27 }  ,draw opacity=1 ][fill={rgb, 255:red, 208; green, 2; blue, 27 }  ,fill opacity=0.3 ] (310.67,39.16) .. controls (310.67,33.63) and (315.14,29.16) .. (320.67,29.16) .. controls (326.19,29.16) and (330.67,33.63) .. (330.67,39.16) .. controls (330.67,44.68) and (326.19,49.16) .. (320.67,49.16) .. controls (315.14,49.16) and (310.67,44.68) .. (310.67,39.16) -- cycle ;
\draw  [color={rgb, 255:red, 128; green, 128; blue, 128 }  ,draw opacity=1 ] (308.67,27.23) -- (332.42,27.23) -- (332.42,51.48) -- (308.67,51.48) -- cycle ;
\draw  [color={rgb, 255:red, 208; green, 2; blue, 27 }  ,draw opacity=1 ][fill={rgb, 255:red, 208; green, 2; blue, 27 }  ,fill opacity=0.3 ] (350.59,39.31) .. controls (350.59,33.79) and (355.07,29.31) .. (360.59,29.31) .. controls (366.11,29.31) and (370.59,33.79) .. (370.59,39.31) .. controls (370.59,44.83) and (366.11,49.31) .. (360.59,49.31) .. controls (355.07,49.31) and (350.59,44.83) .. (350.59,39.31) -- cycle ;
\draw  [color={rgb, 255:red, 128; green, 128; blue, 128 }  ,draw opacity=1 ] (348.59,27.38) -- (372.34,27.38) -- (372.34,51.63) -- (348.59,51.63) -- cycle ;
\draw    (320.89,39.68) -- (357.64,39.54) ;
\draw [shift={(360.64,39.53)}, rotate = 179.78] [fill={rgb, 255:red, 0; green, 0; blue, 0 }  ][line width=0.08]  [draw opacity=0] (5.36,-2.57) -- (0,0) -- (5.36,2.57) -- cycle    ;
\draw  [color={rgb, 255:red, 208; green, 2; blue, 27 }  ,draw opacity=1 ][fill={rgb, 255:red, 208; green, 2; blue, 27 }  ,fill opacity=0.3 ] (390.64,39.76) .. controls (390.64,34.23) and (395.11,29.76) .. (400.64,29.76) .. controls (406.16,29.76) and (410.64,34.23) .. (410.64,39.76) .. controls (410.64,45.28) and (406.16,49.76) .. (400.64,49.76) .. controls (395.11,49.76) and (390.64,45.28) .. (390.64,39.76) -- cycle ;
\draw  [color={rgb, 255:red, 128; green, 128; blue, 128 }  ,draw opacity=1 ] (388.64,27.83) -- (412.39,27.83) -- (412.39,52.08) -- (388.64,52.08) -- cycle ;
\draw    (360.64,39.53) -- (397.64,39.74) ;
\draw [shift={(400.64,39.76)}, rotate = 180.32] [fill={rgb, 255:red, 0; green, 0; blue, 0 }  ][line width=0.08]  [draw opacity=0] (5.36,-2.57) -- (0,0) -- (5.36,2.57) -- cycle    ;
\draw    (400.42,39.23) .. controls (380.29,17.86) and (341.3,16.84) .. (322.36,37.2) ;
\draw [shift={(320.67,39.16)}, rotate = 309.15] [fill={rgb, 255:red, 0; green, 0; blue, 0 }  ][line width=0.08]  [draw opacity=0] (5.36,-2.57) -- (0,0) -- (5.36,2.57) -- cycle    ;
\draw  [color={rgb, 255:red, 208; green, 2; blue, 27 }  ,draw opacity=1 ][fill={rgb, 255:red, 208; green, 2; blue, 27 }  ,fill opacity=0.3 ] (430.64,39.76) .. controls (430.64,34.23) and (435.11,29.76) .. (440.64,29.76) .. controls (446.16,29.76) and (450.64,34.23) .. (450.64,39.76) .. controls (450.64,45.28) and (446.16,49.76) .. (440.64,49.76) .. controls (435.11,49.76) and (430.64,45.28) .. (430.64,39.76) -- cycle ;
\draw  [color={rgb, 255:red, 128; green, 128; blue, 128 }  ,draw opacity=1 ] (428.64,27.83) -- (452.39,27.83) -- (452.39,52.08) -- (428.64,52.08) -- cycle ;
\draw  [dash pattern={on 0.84pt off 2.51pt}]  (401.5,52.93) -- (437.74,72.37) ;
\draw [shift={(439.5,73.32)}, rotate = 208.21] [color={rgb, 255:red, 0; green, 0; blue, 0 }  ][line width=0.75]    (7.65,-2.3) .. controls (4.86,-0.97) and (2.31,-0.21) .. (0,0) .. controls (2.31,0.21) and (4.86,0.98) .. (7.65,2.3)   ;
\draw [color={rgb, 255:red, 144; green, 19; blue, 254 }  ,draw opacity=1 ]   (320.42,87.09) .. controls (310.88,96.1) and (310.61,105.62) .. (318.5,114.98) ;
\draw [shift={(320.42,117.09)}, rotate = 225.67] [fill={rgb, 255:red, 144; green, 19; blue, 254 }  ,fill opacity=1 ][line width=0.08]  [draw opacity=0] (5.36,-2.57) -- (0,0) -- (5.36,2.57) -- cycle    ;
\draw [color={rgb, 255:red, 144; green, 19; blue, 254 }  ,draw opacity=1 ]   (320.42,117.09) .. controls (329.9,107.85) and (330.57,98.56) .. (322.41,89.2) ;
\draw [shift={(320.42,87.09)}, rotate = 44.62] [fill={rgb, 255:red, 144; green, 19; blue, 254 }  ,fill opacity=1 ][line width=0.08]  [draw opacity=0] (5.36,-2.57) -- (0,0) -- (5.36,2.57) -- cycle    ;

\draw (253,96) node [anchor=north west][inner sep=0.75pt]  [color={rgb, 255:red, 128; green, 128; blue, 128 }  ,opacity=1 ]  {$\Gamma _{j} \times \Delta _{j}$};
\draw (253,32) node [anchor=north west][inner sep=0.75pt]  [color={rgb, 255:red, 128; green, 128; blue, 128 }  ,opacity=1 ]  {$\Gamma _{i} \times \Delta _{i}$};
\draw (415,47) node [anchor=north west][inner sep=0.75pt]    {$\xi $};

\end{tikzpicture}
}
\newcommand{\figuremerged}{

\tikzset{every picture/.style={line width=0.75pt}} 

\begin{tikzpicture}[x=0.75pt,y=0.75pt,yscale=-1,xscale=1]
\path (600,100);

\draw  [color={rgb, 255:red, 208; green, 2; blue, 27 }  ,draw opacity=1 ][fill={rgb, 255:red, 208; green, 2; blue, 27 }  ,fill opacity=1 ] (350.67,86.58) .. controls (350.67,81.05) and (355.14,76.58) .. (360.67,76.58) .. controls (366.19,76.58) and (370.67,81.05) .. (370.67,86.58) .. controls (370.67,92.1) and (366.19,96.58) .. (360.67,96.58) .. controls (355.14,96.58) and (350.67,92.1) .. (350.67,86.58) -- cycle ;
\draw  [color={rgb, 255:red, 208; green, 2; blue, 27 }  ,draw opacity=1 ][fill={rgb, 255:red, 208; green, 2; blue, 27 }  ,fill opacity=0 ] (390.67,116.8) .. controls (390.67,111.28) and (395.14,106.8) .. (400.67,106.8) .. controls (406.19,106.8) and (410.67,111.28) .. (410.67,116.8) .. controls (410.67,122.32) and (406.19,126.8) .. (400.67,126.8) .. controls (395.14,126.8) and (390.67,122.32) .. (390.67,116.8) -- cycle ;
\draw  [color={rgb, 255:red, 208; green, 2; blue, 27 }  ,draw opacity=1 ][fill={rgb, 255:red, 208; green, 2; blue, 27 }  ,fill opacity=0 ] (350.67,116.58) .. controls (350.67,111.05) and (355.14,106.58) .. (360.67,106.58) .. controls (366.19,106.58) and (370.67,111.05) .. (370.67,116.58) .. controls (370.67,122.1) and (366.19,126.58) .. (360.67,126.58) .. controls (355.14,126.58) and (350.67,122.1) .. (350.67,116.58) -- cycle ;
\draw  [color={rgb, 255:red, 208; green, 2; blue, 27 }  ,draw opacity=1 ][fill={rgb, 255:red, 208; green, 2; blue, 27 }  ,fill opacity=1 ] (390.42,86.43) .. controls (390.42,80.9) and (394.89,76.43) .. (400.42,76.43) .. controls (405.94,76.43) and (410.42,80.9) .. (410.42,86.43) .. controls (410.42,91.95) and (405.94,96.43) .. (400.42,96.43) .. controls (394.89,96.43) and (390.42,91.95) .. (390.42,86.43) -- cycle ;
\draw  [color={rgb, 255:red, 208; green, 2; blue, 27 }  ,draw opacity=1 ][fill={rgb, 255:red, 208; green, 2; blue, 27 }  ,fill opacity=1 ] (430.42,86.65) .. controls (430.42,81.13) and (434.89,76.65) .. (440.42,76.65) .. controls (445.94,76.65) and (450.42,81.13) .. (450.42,86.65) .. controls (450.42,92.17) and (445.94,96.65) .. (440.42,96.65) .. controls (434.89,96.65) and (430.42,92.17) .. (430.42,86.65) -- cycle ;
\draw  [color={rgb, 255:red, 208; green, 2; blue, 27 }  ,draw opacity=1 ] (310.42,117.09) .. controls (310.42,111.57) and (314.89,107.09) .. (320.42,107.09) .. controls (325.94,107.09) and (330.42,111.57) .. (330.42,117.09) .. controls (330.42,122.61) and (325.94,127.09) .. (320.42,127.09) .. controls (314.89,127.09) and (310.42,122.61) .. (310.42,117.09) -- cycle ;
\draw  [color={rgb, 255:red, 208; green, 2; blue, 27 }  ,draw opacity=1 ] (430.42,116.65) .. controls (430.42,111.13) and (434.89,106.65) .. (440.42,106.65) .. controls (445.94,106.65) and (450.42,111.13) .. (450.42,116.65) .. controls (450.42,122.17) and (445.94,126.65) .. (440.42,126.65) .. controls (434.89,126.65) and (430.42,122.17) .. (430.42,116.65) -- cycle ;
\draw  [color={rgb, 255:red, 208; green, 2; blue, 27 }  ,draw opacity=1 ][fill={rgb, 255:red, 208; green, 2; blue, 27 }  ,fill opacity=1 ] (310.42,87.09) .. controls (310.42,81.57) and (314.89,77.09) .. (320.42,77.09) .. controls (325.94,77.09) and (330.42,81.57) .. (330.42,87.09) .. controls (330.42,92.61) and (325.94,97.09) .. (320.42,97.09) .. controls (314.89,97.09) and (310.42,92.61) .. (310.42,87.09) -- cycle ;
\draw    (440.42,86.65) .. controls (420.29,65.28) and (381.3,64.26) .. (362.36,84.62) ;
\draw [shift={(360.67,86.58)}, rotate = 309.15] [fill={rgb, 255:red, 0; green, 0; blue, 0 }  ][line width=0.08]  [draw opacity=0] (5.36,-2.57) -- (0,0) -- (5.36,2.57) -- cycle    ;
\draw    (440.42,116.65) .. controls (421.55,137.08) and (382.15,138.16) .. (362.7,118.77) ;
\draw [shift={(360.67,116.58)}, rotate = 49.32] [fill={rgb, 255:red, 0; green, 0; blue, 0 }  ][line width=0.08]  [draw opacity=0] (5.36,-2.57) -- (0,0) -- (5.36,2.57) -- cycle    ;
\draw    (400.67,116.8) -- (437.42,116.66) ;
\draw [shift={(440.42,116.65)}, rotate = 179.78] [fill={rgb, 255:red, 0; green, 0; blue, 0 }  ][line width=0.08]  [draw opacity=0] (5.36,-2.57) -- (0,0) -- (5.36,2.57) -- cycle    ;
\draw    (360.67,86.58) -- (397.42,86.44) ;
\draw [shift={(400.42,86.43)}, rotate = 179.78] [fill={rgb, 255:red, 0; green, 0; blue, 0 }  ][line width=0.08]  [draw opacity=0] (5.36,-2.57) -- (0,0) -- (5.36,2.57) -- cycle    ;
\draw    (360.67,116.58) -- (397.67,116.78) ;
\draw [shift={(400.67,116.8)}, rotate = 180.32] [fill={rgb, 255:red, 0; green, 0; blue, 0 }  ][line width=0.08]  [draw opacity=0] (5.36,-2.57) -- (0,0) -- (5.36,2.57) -- cycle    ;
\draw    (400.42,86.43) -- (437.42,86.63) ;
\draw [shift={(440.42,86.65)}, rotate = 180.32] [fill={rgb, 255:red, 0; green, 0; blue, 0 }  ][line width=0.08]  [draw opacity=0] (5.36,-2.57) -- (0,0) -- (5.36,2.57) -- cycle    ;
\draw  [color={rgb, 255:red, 208; green, 2; blue, 27 }  ,draw opacity=1 ][fill={rgb, 255:red, 208; green, 2; blue, 27 }  ,fill opacity=0.3 ] (310.67,39.16) .. controls (310.67,33.63) and (315.14,29.16) .. (320.67,29.16) .. controls (326.19,29.16) and (330.67,33.63) .. (330.67,39.16) .. controls (330.67,44.68) and (326.19,49.16) .. (320.67,49.16) .. controls (315.14,49.16) and (310.67,44.68) .. (310.67,39.16) -- cycle ;
\draw  [color={rgb, 255:red, 208; green, 2; blue, 27 }  ,draw opacity=1 ][fill={rgb, 255:red, 208; green, 2; blue, 27 }  ,fill opacity=0.3 ] (350.59,39.31) .. controls (350.59,33.79) and (355.07,29.31) .. (360.59,29.31) .. controls (366.11,29.31) and (370.59,33.79) .. (370.59,39.31) .. controls (370.59,44.83) and (366.11,49.31) .. (360.59,49.31) .. controls (355.07,49.31) and (350.59,44.83) .. (350.59,39.31) -- cycle ;
\draw    (320.89,39.68) -- (357.64,39.54) ;
\draw [shift={(360.64,39.53)}, rotate = 179.78] [fill={rgb, 255:red, 0; green, 0; blue, 0 }  ][line width=0.08]  [draw opacity=0] (5.36,-2.57) -- (0,0) -- (5.36,2.57) -- cycle    ;
\draw  [color={rgb, 255:red, 208; green, 2; blue, 27 }  ,draw opacity=1 ][fill={rgb, 255:red, 208; green, 2; blue, 27 }  ,fill opacity=0.3 ] (390.64,39.76) .. controls (390.64,34.23) and (395.11,29.76) .. (400.64,29.76) .. controls (406.16,29.76) and (410.64,34.23) .. (410.64,39.76) .. controls (410.64,45.28) and (406.16,49.76) .. (400.64,49.76) .. controls (395.11,49.76) and (390.64,45.28) .. (390.64,39.76) -- cycle ;
\draw    (360.64,39.53) -- (397.64,39.74) ;
\draw [shift={(400.64,39.76)}, rotate = 180.32] [fill={rgb, 255:red, 0; green, 0; blue, 0 }  ][line width=0.08]  [draw opacity=0] (5.36,-2.57) -- (0,0) -- (5.36,2.57) -- cycle    ;
\draw    (400.42,39.23) .. controls (380.29,17.86) and (341.3,16.84) .. (322.36,37.2) ;
\draw [shift={(320.67,39.16)}, rotate = 309.15] [fill={rgb, 255:red, 0; green, 0; blue, 0 }  ][line width=0.08]  [draw opacity=0] (5.36,-2.57) -- (0,0) -- (5.36,2.57) -- cycle    ;
\draw  [color={rgb, 255:red, 208; green, 2; blue, 27 }  ,draw opacity=1 ][fill={rgb, 255:red, 208; green, 2; blue, 27 }  ,fill opacity=0.3 ] (430.64,39.76) .. controls (430.64,34.23) and (435.11,29.76) .. (440.64,29.76) .. controls (446.16,29.76) and (450.64,34.23) .. (450.64,39.76) .. controls (450.64,45.28) and (446.16,49.76) .. (440.64,49.76) .. controls (435.11,49.76) and (430.64,45.28) .. (430.64,39.76) -- cycle ;
\draw  [color={rgb, 255:red, 128; green, 128; blue, 128 }  ,draw opacity=1 ] (308.71,27.28) -- (332.46,27.28) -- (332.46,51.53) -- (372.54,74.55) -- (372.54,128.58) -- (348.79,128.58) -- (348.79,74.55) -- (308.71,51.53) -- cycle ;
\draw  [color={rgb, 255:red, 128; green, 128; blue, 128 }  ,draw opacity=1 ] (348.71,27.51) -- (372.46,27.51) -- (372.46,51.76) -- (412.54,74.78) -- (412.54,128.8) -- (388.79,128.8) -- (388.79,74.78) -- (348.71,51.76) -- cycle ;
\draw  [color={rgb, 255:red, 128; green, 128; blue, 128 }  ,draw opacity=1 ] (388.46,27.36) -- (412.21,27.36) -- (412.21,51.61) -- (452.29,74.63) -- (452.29,128.65) -- (428.54,128.65) -- (428.54,74.63) -- (388.46,51.61) -- cycle ;
\draw [color={rgb, 255:red, 128; green, 128; blue, 128 }  ,draw opacity=1 ]   (428.76,52) -- (428.76,27.76) ;
\draw [color={rgb, 255:red, 128; green, 128; blue, 128 }  ,draw opacity=1 ]   (452.51,27.76) -- (428.76,27.76) ;
\draw [color={rgb, 255:red, 128; green, 128; blue, 128 }  ,draw opacity=1 ]   (452.51,52) -- (452.51,27.76) ;
\draw [color={rgb, 255:red, 128; green, 128; blue, 128 }  ,draw opacity=1 ]   (308.67,128.69) -- (308.67,74.67) ;
\draw [color={rgb, 255:red, 128; green, 128; blue, 128 }  ,draw opacity=1 ]   (332.42,128.69) -- (332.42,74.67) ;
\draw [color={rgb, 255:red, 128; green, 128; blue, 128 }  ,draw opacity=1 ]   (332.42,128.69) -- (308.67,128.69) ;
\draw [color={rgb, 255:red, 128; green, 128; blue, 128 }  ,draw opacity=1 ]   (452.33,64.73) -- (428.76,52) ;
\draw [color={rgb, 255:red, 128; green, 128; blue, 128 }  ,draw opacity=1 ]   (332.42,74.67) -- (308.33,61.4) ;
\draw [color={rgb, 255:red, 144; green, 19; blue, 254 }  ,draw opacity=1 ]   (320.42,87.09) .. controls (310.88,96.1) and (310.61,105.62) .. (318.5,114.98) ;
\draw [shift={(320.42,117.09)}, rotate = 225.67] [fill={rgb, 255:red, 144; green, 19; blue, 254 }  ,fill opacity=1 ][line width=0.08]  [draw opacity=0] (5.36,-2.57) -- (0,0) -- (5.36,2.57) -- cycle    ;
\draw [color={rgb, 255:red, 144; green, 19; blue, 254 }  ,draw opacity=1 ]   (320.42,117.09) .. controls (329.9,107.85) and (330.57,98.56) .. (322.41,89.2) ;
\draw [shift={(320.42,87.09)}, rotate = 44.62] [fill={rgb, 255:red, 144; green, 19; blue, 254 }  ,fill opacity=1 ][line width=0.08]  [draw opacity=0] (5.36,-2.57) -- (0,0) -- (5.36,2.57) -- cycle    ;

\draw (210, 70) node [anchor=north west][inner sep=0.75pt]  [color={rgb, 255:red, 128; green, 128; blue, 128 }  ,opacity=1 ]  {$( \Gamma _{i} \sqcup \Gamma _{j}) \times \Delta _{i}$};

\end{tikzpicture}

}
\newcommand{\figuremix}{

\tikzset{every picture/.style={line width=0.75pt}} 

\begin{tikzpicture}[x=0.75pt,y=0.75pt,yscale=-1,xscale=1]
\path (450,120); 

\draw  [color={rgb, 255:red, 208; green, 2; blue, 27 }  ,draw opacity=1 ][fill={rgb, 255:red, 208; green, 2; blue, 27 }  ,fill opacity=0.3 ] (310.67,39.16) .. controls (310.67,33.63) and (315.14,29.16) .. (320.67,29.16) .. controls (326.19,29.16) and (330.67,33.63) .. (330.67,39.16) .. controls (330.67,44.68) and (326.19,49.16) .. (320.67,49.16) .. controls (315.14,49.16) and (310.67,44.68) .. (310.67,39.16) -- cycle ;
\draw  [color={rgb, 255:red, 128; green, 128; blue, 128 }  ,draw opacity=1 ] (308.67,27.23) -- (332.42,27.23) -- (332.42,51.48) -- (308.67,51.48) -- cycle ;
\draw  [color={rgb, 255:red, 208; green, 2; blue, 27 }  ,draw opacity=1 ][fill={rgb, 255:red, 208; green, 2; blue, 27 }  ,fill opacity=0.3 ] (350.59,39.31) .. controls (350.59,33.79) and (355.07,29.31) .. (360.59,29.31) .. controls (366.11,29.31) and (370.59,33.79) .. (370.59,39.31) .. controls (370.59,44.83) and (366.11,49.31) .. (360.59,49.31) .. controls (355.07,49.31) and (350.59,44.83) .. (350.59,39.31) -- cycle ;
\draw  [color={rgb, 255:red, 128; green, 128; blue, 128 }  ,draw opacity=1 ] (348.59,27.38) -- (372.34,27.38) -- (372.34,51.63) -- (348.59,51.63) -- cycle ;
\draw  [color={rgb, 255:red, 208; green, 2; blue, 27 }  ,draw opacity=1 ][fill={rgb, 255:red, 208; green, 2; blue, 27 }  ,fill opacity=0.3 ] (390.64,39.76) .. controls (390.64,34.23) and (395.11,29.76) .. (400.64,29.76) .. controls (406.16,29.76) and (410.64,34.23) .. (410.64,39.76) .. controls (410.64,45.28) and (406.16,49.76) .. (400.64,49.76) .. controls (395.11,49.76) and (390.64,45.28) .. (390.64,39.76) -- cycle ;
\draw  [color={rgb, 255:red, 128; green, 128; blue, 128 }  ,draw opacity=1 ] (388.64,27.83) -- (412.39,27.83) -- (412.39,52.08) -- (388.64,52.08) -- cycle ;
\draw  [color={rgb, 255:red, 208; green, 2; blue, 27 }  ,draw opacity=1 ][fill={rgb, 255:red, 208; green, 2; blue, 27 }  ,fill opacity=0.3 ] (430.64,39.76) .. controls (430.64,34.23) and (435.11,29.76) .. (440.64,29.76) .. controls (446.16,29.76) and (450.64,34.23) .. (450.64,39.76) .. controls (450.64,45.28) and (446.16,49.76) .. (440.64,49.76) .. controls (435.11,49.76) and (430.64,45.28) .. (430.64,39.76) -- cycle ;
\draw  [color={rgb, 255:red, 128; green, 128; blue, 128 }  ,draw opacity=1 ] (428.64,27.83) -- (452.39,27.83) -- (452.39,52.08) -- (428.64,52.08) -- cycle ;
\draw  [color={rgb, 255:red, 208; green, 2; blue, 27 }  ,draw opacity=1 ][fill={rgb, 255:red, 208; green, 2; blue, 27 }  ,fill opacity=1 ] (310.67,86.16) .. controls (310.67,80.63) and (315.14,76.16) .. (320.67,76.16) .. controls (326.19,76.16) and (330.67,80.63) .. (330.67,86.16) .. controls (330.67,91.68) and (326.19,96.16) .. (320.67,96.16) .. controls (315.14,96.16) and (310.67,91.68) .. (310.67,86.16) -- cycle ;
\draw  [color={rgb, 255:red, 128; green, 128; blue, 128 }  ,draw opacity=1 ] (308.67,74.23) -- (332.42,74.23) -- (332.42,98.48) -- (308.67,98.48) -- cycle ;
\draw  [color={rgb, 255:red, 208; green, 2; blue, 27 }  ,draw opacity=1 ][fill={rgb, 255:red, 208; green, 2; blue, 27 }  ,fill opacity=1 ] (350.59,86.31) .. controls (350.59,80.79) and (355.07,76.31) .. (360.59,76.31) .. controls (366.11,76.31) and (370.59,80.79) .. (370.59,86.31) .. controls (370.59,91.83) and (366.11,96.31) .. (360.59,96.31) .. controls (355.07,96.31) and (350.59,91.83) .. (350.59,86.31) -- cycle ;
\draw  [color={rgb, 255:red, 128; green, 128; blue, 128 }  ,draw opacity=1 ] (348.59,74.38) -- (372.34,74.38) -- (372.34,98.63) -- (348.59,98.63) -- cycle ;
\draw    (320.89,86.68) -- (357.64,86.54) ;
\draw [shift={(360.64,86.53)}, rotate = 179.78] [fill={rgb, 255:red, 0; green, 0; blue, 0 }  ][line width=0.08]  [draw opacity=0] (5.36,-2.57) -- (0,0) -- (5.36,2.57) -- cycle    ;
\draw  [color={rgb, 255:red, 208; green, 2; blue, 27 }  ,draw opacity=1 ][fill={rgb, 255:red, 208; green, 2; blue, 27 }  ,fill opacity=1 ] (390.64,86.76) .. controls (390.64,81.23) and (395.11,76.76) .. (400.64,76.76) .. controls (406.16,76.76) and (410.64,81.23) .. (410.64,86.76) .. controls (410.64,92.28) and (406.16,96.76) .. (400.64,96.76) .. controls (395.11,96.76) and (390.64,92.28) .. (390.64,86.76) -- cycle ;
\draw  [color={rgb, 255:red, 128; green, 128; blue, 128 }  ,draw opacity=1 ] (388.64,74.83) -- (412.39,74.83) -- (412.39,99.08) -- (388.64,99.08) -- cycle ;
\draw    (360.64,86.53) -- (397.64,86.74) ;
\draw [shift={(400.64,86.76)}, rotate = 180.32] [fill={rgb, 255:red, 0; green, 0; blue, 0 }  ][line width=0.08]  [draw opacity=0] (5.36,-2.57) -- (0,0) -- (5.36,2.57) -- cycle    ;
\draw    (400.42,86.23) .. controls (380.29,64.86) and (341.3,63.84) .. (322.36,84.2) ;
\draw [shift={(320.67,86.16)}, rotate = 309.15] [fill={rgb, 255:red, 0; green, 0; blue, 0 }  ][line width=0.08]  [draw opacity=0] (5.36,-2.57) -- (0,0) -- (5.36,2.57) -- cycle    ;
\draw  [color={rgb, 255:red, 208; green, 2; blue, 27 }  ,draw opacity=1 ][fill={rgb, 255:red, 208; green, 2; blue, 27 }  ,fill opacity=1 ] (430.64,86.76) .. controls (430.64,81.23) and (435.11,76.76) .. (440.64,76.76) .. controls (446.16,76.76) and (450.64,81.23) .. (450.64,86.76) .. controls (450.64,92.28) and (446.16,96.76) .. (440.64,96.76) .. controls (435.11,96.76) and (430.64,92.28) .. (430.64,86.76) -- cycle ;
\draw  [color={rgb, 255:red, 128; green, 128; blue, 128 }  ,draw opacity=1 ] (428.64,74.83) -- (452.39,74.83) -- (452.39,99.08) -- (428.64,99.08) -- cycle ;
\draw    (400.51,39.95) .. controls (389.38,29.09) and (373.69,28.15) .. (362.54,37.58) ;
\draw [shift={(360.46,39.51)}, rotate = 314.39] [fill={rgb, 255:red, 0; green, 0; blue, 0 }  ][line width=0.08]  [draw opacity=0] (5.36,-2.57) -- (0,0) -- (5.36,2.57) -- cycle    ;
\draw    (360.46,39.51) .. controls (371.62,50.69) and (386.8,51.39) .. (398.35,41.9) ;
\draw [shift={(400.51,39.95)}, rotate = 135.68] [fill={rgb, 255:red, 0; green, 0; blue, 0 }  ][line width=0.08]  [draw opacity=0] (5.36,-2.57) -- (0,0) -- (5.36,2.57) -- cycle    ;

\draw (243.67,30.07) node [anchor=north west][inner sep=0.75pt]  [color={rgb, 255:red, 128; green, 128; blue, 128 }  ,opacity=1 ]  {$\Gamma _{i} \times \Delta _{i}$};
\draw (243.67,77.07) node [anchor=north west][inner sep=0.75pt]  [color={rgb, 255:red, 128; green, 128; blue, 128 }  ,opacity=1 ]  {$\Gamma _{j} \times \Delta _{j}$};
\draw (370.75,56) node [anchor=north west][inner sep=0.75pt]    {$\tau $};

\end{tikzpicture}

}
\newcommand{\figuremixed}{

\tikzset{every picture/.style={line width=0.75pt}} 

\begin{tikzpicture}[x=0.75pt,y=0.75pt,yscale=-1,xscale=1]
\path (450,120);

\draw  [color={rgb, 255:red, 208; green, 2; blue, 27 }  ,draw opacity=1 ][fill={rgb, 255:red, 208; green, 2; blue, 27 }  ,fill opacity=0.3 ] (310.67,39.16) .. controls (310.67,33.63) and (315.14,29.16) .. (320.67,29.16) .. controls (326.19,29.16) and (330.67,33.63) .. (330.67,39.16) .. controls (330.67,44.68) and (326.19,49.16) .. (320.67,49.16) .. controls (315.14,49.16) and (310.67,44.68) .. (310.67,39.16) -- cycle ;
\draw  [color={rgb, 255:red, 128; green, 128; blue, 128 }  ,draw opacity=1 ] (308.67,27.23) -- (332.42,27.23) -- (332.42,51.48) -- (308.67,51.48) -- cycle ;
\draw  [color={rgb, 255:red, 208; green, 2; blue, 27 }  ,draw opacity=1 ][fill={rgb, 255:red, 208; green, 2; blue, 27 }  ,fill opacity=0.3 ] (350.59,39.31) .. controls (350.59,33.79) and (355.07,29.31) .. (360.59,29.31) .. controls (366.11,29.31) and (370.59,33.79) .. (370.59,39.31) .. controls (370.59,44.83) and (366.11,49.31) .. (360.59,49.31) .. controls (355.07,49.31) and (350.59,44.83) .. (350.59,39.31) -- cycle ;
\draw  [color={rgb, 255:red, 128; green, 128; blue, 128 }  ,draw opacity=1 ] (348.59,27.38) -- (372.34,27.38) -- (372.34,51.63) -- (348.59,51.63) -- cycle ;
\draw  [color={rgb, 255:red, 208; green, 2; blue, 27 }  ,draw opacity=1 ][fill={rgb, 255:red, 208; green, 2; blue, 27 }  ,fill opacity=0.3 ] (390.64,39.76) .. controls (390.64,34.23) and (395.11,29.76) .. (400.64,29.76) .. controls (406.16,29.76) and (410.64,34.23) .. (410.64,39.76) .. controls (410.64,45.28) and (406.16,49.76) .. (400.64,49.76) .. controls (395.11,49.76) and (390.64,45.28) .. (390.64,39.76) -- cycle ;
\draw  [color={rgb, 255:red, 128; green, 128; blue, 128 }  ,draw opacity=1 ] (388.64,27.83) -- (412.39,27.83) -- (412.39,52.08) -- (388.64,52.08) -- cycle ;
\draw  [color={rgb, 255:red, 208; green, 2; blue, 27 }  ,draw opacity=1 ][fill={rgb, 255:red, 208; green, 2; blue, 27 }  ,fill opacity=0.3 ] (430.64,39.76) .. controls (430.64,34.23) and (435.11,29.76) .. (440.64,29.76) .. controls (446.16,29.76) and (450.64,34.23) .. (450.64,39.76) .. controls (450.64,45.28) and (446.16,49.76) .. (440.64,49.76) .. controls (435.11,49.76) and (430.64,45.28) .. (430.64,39.76) -- cycle ;
\draw  [color={rgb, 255:red, 128; green, 128; blue, 128 }  ,draw opacity=1 ] (428.64,27.83) -- (452.39,27.83) -- (452.39,52.08) -- (428.64,52.08) -- cycle ;
\draw  [color={rgb, 255:red, 208; green, 2; blue, 27 }  ,draw opacity=1 ][fill={rgb, 255:red, 208; green, 2; blue, 27 }  ,fill opacity=1 ] (310.67,86.16) .. controls (310.67,80.63) and (315.14,76.16) .. (320.67,76.16) .. controls (326.19,76.16) and (330.67,80.63) .. (330.67,86.16) .. controls (330.67,91.68) and (326.19,96.16) .. (320.67,96.16) .. controls (315.14,96.16) and (310.67,91.68) .. (310.67,86.16) -- cycle ;
\draw  [color={rgb, 255:red, 128; green, 128; blue, 128 }  ,draw opacity=1 ] (308.67,74.23) -- (332.42,74.23) -- (332.42,98.48) -- (308.67,98.48) -- cycle ;
\draw  [color={rgb, 255:red, 208; green, 2; blue, 27 }  ,draw opacity=1 ][fill={rgb, 255:red, 208; green, 2; blue, 27 }  ,fill opacity=1 ] (350.59,86.31) .. controls (350.59,80.79) and (355.07,76.31) .. (360.59,76.31) .. controls (366.11,76.31) and (370.59,80.79) .. (370.59,86.31) .. controls (370.59,91.83) and (366.11,96.31) .. (360.59,96.31) .. controls (355.07,96.31) and (350.59,91.83) .. (350.59,86.31) -- cycle ;
\draw  [color={rgb, 255:red, 128; green, 128; blue, 128 }  ,draw opacity=1 ] (348.59,74.38) -- (372.34,74.38) -- (372.34,98.63) -- (348.59,98.63) -- cycle ;
\draw [color={rgb, 255:red, 74; green, 144; blue, 226 }  ,draw opacity=1 ]   (323.89,86.67) -- (360.64,86.53) ;
\draw [shift={(320.89,86.68)}, rotate = 359.78] [fill={rgb, 255:red, 74; green, 144; blue, 226 }  ,fill opacity=1 ][line width=0.08]  [draw opacity=0] (5.36,-2.57) -- (0,0) -- (5.36,2.57) -- cycle    ;
\draw  [color={rgb, 255:red, 208; green, 2; blue, 27 }  ,draw opacity=1 ][fill={rgb, 255:red, 208; green, 2; blue, 27 }  ,fill opacity=1 ] (390.64,86.76) .. controls (390.64,81.23) and (395.11,76.76) .. (400.64,76.76) .. controls (406.16,76.76) and (410.64,81.23) .. (410.64,86.76) .. controls (410.64,92.28) and (406.16,96.76) .. (400.64,96.76) .. controls (395.11,96.76) and (390.64,92.28) .. (390.64,86.76) -- cycle ;
\draw  [color={rgb, 255:red, 128; green, 128; blue, 128 }  ,draw opacity=1 ] (388.64,74.83) -- (412.39,74.83) -- (412.39,99.08) -- (388.64,99.08) -- cycle ;
\draw [color={rgb, 255:red, 74; green, 144; blue, 226 }  ,draw opacity=1 ]   (363.64,86.55) -- (400.64,86.76) ;
\draw [shift={(360.64,86.53)}, rotate = 0.32] [fill={rgb, 255:red, 74; green, 144; blue, 226 }  ,fill opacity=1 ][line width=0.08]  [draw opacity=0] (5.36,-2.57) -- (0,0) -- (5.36,2.57) -- cycle    ;
\draw [color={rgb, 255:red, 74; green, 144; blue, 226 }  ,draw opacity=1 ]   (398.08,88.62) .. controls (378.47,106.92) and (338.27,105.61) .. (320.67,86.16) ;
\draw [shift={(400.42,86.23)}, rotate = 131.79] [fill={rgb, 255:red, 74; green, 144; blue, 226 }  ,fill opacity=1 ][line width=0.08]  [draw opacity=0] (5.36,-2.57) -- (0,0) -- (5.36,2.57) -- cycle    ;
\draw  [color={rgb, 255:red, 208; green, 2; blue, 27 }  ,draw opacity=1 ][fill={rgb, 255:red, 208; green, 2; blue, 27 }  ,fill opacity=1 ] (430.64,86.76) .. controls (430.64,81.23) and (435.11,76.76) .. (440.64,76.76) .. controls (446.16,76.76) and (450.64,81.23) .. (450.64,86.76) .. controls (450.64,92.28) and (446.16,96.76) .. (440.64,96.76) .. controls (435.11,96.76) and (430.64,92.28) .. (430.64,86.76) -- cycle ;
\draw  [color={rgb, 255:red, 128; green, 128; blue, 128 }  ,draw opacity=1 ] (428.64,74.83) -- (452.39,74.83) -- (452.39,99.08) -- (428.64,99.08) -- cycle ;
\draw [color={rgb, 255:red, 65; green, 117; blue, 5 }  ,draw opacity=1 ]   (400.51,39.95) .. controls (389.38,29.09) and (373.69,28.15) .. (362.54,37.58) ;
\draw [shift={(360.46,39.51)}, rotate = 314.39] [fill={rgb, 255:red, 65; green, 117; blue, 5 }  ,fill opacity=1 ][line width=0.08]  [draw opacity=0] (5.36,-2.57) -- (0,0) -- (5.36,2.57) -- cycle    ;
\draw [color={rgb, 255:red, 65; green, 117; blue, 5 }  ,draw opacity=1 ]   (360.46,39.51) .. controls (371.62,50.69) and (386.8,51.39) .. (398.35,41.9) ;
\draw [shift={(400.51,39.95)}, rotate = 135.68] [fill={rgb, 255:red, 65; green, 117; blue, 5 }  ,fill opacity=1 ][line width=0.08]  [draw opacity=0] (5.36,-2.57) -- (0,0) -- (5.36,2.57) -- cycle    ;

\draw (243.67,30.07) node [anchor=north west][inner sep=0.75pt]  [color={rgb, 255:red, 128; green, 128; blue, 128 }  ,opacity=1 ]  {$\Gamma _{i} \times \Delta _{i}$};
\draw (243.67,77.07) node [anchor=north west][inner sep=0.75pt]  [color={rgb, 255:red, 128; green, 128; blue, 128 }  ,opacity=1 ]  {$\Gamma _{j} \times \Delta _{j}$};
\draw (371.25,10) node [anchor=north west][inner sep=0.75pt]  [color={rgb, 255:red, 65; green, 117; blue, 5 }  ,opacity=1 ]  {$\tau ^{3}$};
\draw (350.59,102.03) node [anchor=north west][inner sep=0.75pt]  [color={rgb, 255:red, 74; green, 144; blue, 226 }  ,opacity=1 ]  {$\tau ^{2}$};

\end{tikzpicture}

}
\newcommand{\figurerhoone}{

\tikzset{every picture/.style={line width=0.75pt}} 

\begin{tikzpicture}[x=0.6pt,y=0.6pt,yscale=-1,xscale=1]

\draw  [color={rgb, 255:red, 208; green, 2; blue, 27 }  ,draw opacity=1 ][fill={rgb, 255:red, 208; green, 2; blue, 27 }  ,fill opacity=1 ] (200,50) .. controls (200,44.48) and (204.48,40) .. (210,40) .. controls (215.52,40) and (220,44.48) .. (220,50) .. controls (220,55.52) and (215.52,60) .. (210,60) .. controls (204.48,60) and (200,55.52) .. (200,50) -- cycle ;
\draw  [color={rgb, 255:red, 208; green, 2; blue, 27 }  ,draw opacity=1 ][fill={rgb, 255:red, 208; green, 2; blue, 27 }  ,fill opacity=1 ] (230,50) .. controls (230,44.48) and (234.48,40) .. (240,40) .. controls (245.52,40) and (250,44.48) .. (250,50) .. controls (250,55.52) and (245.52,60) .. (240,60) .. controls (234.48,60) and (230,55.52) .. (230,50) -- cycle ;
\draw  [color={rgb, 255:red, 208; green, 2; blue, 27 }  ,draw opacity=1 ][fill={rgb, 255:red, 208; green, 2; blue, 27 }  ,fill opacity=1 ] (260,50) .. controls (260,44.48) and (264.48,40) .. (270,40) .. controls (275.52,40) and (280,44.48) .. (280,50) .. controls (280,55.52) and (275.52,60) .. (270,60) .. controls (264.48,60) and (260,55.52) .. (260,50) -- cycle ;
\draw  [color={rgb, 255:red, 208; green, 2; blue, 27 }  ,draw opacity=1 ][fill={rgb, 255:red, 208; green, 2; blue, 27 }  ,fill opacity=1 ] (290,50) .. controls (290,44.48) and (294.48,40) .. (300,40) .. controls (305.52,40) and (310,44.48) .. (310,50) .. controls (310,55.52) and (305.52,60) .. (300,60) .. controls (294.48,60) and (290,55.52) .. (290,50) -- cycle ;
\draw  [color={rgb, 255:red, 248; green, 231; blue, 28 }  ,draw opacity=1 ][fill={rgb, 255:red, 248; green, 231; blue, 28 }  ,fill opacity=1 ] (320,50) .. controls (320,44.48) and (324.48,40) .. (330,40) .. controls (335.52,40) and (340,44.48) .. (340,50) .. controls (340,55.52) and (335.52,60) .. (330,60) .. controls (324.48,60) and (320,55.52) .. (320,50) -- cycle ;
\draw  [color={rgb, 255:red, 208; green, 2; blue, 27 }  ,draw opacity=1 ][fill={rgb, 255:red, 208; green, 2; blue, 27 }  ,fill opacity=1 ] (380,50) .. controls (380,44.48) and (384.48,40) .. (390,40) .. controls (395.52,40) and (400,44.48) .. (400,50) .. controls (400,55.52) and (395.52,60) .. (390,60) .. controls (384.48,60) and (380,55.52) .. (380,50) -- cycle ;
\draw  [color={rgb, 255:red, 208; green, 2; blue, 27 }  ,draw opacity=1 ][fill={rgb, 255:red, 208; green, 2; blue, 27 }  ,fill opacity=1 ] (410,50) .. controls (410,44.48) and (414.48,40) .. (420,40) .. controls (425.52,40) and (430,44.48) .. (430,50) .. controls (430,55.52) and (425.52,60) .. (420,60) .. controls (414.48,60) and (410,55.52) .. (410,50) -- cycle ;
\draw  [color={rgb, 255:red, 208; green, 2; blue, 27 }  ,draw opacity=1 ][fill={rgb, 255:red, 208; green, 2; blue, 27 }  ,fill opacity=1 ] (440,50) .. controls (440,44.48) and (444.48,40) .. (450,40) .. controls (455.52,40) and (460,44.48) .. (460,50) .. controls (460,55.52) and (455.52,60) .. (450,60) .. controls (444.48,60) and (440,55.52) .. (440,50) -- cycle ;
\draw  [color={rgb, 255:red, 208; green, 2; blue, 27 }  ,draw opacity=1 ][fill={rgb, 255:red, 208; green, 2; blue, 27 }  ,fill opacity=1 ] (470,50) .. controls (470,44.48) and (474.48,40) .. (480,40) .. controls (485.52,40) and (490,44.48) .. (490,50) .. controls (490,55.52) and (485.52,60) .. (480,60) .. controls (474.48,60) and (470,55.52) .. (470,50) -- cycle ;
\draw    (210,50) .. controls (218.84,59.75) and (228.53,60.79) .. (237.89,52.12) ;
\draw [shift={(240,50)}, rotate = 132.53] [fill={rgb, 255:red, 0; green, 0; blue, 0 }  ][line width=0.08]  [draw opacity=0] (5.36,-2.57) -- (0,0) -- (5.36,2.57) -- cycle    ;
\draw    (240,50) .. controls (248.84,59.75) and (258.53,60.79) .. (267.89,52.12) ;
\draw [shift={(270,50)}, rotate = 132.53] [fill={rgb, 255:red, 0; green, 0; blue, 0 }  ][line width=0.08]  [draw opacity=0] (5.36,-2.57) -- (0,0) -- (5.36,2.57) -- cycle    ;
\draw    (270,50) .. controls (278.83,59.75) and (288.53,60.79) .. (297.89,52.12) ;
\draw [shift={(300,50)}, rotate = 132.53] [fill={rgb, 255:red, 0; green, 0; blue, 0 }  ][line width=0.08]  [draw opacity=0] (5.36,-2.57) -- (0,0) -- (5.36,2.57) -- cycle    ;
\draw    (420,50) .. controls (428.84,59.75) and (438.53,60.79) .. (447.89,52.12) ;
\draw [shift={(450,50)}, rotate = 132.53] [fill={rgb, 255:red, 0; green, 0; blue, 0 }  ][line width=0.08]  [draw opacity=0] (5.36,-2.57) -- (0,0) -- (5.36,2.57) -- cycle    ;
\draw    (450,50) .. controls (458.84,59.75) and (468.53,60.79) .. (477.89,52.12) ;
\draw [shift={(480,50)}, rotate = 132.53] [fill={rgb, 255:red, 0; green, 0; blue, 0 }  ][line width=0.08]  [draw opacity=0] (5.36,-2.57) -- (0,0) -- (5.36,2.57) -- cycle    ;
\draw    (480,50) .. controls (488.84,59.75) and (498.53,60.79) .. (507.89,52.12) ;
\draw [shift={(510,50)}, rotate = 132.53] [fill={rgb, 255:red, 0; green, 0; blue, 0 }  ][line width=0.08]  [draw opacity=0] (5.36,-2.57) -- (0,0) -- (5.36,2.57) -- cycle    ;
\draw    (180,50) .. controls (188.84,59.75) and (198.53,60.79) .. (207.89,52.12) ;
\draw [shift={(210,50)}, rotate = 132.53] [fill={rgb, 255:red, 0; green, 0; blue, 0 }  ][line width=0.08]  [draw opacity=0] (5.36,-2.57) -- (0,0) -- (5.36,2.57) -- cycle    ;
\draw    (390,50) .. controls (398.84,59.75) and (408.53,60.79) .. (417.89,52.12) ;
\draw [shift={(420,50)}, rotate = 132.53] [fill={rgb, 255:red, 0; green, 0; blue, 0 }  ][line width=0.08]  [draw opacity=0] (5.36,-2.57) -- (0,0) -- (5.36,2.57) -- cycle    ;
\draw    (300,50) .. controls (319.3,68.88) and (366.85,70.16) .. (387.82,52.05) ;
\draw [shift={(390,50)}, rotate = 134.19] [fill={rgb, 255:red, 0; green, 0; blue, 0 }  ][line width=0.08]  [draw opacity=0] (5.36,-2.57) -- (0,0) -- (5.36,2.57) -- cycle    ;
\draw [color={rgb, 255:red, 144; green, 19; blue, 254 }  ,draw opacity=1 ] [dash pattern={on 3pt off 1.5pt}]  (270,50) .. controls (279.3,41.23) and (288.31,40.04) .. (297.84,48.04) ;
\draw [shift={(300,50)}, rotate = 224.35] [fill={rgb, 255:red, 144; green, 19; blue, 254 }  ,fill opacity=1 ][line width=0.08]  [draw opacity=0] (5.36,-2.57) -- (0,0) -- (5.36,2.57) -- cycle    ;
\draw [color={rgb, 255:red, 144; green, 19; blue, 254 }  ,draw opacity=1 ] [dash pattern={on 3pt off 1.5pt}]  (300,50) .. controls (309.3,41.23) and (318.31,40.04) .. (327.84,48.04) ;
\draw [shift={(330,50)}, rotate = 224.35] [fill={rgb, 255:red, 144; green, 19; blue, 254 }  ,fill opacity=1 ][line width=0.08]  [draw opacity=0] (5.36,-2.57) -- (0,0) -- (5.36,2.57) -- cycle    ;
\draw [color={rgb, 255:red, 144; green, 19; blue, 254 }  ,draw opacity=1 ] [dash pattern={on 3pt off 1.5pt}]  (330,50) .. controls (311.02,69.2) and (291.73,69.57) .. (272.13,51.98) ;
\draw [shift={(270,50)}, rotate = 43.9] [fill={rgb, 255:red, 144; green, 19; blue, 254 }  ,fill opacity=1 ][line width=0.08]  [draw opacity=0] (5.36,-2.57) -- (0,0) -- (5.36,2.57) -- cycle    ;
\draw  [color={rgb, 255:red, 0; green, 0; blue, 0 }  ,draw opacity=1 ][fill={rgb, 255:red, 0; green, 0; blue, 0 }  ,fill opacity=1 ] (350,50) .. controls (350,44.48) and (354.48,40) .. (360,40) .. controls (365.52,40) and (370,44.48) .. (370,50) .. controls (370,55.52) and (365.52,60) .. (360,60) .. controls (354.48,60) and (350,55.52) .. (350,50) -- cycle ;

\draw (188,58.4) node [anchor=north west][inner sep=0.75pt]    {$g$};
\draw (277,22.4) node [anchor=north west][inner sep=0.75pt]  [color={rgb, 255:red, 144; green, 19; blue, 254 }  ,opacity=1 ]  {$k$};
\draw (207,72.4) node [anchor=north west][inner sep=0.75pt]  [color={rgb, 255:red, 128; green, 128; blue, 128 }  ,opacity=1 ]  {$0$};
\draw (237,72.4) node [anchor=north west][inner sep=0.75pt]  [color={rgb, 255:red, 128; green, 128; blue, 128 }  ,opacity=1 ]  {$1$};
\draw (267,72.4) node [anchor=north west][inner sep=0.75pt]  [color={rgb, 255:red, 128; green, 128; blue, 128 }  ,opacity=1 ]  {$2$};
\draw (297,72.4) node [anchor=north west][inner sep=0.75pt]  [color={rgb, 255:red, 128; green, 128; blue, 128 }  ,opacity=1 ]  {$3$};
\draw (327,72.4) node [anchor=north west][inner sep=0.75pt]  [color={rgb, 255:red, 128; green, 128; blue, 128 }  ,opacity=1 ]  {$4$};
\draw (357,72.4) node [anchor=north west][inner sep=0.75pt]  [color={rgb, 255:red, 128; green, 128; blue, 128 }  ,opacity=1 ]  {$5$};
\draw (387,72.4) node [anchor=north west][inner sep=0.75pt]  [color={rgb, 255:red, 128; green, 128; blue, 128 }  ,opacity=1 ]  {$6$};
\draw (417,72.4) node [anchor=north west][inner sep=0.75pt]  [color={rgb, 255:red, 128; green, 128; blue, 128 }  ,opacity=1 ]  {$7$};
\draw (447,72.4) node [anchor=north west][inner sep=0.75pt]  [color={rgb, 255:red, 128; green, 128; blue, 128 }  ,opacity=1 ]  {$8$};
\draw (477,72.4) node [anchor=north west][inner sep=0.75pt]  [color={rgb, 255:red, 128; green, 128; blue, 128 }  ,opacity=1 ]  {$9$};
\draw (351,22.4) node [anchor=north west][inner sep=0.75pt]  [color={rgb, 255:red, 0; green, 0; blue, 0 }  ,opacity=1 ]  {$\mathcal{F}$};

\end{tikzpicture}

}
\newcommand{\figurerhotwo}{

\tikzset{every picture/.style={line width=0.75pt}} 

\begin{tikzpicture}[x=0.6pt,y=0.6pt,yscale=-1,xscale=1]

\draw  [color={rgb, 255:red, 208; green, 2; blue, 27 }  ,draw opacity=1 ][fill={rgb, 255:red, 208; green, 2; blue, 27 }  ,fill opacity=1 ] (200,50) .. controls (200,44.48) and (204.48,40) .. (210,40) .. controls (215.52,40) and (220,44.48) .. (220,50) .. controls (220,55.52) and (215.52,60) .. (210,60) .. controls (204.48,60) and (200,55.52) .. (200,50) -- cycle ;
\draw  [color={rgb, 255:red, 65; green, 117; blue, 5 }  ,draw opacity=1 ][fill={rgb, 255:red, 65; green, 117; blue, 5 }  ,fill opacity=1 ] (230,50) .. controls (230,44.48) and (234.48,40) .. (240,40) .. controls (245.52,40) and (250,44.48) .. (250,50) .. controls (250,55.52) and (245.52,60) .. (240,60) .. controls (234.48,60) and (230,55.52) .. (230,50) -- cycle ;
\draw  [color={rgb, 255:red, 65; green, 117; blue, 5 }  ,draw opacity=1 ][fill={rgb, 255:red, 65; green, 117; blue, 5 }  ,fill opacity=1 ] (260,50) .. controls (260,44.48) and (264.48,40) .. (270,40) .. controls (275.52,40) and (280,44.48) .. (280,50) .. controls (280,55.52) and (275.52,60) .. (270,60) .. controls (264.48,60) and (260,55.52) .. (260,50) -- cycle ;
\draw  [color={rgb, 255:red, 208; green, 2; blue, 27 }  ,draw opacity=1 ][fill={rgb, 255:red, 208; green, 2; blue, 27 }  ,fill opacity=1 ] (290,50) .. controls (290,44.48) and (294.48,40) .. (300,40) .. controls (305.52,40) and (310,44.48) .. (310,50) .. controls (310,55.52) and (305.52,60) .. (300,60) .. controls (294.48,60) and (290,55.52) .. (290,50) -- cycle ;
\draw  [color={rgb, 255:red, 65; green, 117; blue, 5 }  ,draw opacity=1 ][fill={rgb, 255:red, 65; green, 117; blue, 5 }  ,fill opacity=1 ] (320,50) .. controls (320,44.48) and (324.48,40) .. (330,40) .. controls (335.52,40) and (340,44.48) .. (340,50) .. controls (340,55.52) and (335.52,60) .. (330,60) .. controls (324.48,60) and (320,55.52) .. (320,50) -- cycle ;
\draw  [color={rgb, 255:red, 65; green, 117; blue, 5 }  ,draw opacity=1 ][fill={rgb, 255:red, 65; green, 117; blue, 5 }  ,fill opacity=1 ] (380,50) .. controls (380,44.48) and (384.48,40) .. (390,40) .. controls (395.52,40) and (400,44.48) .. (400,50) .. controls (400,55.52) and (395.52,60) .. (390,60) .. controls (384.48,60) and (380,55.52) .. (380,50) -- cycle ;
\draw  [color={rgb, 255:red, 208; green, 2; blue, 27 }  ,draw opacity=1 ][fill={rgb, 255:red, 208; green, 2; blue, 27 }  ,fill opacity=1 ] (410,50) .. controls (410,44.48) and (414.48,40) .. (420,40) .. controls (425.52,40) and (430,44.48) .. (430,50) .. controls (430,55.52) and (425.52,60) .. (420,60) .. controls (414.48,60) and (410,55.52) .. (410,50) -- cycle ;
\draw  [color={rgb, 255:red, 65; green, 117; blue, 5 }  ,draw opacity=1 ][fill={rgb, 255:red, 65; green, 117; blue, 5 }  ,fill opacity=1 ] (440,50) .. controls (440,44.48) and (444.48,40) .. (450,40) .. controls (455.52,40) and (460,44.48) .. (460,50) .. controls (460,55.52) and (455.52,60) .. (450,60) .. controls (444.48,60) and (440,55.52) .. (440,50) -- cycle ;
\draw  [color={rgb, 255:red, 208; green, 2; blue, 27 }  ,draw opacity=1 ][fill={rgb, 255:red, 208; green, 2; blue, 27 }  ,fill opacity=1 ] (470,50) .. controls (470,44.48) and (474.48,40) .. (480,40) .. controls (485.52,40) and (490,44.48) .. (490,50) .. controls (490,55.52) and (485.52,60) .. (480,60) .. controls (474.48,60) and (470,55.52) .. (470,50) -- cycle ;
\draw [color={rgb, 255:red, 0; green, 0; blue, 0 }  ,draw opacity=1 ]   (180,50) .. controls (199.3,68.88) and (218.91,70.16) .. (237.93,52.05) ;
\draw [shift={(240,50)}, rotate = 134.19] [fill={rgb, 255:red, 0; green, 0; blue, 0 }  ,fill opacity=1 ][line width=0.08]  [draw opacity=0] (5.36,-2.57) -- (0,0) -- (5.36,2.57) -- cycle    ;
\draw [color={rgb, 255:red, 0; green, 0; blue, 0 }  ,draw opacity=1 ]   (210,50) .. controls (239.4,20.5) and (268.8,19.91) .. (298.2,48.23) ;
\draw [shift={(300,50)}, rotate = 225.1] [fill={rgb, 255:red, 0; green, 0; blue, 0 }  ,fill opacity=1 ][line width=0.08]  [draw opacity=0] (5.36,-2.57) -- (0,0) -- (5.36,2.57) -- cycle    ;
\draw [color={rgb, 255:red, 0; green, 0; blue, 0 }  ,draw opacity=1 ]   (240,50) .. controls (269.4,79.63) and (299.12,80.54) .. (328.22,51.8) ;
\draw [shift={(330,50)}, rotate = 134.14] [fill={rgb, 255:red, 0; green, 0; blue, 0 }  ,fill opacity=1 ][line width=0.08]  [draw opacity=0] (5.36,-2.57) -- (0,0) -- (5.36,2.57) -- cycle    ;
\draw [color={rgb, 255:red, 0; green, 0; blue, 0 }  ,draw opacity=1 ]   (270,50) .. controls (299.4,79.63) and (357.93,80.54) .. (388.18,51.8) ;
\draw [shift={(390,50)}, rotate = 134.14] [fill={rgb, 255:red, 0; green, 0; blue, 0 }  ,fill opacity=1 ][line width=0.08]  [draw opacity=0] (5.36,-2.57) -- (0,0) -- (5.36,2.57) -- cycle    ;
\draw [color={rgb, 255:red, 0; green, 0; blue, 0 }  ,draw opacity=1 ]   (300,50) .. controls (329.73,20.18) and (387.95,20.54) .. (418.18,48.27) ;
\draw [shift={(420,50)}, rotate = 224.77] [fill={rgb, 255:red, 0; green, 0; blue, 0 }  ,fill opacity=1 ][line width=0.08]  [draw opacity=0] (5.36,-2.57) -- (0,0) -- (5.36,2.57) -- cycle    ;
\draw [color={rgb, 255:red, 0; green, 0; blue, 0 }  ,draw opacity=1 ]   (270,50) .. controls (290,60.57) and (310,60.23) .. (330,50) ;
\draw [shift={(296.1,57.68)}, rotate = 0.66] [fill={rgb, 255:red, 0; green, 0; blue, 0 }  ,fill opacity=1 ][line width=0.08]  [draw opacity=0] (5.36,-2.57) -- (0,0) -- (5.36,2.57) -- cycle    ;
\draw [color={rgb, 255:red, 0; green, 0; blue, 0 }  ,draw opacity=1 ]   (390,50) .. controls (409.3,68.88) and (428.91,70.16) .. (447.93,52.05) ;
\draw [shift={(450,50)}, rotate = 134.19] [fill={rgb, 255:red, 0; green, 0; blue, 0 }  ,fill opacity=1 ][line width=0.08]  [draw opacity=0] (5.36,-2.57) -- (0,0) -- (5.36,2.57) -- cycle    ;
\draw [color={rgb, 255:red, 0; green, 0; blue, 0 }  ,draw opacity=1 ]   (450,50) .. controls (469.3,68.88) and (488.91,70.16) .. (507.93,52.05) ;
\draw [shift={(510,50)}, rotate = 134.19] [fill={rgb, 255:red, 0; green, 0; blue, 0 }  ,fill opacity=1 ][line width=0.08]  [draw opacity=0] (5.36,-2.57) -- (0,0) -- (5.36,2.57) -- cycle    ;
\draw [color={rgb, 255:red, 0; green, 0; blue, 0 }  ,draw opacity=1 ]   (420,50) .. controls (439.62,30.6) and (458.93,30.24) .. (477.93,48) ;
\draw [shift={(480,50)}, rotate = 225.15] [fill={rgb, 255:red, 0; green, 0; blue, 0 }  ,fill opacity=1 ][line width=0.08]  [draw opacity=0] (5.36,-2.57) -- (0,0) -- (5.36,2.57) -- cycle    ;
\draw  [color={rgb, 255:red, 0; green, 0; blue, 0 }  ,draw opacity=1 ][fill={rgb, 255:red, 0; green, 0; blue, 0 }  ,fill opacity=1 ] (350,50) .. controls (350,44.48) and (354.48,40) .. (360,40) .. controls (365.52,40) and (370,44.48) .. (370,50) .. controls (370,55.52) and (365.52,60) .. (360,60) .. controls (354.48,60) and (350,55.52) .. (350,50) -- cycle ;
\draw [color={rgb, 255:red, 0; green, 0; blue, 0 }  ,draw opacity=1 ]   (180,30) .. controls (193.16,35.78) and (197.48,38.34) .. (207.92,48.05) ;
\draw [shift={(210,50)}, rotate = 223.41] [fill={rgb, 255:red, 0; green, 0; blue, 0 }  ,fill opacity=1 ][line width=0.08]  [draw opacity=0] (5.36,-2.57) -- (0,0) -- (5.36,2.57) -- cycle    ;
\draw [color={rgb, 255:red, 0; green, 0; blue, 0 }  ,draw opacity=1 ]   (480,50) .. controls (489.87,40.62) and (493.78,37.41) .. (507.29,31.22) ;
\draw [shift={(510,30)}, rotate = 155.96] [fill={rgb, 255:red, 0; green, 0; blue, 0 }  ,fill opacity=1 ][line width=0.08]  [draw opacity=0] (5.36,-2.57) -- (0,0) -- (5.36,2.57) -- cycle    ;

\draw (179,63.4) node [anchor=north west][inner sep=0.75pt]    {$kg^{2}$};
\draw (207,72.4) node [anchor=north west][inner sep=0.75pt]  [color={rgb, 255:red, 128; green, 128; blue, 128 }  ,opacity=1 ]  {$0$};
\draw (237,72.4) node [anchor=north west][inner sep=0.75pt]  [color={rgb, 255:red, 128; green, 128; blue, 128 }  ,opacity=1 ]  {$1$};
\draw (267,72.4) node [anchor=north west][inner sep=0.75pt]  [color={rgb, 255:red, 128; green, 128; blue, 128 }  ,opacity=1 ]  {$2$};
\draw (297,72.4) node [anchor=north west][inner sep=0.75pt]  [color={rgb, 255:red, 128; green, 128; blue, 128 }  ,opacity=1 ]  {$3$};
\draw (327,72.4) node [anchor=north west][inner sep=0.75pt]  [color={rgb, 255:red, 128; green, 128; blue, 128 }  ,opacity=1 ]  {$4$};
\draw (357,72.4) node [anchor=north west][inner sep=0.75pt]  [color={rgb, 255:red, 128; green, 128; blue, 128 }  ,opacity=1 ]  {$5$};
\draw (387,72.4) node [anchor=north west][inner sep=0.75pt]  [color={rgb, 255:red, 128; green, 128; blue, 128 }  ,opacity=1 ]  {$6$};
\draw (417,72.4) node [anchor=north west][inner sep=0.75pt]  [color={rgb, 255:red, 128; green, 128; blue, 128 }  ,opacity=1 ]  {$7$};
\draw (447,72.4) node [anchor=north west][inner sep=0.75pt]  [color={rgb, 255:red, 128; green, 128; blue, 128 }  ,opacity=1 ]  {$8$};
\draw (351,22.4) node [anchor=north west][inner sep=0.75pt]  [color={rgb, 255:red, 0; green, 0; blue, 0 }  ,opacity=1 ]  {$\mathcal{F}$};
\draw (477,72.4) node [anchor=north west][inner sep=0.75pt]  [color={rgb, 255:red, 128; green, 128; blue, 128 }  ,opacity=1 ]  {$9$};

\end{tikzpicture}

}

\title{Subgroup and Submonoid Membership in the lampshuffler of $\mathbb{Z}$}
\author{Corentin Bodart\footnote{Département de Mathématiques, Universit\'{e} du Luxembourg, Luxembourg. Email: \url{cobodart123@gmail.com}. C.B. was funded by the Swiss SNF grant P500PT-225420.} \and Ruiwen Dong\footnote{Magdalen College, University of Oxford, United Kingdom. Email: \url{ruiwen.dong@magd.ox.ac.uk}. R.D.\ is supported by a Fellowship by Examination at Magdalen College.}}
\date{}

\begin{document}

\maketitle

\begin{abstract}    
    The \textit{lampshuffler group} of $\mathbb{Z}$ is the semidirect product $\operatorname{FSym}(\mathbb{Z}) \rtimes \mathbb{Z}$, which consists of all permutations of $\mathbb{Z}$ that act as a translation outside a finite set. This infinite permutation group naturally contains as subgroups the wreath products $H \wr \mathbb{Z}$ for every finite group $H$. We prove that the Subgroup Membership Problem, and more generally, the Submonoid Membership Problem, are decidable in $\operatorname{FSym}(\mathbb{Z}) \rtimes \mathbb{Z}$.
    Our proof reduces Subgroup and Submonoid Membership in $\operatorname{FSym}(\mathbb{Z}) \rtimes \mathbb{Z}$ to Subgroup Membership in the wreath products $H \wr \mathbb{Z}$, which was shown to be decidable by Lohrey, Steinberg and Zetzsche (2015). 
\end{abstract}

\section{Introduction and main results}
\paragraph*{Membership problems in infinite groups.}
The study of algorithmic problems in groups and semigroups has a long history dating back to the beginning of the 20th century.
Two of the central decision problems in this area are the \emph{Subgroup Membership} problem and the \emph{Submonoid Membership} problem.
For both problems, the input is a finite number of elements $g_1, g_2, \ldots, g_n$ in a (possibly infinite) group $G$ and an element $g \in G$.
The Subgroup Membership problem asks whether $g$ belong to the sub\emph{group} generated by $g_1, g_2, \ldots, g_n$, and the more general Submonoid Membership problem asks whether $g$ belong to the sub\emph{monoid} generated by $g_1, g_2, \ldots, g_n$.

In the seminal work of Markov from the 1950s~\cite{markov1951certain}, it was shown that Submonoid Membership is undecidable for a matrix group of dimension six.
Historically, this provided one of the earliest examples of an undecidable problem in the theory of computing.
Markov's result was subsequently strengthened by Mikhailova~\cite{mikhailova1958occurrence}, who showed undecidability of Subgroup Membership in the group $\SL(4, \Z)$ of $4 \times 4$ integer matrices of determinant one.
Since then, significant advances in this area gave a large number of both positive and negative decidability results.
Examples include matrix groups~\cite{kopytov1968solvability, babai1996multiplicative, choffrut2005some, potapov2017decidability, colcombet2019reachability, bodart2025membership}, groups arising from the study of automata and geometric group theory~\cite{GrigorchukWilson2003, lohrey2011tilings, cadilhac2020rational, Gray2020OneRelatorInverseMonoids, foniqi2026membership}, as well as groups constructed from finite extensions, direct products, and wreath products~\cite{romanovskii1974some, lohrey2015rational, roman2022undecidability, shafrir2024decidability, dong2025submonoid, dong2026standard}.
These results have had far-reaching consequences in areas such as weighted automata, program verification, and control theory~\cite{blondel2000survey, derksen2005quantum, hrushovski2018polynomial, haase2019presburger}.
See~\cite{lohrey2024membership} for a comprehensive survey.

\paragraph*{Infinite permutation groups.}
Membership problems in \emph{finite} permutation groups are a staple of algorithmic and complexity theory~\cite{sims1971computation, knuth1991efficient, lohrey2023complexity, dahan2025group}.
In a finite permutation group, Submonoid Membership is equivalent to Subgroup Membership, and they can be efficiently decided using the celebrated \emph{Schreier-Sims algorithm}~\cite{sims1970computational}.
In contrast, \emph{infinite} permutation groups are deeply related to logic and computability theory, with applications in orbit-finite programming~\cite{ghosh2023orbit, bojanczyk2024orbit}, model theory~\cite{cameron1990oligomorphic, ackerman2016invariant}, as well as automata and formal language theory~\cite{bojanczyk2011automata, lohrey2015rational}.
Consider the set of integers $\Z$.
The simplest permutation group over $\Z$ is the finitary symmetric group $\FSym(\Z)$, consisting of all \emph{finitely supported} permutations of $\Z$.
It is easy to see that the Subgroup and Submonoid Membership problems are decidable in $\FSym(\Z)$ since they are equivalent to the case of finite groups.
Thus, a natural next step is to study extensions of $\FSym(\Z)$ with simple, infinitely supported elements.

In this paper we consider Subgroup and Submonoid Membership in the \emph{lampshuffler group}\footnote{This terminology appeared in~\cite{gheysens2022between, bonnet2022twin, genevois2024lamplighter, silva2024poisson}, where the semidirect product $\FSym(G) \rtimes G$ is referred to as the \textit{lampshuffler} of the group $G$. The name originated from their resemblance to lamplighter groups.} $\FSym(\Z) \rtimes \Z$, which is the permutation group over $\Z$ generated by the elements of $\FSym(\Z)$ and the \emph{shift} $\sigma \colon \Z \rightarrow \Z\colon x \mapsto x+1$.
The group $\FSym(\Z) \rtimes \Z$ consists of all permutations of $\Z$ that act \emph{ultimately} as a translation.
In other words, elements of $\FSym(\Z) \rtimes \Z$ are permutations $f \colon \Z \rightarrow \Z$, for which $f(z)  - z$ is constant for all but finitely many $z$.
Intuitively, one may regard computation in $\FSym(\Z) \rtimes \Z$ as operations on a bi-infinite tape by finite permutation of tape cells and shifting of the entire tape.

The groups $H_1 = \FSym(\N)$ and $H_2 = \FSym(\Z) \rtimes \Z$ are the first two entries in a series of infinite groups $H_n, n \geq 1$, called the \emph{Houghton groups}~\cite{houghton1978first}.
These groups have been extensively studied for their finiteness properties~\cite{brown1987finiteness, cox2025finiteness}, growth \cite{metric_in_Houghton} and combinatorial features~\cite{wiegold1977transitive}.
In particular, the \emph{Conjugacy Problem} was shown to be decidable in all Houghton groups~\cite{conjugacy_in_Houghton}, while decidability of Subgroup and Submonoid Membership remains open.

For every finite group $H$, the lampshuffler $\FSym(\Z) \rtimes \Z$ contains the wreath product $H \wr \Z$ as a subgroup (see Section~\ref{sec:prelim}).
The groups $H \wr \Z$ are often referred to as \emph{lamplighter} groups.
Lohrey, Steinberg and Zetzsche~\cite{lohrey2015rational} showed that \emph{Rational Subset Membership} (which subsumes Subgroup and Submonoid Membership) is decidable in $H \wr \Z$ for every finite group $H$.
It is unclear whether this decidability result generalizes to the lampshuffler $\FSym(\Z) \rtimes \Z$, which can be considered as a ``limit'' of the lamplighters $H \wr \Z$.

Another motivation to study membership problems in $\FSym(\Z) \rtimes \Z$ is to understand how decidability behaves with respect to taking various extensions of groups.
It is a classic result that, for two groups $H \leq G$ where the quotient $G/H$ is finite, \emph{Subgroup Membership} is decidable in $G$ if and only if it is decidable in $H$~\cite{grunschlag1999algorithms}.
The same statement is however false for \emph{Submonoid Membership}, as shown by the recent work of Shafrir~\cite{shafrir2024decidability} and Dong~\cite{dong2025submonoid}.
The natural next step is thus to consider the case with $H \unlhd G$ such that $G/H \cong \Z$, and find sufficient conditions on $H$ that render Subgroup and Submonoid Membership in $G$ decidable.

\paragraph{Main result.}

The main result of this paper is the decidability of Submonoid (and Subgroup) Membership in $\FSym(\Z) \rtimes \Z$:

\begin{restatable}{theorem}{thmmain}\label{thm:submonoid_membership}
    Submonoid Membership is decidable in $\FSym(\Z) \rtimes \Z$.
\end{restatable}

In fact, it is not difficult to reduce Submonoid Membership in $\FSym(\Z) \rtimes \Z$ to Subgroup Membership (see Lemma~\ref{lem:submonoid_to_subgroup}).
Then, to show decidability of Subgroup Membership, our proof of Theorem~\ref{thm:submonoid_membership} implicitly provides a classification and description of all finitely generated subgroups of $\FSym(\Z) \rtimes \Z$ (see Subsection~\ref{subsec:invar} and Proposition~\ref{prop:structure_invariant_subgroup}).
However, we do not know the complexity of our algorithm for Theorem~\ref{thm:submonoid_membership}, as it relies on a reduction to the Subgroup Membership problem in $H \wr \Z$~\cite{lohrey2015rational}, whose exact complexity remains unknown.

\section{Preliminaries}\label{sec:prelim}

\paragraph{Permutation groups.}

Let $\Omega$ be a (finite or infinite) set.
The \emph{symmetric group} $\Sym(\Omega)$ on the set $\Omega$ is the group whose elements are bijections from $\Omega$ to $\Omega$, and multiplication is defined by composition of maps, i.e., for $f, g \in \Sym(\Omega)$, the element $fg$ is the bijection that sends $x \in \Omega$ to $f(g(x))$.
The neutral element of $\Sym(\Omega)$ is the identity map, denoted by $\id$.
The cardinality of a set or a group $S$ is denoted by $\card{S}$.

For $f \in \Sym(\Omega)$, we denote by $\supp(f)$ the \emph{support} of $f$, defined as 
\[
\supp(f) \coloneqq \{x \in \Omega \mid f(x) \neq x\}.
\]
This is the set of elements moved by $f$.
The group elements of $\Sym(\Omega)$ with finite support form a subgroup $\FSym(\Omega) \coloneqq \{f \in \Sym(\Omega), \card{\supp(f)} < \infty\}$.
If $\Omega$ is finite then $\FSym(\Omega) = \Sym(\Omega)$.

For distinct elements $x_1, x_2, \ldots, x_k \in \Omega$, let $(x_1 x_2 \cdots x_k)$ denote the cyclic permutation $x_1 \mapsto x_2, x_2 \mapsto x_3, \ldots, x_k \mapsto x_1$.
Such a permutation is called a \emph{$k$-cycle}.
Every $f \in \FSym(\Omega)$ can be decomposed as a product of disjoint cycles $$(x_{11} \cdots x_{1k_1})(x_{21} \cdots x_{2k_2})\cdots(x_{n1} \cdots x_{nk_n}).$$
The \emph{sign} of a permutation $f$, denoted by $\sgn(f)$, is defined as the value $(-1)^m$ where $m$ is the number of cycles of even length in the decomposition of $f$.
The set of elements $f \in \FSym(\Omega)$ with $\sgn(f) = 1$ forms the \emph{alternating subgroup} $\FAlt(\Omega) \unlhd \FSym(\Omega)$.
For an infinite set $\Omega$, the only proper normal subgroups of $\FSym(\Omega)$ are $\FAlt(\Omega)$ and the trivial subgroup $\{\id\}$~\cite[Section~11.3.3]{scott1987group}.

For $f, \rho \in \Sym(\Omega)$, the element $\rho f \rho^{-1}$ is called the \emph{conjugation} of $f$ by $\rho$.
The element $\rho f \rho^{-1}$ maps $\rho(x)$ to $\rho(f(x))$ for every $x \in \Omega$, so $\supp(\rho f \rho^{-1}) = \rho(\supp(f))$.
For two elements $g, h$ of a group, let $[g, h]$ denote the \emph{commutator} $ghg^{-1}h^{-1}$.
For a group $G$, we denote by $G'$ its \emph{commutator subgroup}: this is the group generated by all elements $[g, h]$ with $g, h \in G$.

For $\Omega = \Z$, let $\sigma \in \Sym(\Z)$ denote the shift map $x \mapsto x+1$.
The group $\FSym(\Z) \rtimes \Z$ is the subgroup of $\Sym(\Z)$ generated by $\sigma$ and all elements of $\FSym(\Z)$.
For any element $g \in \FSym(\Z) \rtimes \Z$, since $g$ ultimately acts as a translation $\sigma^z$ for some $z \in \Z$, we have $g = f\sigma^z$ for some $f \in \FSym(\Z)$.
Let
\[
\pi \colon \FSym(\Z) \rtimes \Z \rightarrow \Z, \;  f\sigma^z \mapsto z,
\]
denote the projection, then $\pi$ is a group homomorphism. 

\paragraph{Invariance, orbits, transitivity, and primitivity.}
Given a subset $S$ of a group, denote by $\gen{S}$ the subgroup generated by $S$.

For an element $\rho \in \Sym(\Omega)$ and a subgroup $G \leq \Sym(\Omega)$, we say that $G$ is \emph{$\rho$-invariant} if the group $\rho G \rho^{-1} \coloneqq \{\rho g \rho^{-1} \mid g \in G\}$ is equal to $G$, i.e., $G$ is invariant under the conjugation by $\rho$.
Given a set of elements $S \subset \Sym(\Omega)$, there is a smallest $\rho$-invariant group containing $S$, which we call the $\rho$-invariant group \emph{generated by $S$}, and is denoted by $\gen{S}_{\rho}$.
In particular, $\gen{S}_{\rho}$ is the group generated by the elements $\rho^n s \rho^{-n}$ (for $n \in \Z, s \in S$).

For $G \leq \Sym(\Omega)$, the \emph{$G$-orbit} of an element $x \in \Omega$ is the set $\{g(x) \mid g \in G\} \subseteq \Omega$.
The $G$-orbits of two different elements are either equal or disjoint.
Thus, the set $\Omega$ can be uniquely decomposed as a disjoint union of $G$-orbits.
We say that $G$ is \emph{transitive} (on $\Omega$) if there is only one $G$-orbit.
Equivalently, $G \leq \Sym(\Omega)$ is transitive if 
for all $x, y \in \Omega$, there exists $g \in G$ such that $g(x) = y$.

We say a subset $S \subseteq \Omega$ is \emph{stabilized} by $G$ if, for every $g \in G$, the set $gS \coloneqq \{gs \mid s \in S\}$ is equal to $S$.
In this case, each element of $G$ permutes the set $S$, so $G$ induces a subgroup of $\Sym(S)$.
In particular, $S$ is a union of $G$-orbits.

A \emph{partition} of $\Omega$ is a family of disjoint sets $\Gamma_{\delta} \subseteq \Omega$ (for $\delta \in \Delta$), such that $\bigsqcup_{\delta \in \Delta} \Gamma_{\delta} = \Omega$.
The subsets $\Gamma_{\delta}$ are called \emph{blocks} of the partition.
We say that $g \in \Sym(\Omega)$ \emph{preserves} the partition $\bigsqcup_{\delta \in \Delta} \Gamma_{\delta}$ if for every block $\Gamma_{\delta}$, the set $g \Gamma_{\delta}$ is also a block in this partition.
We say that a group $G \leq \Sym(\Omega)$ preserves a partition $\bigsqcup_{\delta \in \Delta} \Gamma_{\delta}$ if every element of $G$ preserves the partition.
In this case, the group $G$ induces a permutation group on the index set $\Delta$ of blocks.
The trivial partitions of $\Omega$ are the partition with a single block $\Omega$ and the partition into singletons.
A group $G \leq \Sym(\Omega)$ is called \emph{primitive} if it is transitive and does not preserve any non-trivial partition.
For example, the group $\gen{f} \leq \Sym(\{1,2,3,4,5,6\})$ generated by the cycle $f = (123456)$ is transitive but not primitive, since it preserves the partition $\{1, 3, 5\} \sqcup \{2, 4, 6\}$.
However, the group $\FAlt(\{1,2,3,4,5,6\}) \leq \Sym(\{1,2,3,4,5,6\})$ is primitive.

\paragraph{Wreath products and subdirect products.}

An important way to construct non-primitive groups is the \emph{wreath product}.
Let $H$ be an arbitrary group and let $\Delta$ be a set.
The \emph{direct sum} $\bigoplus_{\Delta}H$ is the group whose elements are tuples $(h_{\delta})_{\delta \in \Delta}$ with $h_{\delta} \in H$ for all $\delta \in \Delta$, such that all but finitely many $h_{\delta}$'s are the neutral element.
The set $\bigoplus_{\Delta}H$ forms a group under the pointwise multiplication $(h_{\delta})_{\delta \in \Delta} \cdot (h'_{\delta})_{\delta \in \Delta} = (h_{\delta}h'_{\delta})_{\delta \in \Delta}$.

Let $G$ be a group acting on $\Delta$, i.e., $G$ maps to a subgroup of $\Sym(\Delta)$.
The (restricted) \emph{wreath product} $H \wr_{\Delta} G$ is the group constructed as follows.
Its elements are pairs $\left((h_{\delta})_{\delta \in \Delta}, g\right)$ where $(h_{\delta})_{\delta \in \Delta} \in \bigoplus_{\Delta}H$ and $g \in G$; multiplication is defined by
\begin{equation}\label{eq:wreath_multiplication}
\bigl((h_{\delta})_{\delta \in \Delta},\, g\bigr) \cdot \bigl((h'_{\delta})_{\delta \in \Delta},\, g'\bigr) \coloneqq \left((h_{\delta} h'_{g^{-1}(\delta)})_{\delta \in \Delta},\, gg'\right);
\end{equation}
and the neutral element is $((\id)_{\delta \in \Delta}, \id)$.

For example, let $\Gamma$ be a finite set and $\Delta$ be any set, and take $\Omega$ to be the direct product $\Gamma \times \Delta$.
Then we can consider both $\bigoplus_{\Delta}\Sym(\Gamma)$ and $\Sym(\Gamma) \wr_{\Delta} \FSym(\Delta)$ as subgroups of $\FSym(\Omega)$: the direct sum $\bigoplus_{\Delta}\Sym(\Gamma)$ contains permutations that stabilize every set $\Gamma \times \{\delta\}$ (for $\delta \in \Delta$); while the wreath product $\Sym(\Gamma) \wr_{\Delta} \FSym(\Delta)$ contains permutations that preserve the partition $\Omega = \bigsqcup_{\delta \in \Delta} (\Gamma \times \{\delta\})$.
Intuitively, in Equation~\eqref{eq:wreath_multiplication}, the element $(h_{\delta})_{\delta \in \Delta}$ permutes the elements within each block, whereas $g$ permutes the blocks as a whole.
More examples will be given in the next section (Examples~\ref{example:invariant},~\ref{example:product}).

Note that the direct sum $\bigoplus_{\Delta}H$ can be identified as the subgroup $\left\{\left(h, \id\right) \;\middle|\; h \in \bigoplus_{\Delta}H \right\}$ of $H \wr_{\Delta} G$.
Similarly, the group $G$ can be identified as the subgroup $\left\{\left(\id, g\right) \;\middle|\; g \in G \right\}$ of $H \wr_{\Delta} G$.
Under these identifications, the element $\left((h_{\delta})_{\delta \in \Delta}, g\right) \in H \wr_{\Delta} G$ is simply a product $(h_{\delta})_{\delta \in \Delta} \cdot g$, where $(h_{\delta})_{\delta \in \Delta} \in \bigoplus_{\Delta}H$ and $g \in G$.

Consider $\Z$ as a subgroup of $\Sym(\Z)$ by identifying $n \in \Z$ with the shift $\sigma^n \colon z \mapsto z+n$.
Then for any group $H$ we can define the wreath product $H \wr_{\Z} \Z$, which we simply denote by $H \wr \Z$.
Is it easy to see that when $H$ is finite, the wreath product $H \wr \Z$ can be embedded as a subgroup of $\FSym(\Z) \rtimes \Z$.
Indeed, take any $n \geq 1$ such that $H$ can be embedded as a subgroup of $\Sym(\{0, 1, \ldots, n-1\})$, then $H \wr \Z$ can be embedded in $\Sym(\{0, 1, \ldots, n-1\}) \wr \Z$, which is the subgroup of $\FSym(\Z) \rtimes \Z$ generated by the element $\sigma^n$ and $\Sym(\{0, 1, \ldots, n-1\})$.

The wreath product $H \wr \Z$ will play an important role in our paper for the classification of finitely generated subgroups of $\FSym(\Z) \rtimes \Z$.
A special case of a result of Lohrey, Steinberg, and Zetzsche~\cite{lohrey2015rational} showed that Subgroup Membership is decidable in $H \wr \Z$:

\begin{lemma}[{Special case of~\cite[Theorem~7]{lohrey2015rational}}]\label{lem:decidable_H_wreath_Z}
    Subgroup Membership is decidable in the group $H \wr \Z$ for finite $H$.
\end{lemma}

In this paper we will also frequently deal with subgroups of direct products $G_1 \times G_2$.
This can be characterized by the following classic lemma of Goursat:

\begin{lemma}[{Goursat's lemma, c.f.~\cite[Section~4.3.1]{scott1987group}}]\label{lem:Goursat}
    Let $G$ be a subgroup of $G_1 \times G_2$, such that the projections $p_1 \colon G \rightarrow G_1, p_2 \colon G \rightarrow G_2$ are surjective.
    Denote
    \[
    N_1 \coloneqq p_1(\ker p_2) = \{g_1 \mid (g_1, \id) \in G \} \unlhd G_1, \;\; \text{ and } \;\; N_2 \coloneqq p_2(\ker p_1) = \{g_2 \mid (\id, g_2) \in G \} \unlhd G_2.
    \]
    Then there is an isomorphism $\varphi \colon G_1/N_1 \xrightarrow{\sim} G_2/N_2$, and $G = \{(g_1, g_2) \mid \varphi(g_1 N_1) = g_2 N_2\}$.
\end{lemma}

The following example illustrates Goursat's lemma.
\begin{example}
    Let $G$ be a subgroup of $\Z/4\Z \times \Z/6\Z$, whose projections on $\Z/4\Z$ and $\Z/6\Z$ are surjective.
    We use Goursat's lemma to find all such $G$.
    Since $N_1 \unlhd \Z/4\Z$, we have $(\Z/4\Z)/N_1 = \Z/4\Z, \Z/2\Z$, or $\{\id\}$.
    Similarly for $N_2 \unlhd \Z/6\Z$, we have $(\Z/6\Z)/N_2 = \Z/6\Z, \Z/3\Z, \Z/2\Z$, or $\{\id\}$.
    By Goursat's Lemma. we have $(\Z/4\Z)/N_1 \cong (\Z/6\Z)/N_2$.
    Therefore $(\Z/4\Z)/N_1$ and $(\Z/6\Z)/N_2$ can only be $\Z/2\Z$ or $\{\id\}$.
    \begin{enumerate}[nosep, label = (\roman*)]
        \item If $(\Z/4\Z)/N_1 \cong (\Z/6\Z)/N_2 \cong \Z/2\Z$. Then $N_1 = 2(\Z/4\Z)$ and $N_2 = 2(\Z/6\Z)$.
        In this case, $G = \{(g_1, g_2) \mid (g_1 \bmod 2) = (g_2 \bmod 2)\}$. That is, elements of $G$ are pairs $(a,b)$ of the same parity.
        \item If $(\Z/4\Z)/N_1 \cong (\Z/6\Z)/N_2 \cong \{\id\}$. Then $N_1 = \Z/4\Z$ and $N_2 = \Z/6\Z$. In this case, $G$ is the full group $\Z/4\Z \times \Z/6\Z$.
    \end{enumerate}
\end{example}

\paragraph{From Submonoid Membership to Subgroup Membership.}

In~\cite{bodart2025}, Bodart showed that for $G = N \rtimes \Z$ with $N$ locally finite, Submonoid Membership in $G$ is Turing equivalent to Subgroup Membership.
Specializing to the lampshuffler $\FSym(\Z) \rtimes \Z$, we immediately obtain:

\begin{lemma}[{\cite[Proposition 2.5]{bodart2025}}]\label{lem:submonoid_to_subgroup}
    In $\FSym(\Z) \rtimes \Z$, Submonoid Membership reduces to Subgroup Membership.
\end{lemma}
\begin{proof}
    Let $\sigma$ be the shift map $x \mapsto x+1$ and $\pi \colon \FSym(\Z) \rtimes \Z \rightarrow \Z$ be the projection.
    Let $g_1, \ldots, g_n \in \FSym(\Z) \rtimes \Z$ and let $M$ be the submonoid they generate.
    \begin{itemize}[leftmargin=5mm]
        \item If there exist $i,j$, such that $\pi(g_i) < 0 < \pi(g_j)$, then $M$ is a subgroup of $\FSym(\Z) \rtimes \Z$.
        Indeed, for any $h \in M$, there exists $k \geq 1, m,n \geq 0$ such that $k \pi(h) + m \pi(g_i) + n\pi(g_j)= 0$, so $h^k g_i^m g_j^n \in \FSym(\Z)$ and $(h^k g_i^mg_j^n)^t = \id$ for some $t \geq 1$, which yields
        $$h^{-1} = h^{k-1} g_i^mg_j^n (h^k g_i^mg_j^n)^{t-1} \in M.$$
        Therefore $M$ is a group and membership in $M$ reduces to Subgroup Membership.
        \item Otherwise $\pi(g_i)$ are all non-negative or all non-positive.
        Given $g \in \FSym(\Z) \rtimes \Z$, any product of elements $g_i$ evaluating to $g$ may only contain at most $|\pi(g)|$ occurrence of $g_i$ with $\pi(g_i) \neq 0$.
        Therefore it suffices to verify whether $g$ can be written as a product $h_0 g_{i_1} h_1 g_{i_2} h_2 \cdots g_{i_m} h_m$ with $m \leq |\pi(g)|$ and $h_0, \ldots, h_m$ in the finite group generated by $\{g_j \mid \pi(g_j) = 0\}$. \qedhere
    \end{itemize}
\end{proof}

\section{Submonoid Membership and subgroups of \texorpdfstring{$\FSym(\Z) \rtimes \Z$}{}}

\paragraph{Overview.}
In this section we prove our main result (Theorem~\ref{thm:submonoid_membership}) that Submonoid Membership is decidable in $\FSym(\Z) \rtimes \Z$.
By Lemma~\ref{lem:submonoid_to_subgroup}, it suffices to show that Subgroup Membership is decidable.
We start with an overview, together with a few examples highlighting the different types of subgroups that can appear.

Using an Euclidean algorithm, it is easy to see that a finitely generated subgroup of $\FSym(\Z) \rtimes \Z$ is either contained in $\FSym(\Z)$, or it can be written as $\gen{\tau_1, \ldots, \tau_s, \rho}$ with $\tau_1, \ldots, \tau_s \in \FSym(\Z)$, and $\rho \notin \FSym(\Z)$.
In Subsection~\ref{subsec:invar} we will reduce to the case where all $\gen{\rho}$-orbits are infinite: for example, we reduce to the case $\rho = (\cdots -2 \; 0 \; 2 \cdots)(\cdots -1 \; 1 \; 3 \cdots)$, instead of the case $\rho = (0 \; 1)(\cdots -2 \; -1 \; 2 \cdots)$.
This will require ``eliminating'' the finite orbits.

Next, it is not difficult to see that the membership problem in $\gen{\tau_1, \ldots, \tau_s, \rho}$ reduces to the membership problem in the $\rho$-invariant group $\gen{\tau_1, \ldots, \tau_s}_{\rho} \leq \FSym(\Z)$.
The following example gives some intuition on the structure of $\gen{\tau_1, \ldots, \tau_s, \rho}$ and $\gen{\tau_1, \ldots, \tau_s}_{\rho}$.

\begin{example}\label{example:invariant}
    We illustrate the groups $\gen{\tau_1, \ldots, \tau_s, \rho}$ and $\gen{\tau_1, \ldots, \tau_s}_{\rho}$ with three examples.
    \begin{enumerate}[nosep, leftmargin=8mm, label = (\arabic*)]
        \item Consider the group $\gen{\tau, \rho} \leq \FSym(\Z) \rtimes \Z$, generated by the 2-cycle $\tau \coloneq (01) \in \FSym(\Z)$ and the one-step shift $\rho \coloneqq \sigma$.
        See Figure~\ref{fig:exampleone} for an illustration.
        The intersection $\gen{\tau, \rho} \cap \FSym(\Z)$ is equal to $\gen{\tau}_{\rho}$, the group generated by the (infinitely many) elements $\rho^n\tau\rho^{-n}, n \in \Z$.
        For example, the element $\rho^2 \tau \rho^3 \tau \rho^{-5}$ is in the intersection $\gen{\tau, \rho} \cap \FSym(\Z)$ because $2 + 3 + (-5) = 0$; and we can indeed write $\rho^2 \tau \rho^3 \tau \rho^{-5}$ as $(\rho^2\tau\rho^{-2}) (\rho^5\tau\rho^{-5}) \in \gen{\tau}_{\rho}$.
        Since $\gen{\tau}_{\rho}$ is generated by the 2-cycles $\tau = (01), \rho\tau\rho^{-1} = (12), \rho^2\tau\rho^{-2} = (23), \ldots$, we have $\gen{\tau}_{\rho} = \FSym(\Z)$.
        \item Consider the group $\gen{\tau, \rho} \leq \FSym(\Z) \rtimes \Z$, generated by the 2-cycle $\tau \coloneq (01) \in \FSym(\Z)$ and the two-step shift $\rho \coloneqq \sigma^2$.
        See Figure~\ref{fig:exampletwo} for an illustration.
        In this case, the intersection $\gen{\tau, \rho} \cap \FSym(\Z) = \gen{\tau}_{\rho}$ is generated by the (infinitely many) 2-cycles $\tau = (01), \rho\tau\rho^{-1} = (23), \rho^2\tau\rho^{-2} = (45), \ldots$.
        Therefore $\gen{\tau}_{\rho}$ is a direct sum $\bigoplus_{\Z} C_2$, where $C_2$ is the cyclic group of two elements.
        \item Consider the group $\gen{\tau, \tau', \rho} \leq \FSym(\Z) \rtimes \Z$, generated by $\tau \coloneq (01)$, $\tau' \coloneqq (02)(13)$, and the two-step shift $\rho \coloneqq \sigma^2$.
        See Figure~\ref{fig:examplethree} for an illustration.
        The intersection $\gen{\tau, \tau', \rho} \cap \FSym(\Z) = \gen{\tau, \tau'}_{\rho}$ is generated by the 2-cycles $\rho^n\tau\rho^{-n} = (2n \; 2n+1), n \in \Z$, and the double cycles $\rho^n\tau'\rho^{-n} = (2n \; 2n+2)(2n+1 \; 2n+3), n \in \Z$.
        In this case, $\gen{\tau, \tau'}_{\rho}$ preserves the partition $\bigsqcup_{n \in \Z}\{2n, 2n+1\}$.
        Here, the generator $\rho^n\tau\rho^{-n}$ swaps the two elements within the individual block $\{2n, 2n+1\}$, whereas the generator $\rho^n\tau'\rho^{-n}$ swaps the block $\{2n, 2n+1\}$ with its neighbouring block $\{2n+2, 2n+3\}$.
        Therefore, $\gen{\tau, \tau'}_{\rho}$ is a wreath product $C_2 \wr_{\Z} \FSym(\Z)$.
    \end{enumerate}
\end{example}

\begin{figure}[h!]
    \centering
    \begin{minipage}[t]{.28\textwidth}
        \centering
        \figureexampleone
        \caption{Example~\ref{example:invariant}(1)}
        \label{fig:exampleone}
    \end{minipage}
    \hfill
    \begin{minipage}[t]{.32\textwidth}
        \centering
        \figureexampletwo
        \caption{Example~\ref{example:invariant}(2)}
        \label{fig:exampletwo}
    \end{minipage}
    \hfill
    \begin{minipage}[t]{.32\textwidth}
        \centering
        \figureexamplethree
        \caption{Example~\ref{example:invariant}(3)}
        \label{fig:examplethree}
    \end{minipage}
\end{figure}

We then focus on deciding membership in $\gen{\tau_1, \ldots, \tau_s}_{\rho}$.
Note that in Example~\ref{example:invariant}(2), the orbits of $\gen{\tau}_{\rho}$ are all finite;
whereas in Example~\ref{example:invariant}(3), the orbit of $\gen{\tau, \tau'}_{\rho}$ is infinite.
In both examples, one can also replace $\sigma^2$ by $\sigma^d$ for some $d \geq 1$, and replace $C_2$ by any finite group $H \leq \Sym(\{1, \ldots, d\})$.
This way, we obtain the more general examples of $\rho$-invariant groups $\bigoplus_{\Z}H$ and $H \wr_{\Z} \FSym(\Z)$.
These two examples will be the essential building blocks for constructing all $\rho$-invariant groups $\gen{\tau_1, \ldots, \tau_s}_{\rho}$.

We continue the overview of the section.
In Subsection~\ref{subsec:orbits}, we will classify the finite and infinite orbits of $\gen{\tau_1, \ldots, \tau_s}_{\rho}$.
In this way, $\gen{\tau_1, \ldots, \tau_s}_{\rho}$ can be written as a subgroup of $G_{\fin} \times G_{\infty}$, where $G_{\fin}$ is the permutation group induced on the union of finite orbits, and $G_{\infty}$ is the permutation group induced on the union of infinite orbits.
In Subsection~\ref{subsec:finite_orbits} we compute $G_{\fin}$: we will show that $G_{\fin}$ is a $\rho$-invariant subgroup of $\bigoplus_{\Z}H$ for some finite $H$.
In Subsection~\ref{subsec:infinite_orbits} we compute $G_{\infty}$: we will show that $G_{\infty}$ can be obtained by taking a finite index subgroup of some direct product $\prod_{i = 1}^k \left(H_i \wr_{\Z} \FSym(\Z)\right)$.
In Subsection~\ref{subsec:membership}, we obtain a characterization of the group $\gen{\tau_1, \ldots, \tau_s}_{\rho} \leq G_{\fin} \times G_{\infty}$ by applying Goursat's Lemma (Lemma~\ref{lem:Goursat}) on the direct product $G_{\fin} \times G_{\infty}$.
The following example illustrates the relation between $\gen{\tau_1, \ldots, \tau_s}_{\rho}, G_{\fin}$, and $G_{\infty}$.

\begin{example}\label{example:product}
    Consider the group $\gen{\tau, \tau', \rho} \leq \FSym(\Z) \rtimes \Z$, generated by $\tau \coloneq (01)(23)$, $\tau' \coloneqq (26)(37)$, and the four-step shift $\rho \coloneqq \sigma^4$.
    See Figure~\ref{fig:examplefour} for an illustration.
    The intersection $\gen{\tau, \tau', \rho} \cap \FSym(\Z) = \gen{\tau, \tau'}_{\rho}$ is generated by the double cycles $\rho^n\tau\rho^{-n} = (4n \; 4n+1)(4n+2 \; 4n+3), n \in \Z$, and $\rho^n\tau'\rho^{-n} = (4n+2 \; 4n+6)(4n+3 \; 4n+7), n \in \Z$.
    The finite orbits of $\gen{\tau, \tau'}_{\rho}$ are $\{0, 1\}, \{4, 5\}, \{8, 9\}, \ldots$;
    there is an infinite orbit $\{2, 3, 6, 7, 10, 11, \ldots\}$.
    
    In this case, the group $G = \gen{\tau, \tau'}_{\rho}$ acts on the finite orbits as $G_{\fin} = \bigoplus_{\Z} C_2$, similar to Example~\ref{example:invariant}(2).
    Meanwhile, $G$ acts on the infinite orbit as $G_{\infty} = C_2 \wr_{\Z} \FSym(\Z)$, similar to Example~\ref{example:invariant}(3).
    Thus, $G$ is a subgroup of $G_{\fin} \times G_{\infty}$.
    
    However, $G$ is not equal to $G_{\fin} \times G_{\infty}$.
    To see this, note that each $\rho^n\tau\rho^{-n}$ simultaneously inverts the block $\{4n, 4n+1\}$ in the finite orbit and the block $\{4n+2, 4n+3\}$ in the infinite orbit.
    Therefore, for every element in $G$, the total number of inverted blocks must be even.
    Due to this parity constraint, the element $(0 1) \in G_{\fin} \times G_{\infty}$ cannot appear in $G$.
    Nevertheless, we can show (c.f.\ Proposition~\ref{prop:structure_invariant_subgroup}) that $G$ contains a ``big'' subgroup of $G_{\infty}$.
\end{example}

\begin{figure}[h!]
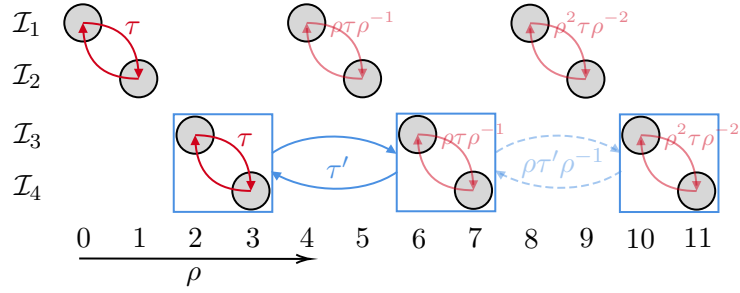

    \centering
    \figureexamplefour
    \caption{Illustration of Example~\ref{example:product}}
    \label{fig:examplefour}
\end{figure}

More precisely, in Subsection~\ref{subsec:membership} we will prove:
\begin{restatable}{proposition}{propstructure}\label{prop:structure_invariant_subgroup}
    Let $G \leq \FSym(\{1, \ldots, d\} \times \Z)$ be finitely generated as a $\rho$-invariant subgroup.
    Then $G$ is a subgroup of $G_{\fin} \times G_{\infty}$ whose projections on $G_{\fin}$ and $G_{\infty}$ are surjective, and
    \begin{enumerate}[itemsep=1pt, topsep=3pt, label = {\normalfont(\roman*)}]
        \item $G_{\fin}$ is a $\rho$-invariant subgroup of $\bigoplus_{\Z} H$ for some finite $H$, where $\rho$ acts on $\bigoplus_{\Z} H$ as the shift $(h_k)_{k \in \Z} \mapsto (h_{k-1})_{k \in \Z}$.
        \item $G_{\infty}$ is either trivial, or it can be effectively written as a subgroup of $W = W_1 \times \cdots \times W_k$, where $W_i = H_i \wr_{\Delta_i} \FSym(\Delta_i)$ for some finite $H_i$ and countably infinite sets $\Delta_1,\ldots,\Delta_k$.
    \end{enumerate}
    Furthermore, $G \geq \{\id\} \times W'$, where $W' = W_1' \times \cdots \times W_k'$ is the commutator subgroup of $W$.
\end{restatable}

The most non-trivial part of Proposition~\ref{prop:structure_invariant_subgroup} is to prove that $G$ contains the ``big'' subgroup $\{\id\} \times W'$.
This characterization allows us to reduce the membership problem in $G = \gen{\tau_1, \ldots, \tau_s}_{\rho}$, to the membership problem in $G/(\{\id\} \times W') \leq (\bigoplus_{\Z} H) \times (W/W')$.
Another key fact is that the quotient $W/W'$ is always finite (c.f.~Lemma~\ref{lem:transitive_finite}, due to Neumann~\cite{neumann1976structure}).
Finally in Subsection~\ref{subsec:membership}, we use the embedding $\bigoplus_{\Z} H \hookrightarrow H \wr \Z$ and the finiteness of $W/W'$ to reduce the above problem to the Subgroup Membership in the wreath product $H \wr \Z$.
This is decidable by the result of Lohrey, Steinberg,
and Zetzsche (Lemma~\ref{lem:decidable_H_wreath_Z}).

This concludes the overview.
We now begin the detailed proof of our main result (Theorem~\ref{thm:submonoid_membership}).

\subsection{Reduction to $\rho$-invariant subgroups}\label{subsec:invar}

In this subsection we reduce Subgroup Membership in $\FSym(\Z) \rtimes \Z$, to deciding membership in $\rho$-invariant subgroups of $\FSym(\{1, \ldots, d\} \times \Z)$, where $\rho \colon (i, z) \mapsto (i, z+1)$ denotes the ``shift on the second coordinate''.
To do this, we need to eliminate the finite orbits.

Let 
\[
\pi \colon \FSym(\Z) \rtimes \Z \rightarrow \Z
\]
denote the projection.
Let $G$ be a finitely generated subgroup of $\FSym(\Z) \rtimes \Z$. 
If $\pi(G) = \{0\}$, then $G$ is finite and membership in $G$ can be decided by enumerating its elements. Otherwise, $\pi(G)$ is a non-trivial subgroup of $\Z$ and $G$ is infinite. We can say a bit more:
\begin{lemma}\label{lem:compute_orbits}
    If $G\le \FSym(\Z)\rtimes\Z$ is finitely generated and $\pi(G)\ne\{0\}$, then the action of $G$ on $\Z$ has finitely many orbits. Moreover one can effectively compute a decomposition
    \[ \Z = (\mF_1\sqcup\ldots\sqcup \mF_c) \sqcup (\mI_1\sqcup \ldots \sqcup \mI_d) \]
    from the generators of $G$, where $\mF_i$ are the finite orbits and $\mI_i$ are the infinite orbits.
\end{lemma}
\begin{proof}
Suppose that $G=\la g_1,g_2,\ldots,g_r\ra$ where $g_i = \tau_i\sigma^{\pi(g_i)}$ with $\tau_i\in\FSym(\Z)$ and $\pi(g_i)\in \Z$ (not all zero). Let $m=\max\abs{\pi(g_i)} > 0$. We can compute an interval $[a,b]$ such that $\supp(\tau_i)\subset [a,b]$ for $i=1,\ldots,r$. 
We claim that for all $x>b$ and $y \in \Z$, $x$ and $y$ are in the same orbit if and only if and $x + m$ and $y$ are in the same orbit.
Indeed, without loss of generality supposing that $\pi(g_1) = m$, we have $g(x+m) = y \iff gg_1(x) = y$.
Similarly, for all $x < a$ and $y \in \Z$, $x$ and $y$ are in the same orbit if and only if and $x - m$ and $y$ are in the same orbit.
Therefore, to compute all the orbits, we can draw the finite graph with
\begin{itemize}[leftmargin=8mm]
    \item vertex set $\Z/\!\!\sim$, where the congruence relation $\sim$ is defined as $x\sim y$ if either ($x,y<a-m$ and $x\equiv y\pmod m$), or ($x,y>b+m$ and $x\equiv y\pmod m$), or $x=y$;
    \item edges given by $([x],[g_i(x)])$ for $x\in\Z$ and $i=1,\ldots,r$, where $[x]$ denotes the equivalence class of $x$ under $\sim$.
\end{itemize}
The congruence classes of each connected graph component corresponds to a $G$-orbit.
This orbit is infinite if it contains a vertex whose congruence class is infinite.
\end{proof}

Our goal is to reduce to the case where there are only infinite orbits.
To achieve this, the next idea is to ``cover'' the finite orbits using a single element:

\begin{lemma}\label{lem:rho_inf_orbit}
    Let $G$ be a finitely generated subgroup of $\FSym(\Z) \rtimes \Z$.
    Suppose $G$ is infinite and let $\mF = \mF_1 \sqcup \cdots \sqcup \mF_c$ be the union of finite orbits of $G$.
    Then, there exists an effectively computable element $\rho \in G$, such that 
    \begin{enumerate}[nosep, label={\normalfont(\roman*)}]
        \item $\pi(\rho) > 0$,
        \item the union of finite orbits of $\gen{\rho}$ is equal to $\mF$,
        \item $\rho(z) = z$ for all $z \in \mF$.
    \end{enumerate}
\end{lemma}
\begin{proof}
    For an illustration of the lemma and its proof, see Figures~\ref{fig:rho1} and \ref{fig:rho2}. We will first find an element $\rho$ satisfying (i) and (ii), it suffices to take a suitable power of $\rho$ so that (iii) is satisfied. \medskip

    \begin{figure}[h]
    \centering
    \begin{minipage}[t]{.48\textwidth}
        \centering
        \figurerhoone
        \caption{The arrows denote two permutations $g,k$. For $G = \gen{g, k}$, there is only one finite $G$-orbit $\mF = \{5\}$. The colours mark different $\gen{g}$-orbits.
        The finite $\gen{g}$-orbits are $\{4\}$ and $\{5\}$.}
        \label{fig:rho1}
    \end{minipage}
    \hfill
    \begin{minipage}[t]{.48\textwidth}
        \centering
        \figurerhotwo
        \caption{For the element $\rho \coloneqq kg^2$, the different colours mark different $\gen{\rho}$-orbits. The only finite $\gen{\rho}$-orbit is $\mF = \{5\}$.}
        \label{fig:rho2}
    \end{minipage}
    \end{figure}
    
    \noindent Let $\mI = \mI_1 \sqcup \cdots \sqcup \mI_d$ denote the union of infinite orbits of $G$.
    Take any $g \in G$ such that $\pi(g)\ge 1$.
    Let $\mF_g$ denote the union of finite orbits of $\gen{g}$, then $\mF_g \supseteq \mF$.
    Similarly, let $\mI_g \subseteq \mI$ denote the union of infinite orbits of $\gen{g}$, so $\Z = \mI_g \sqcup \mF_g$.
    Since $\pi(g) \neq 0$, $\gen{g}$ has only finitely many finite orbits, so the cardinality $\card{\mF_g}$ is finite. If $\card{\mF_g} = \card{\mF}$, then we have found our $\rho$. If on the contrary $\card{\mF_g} > \card{\mF}$, we will find another element $h \in G$ such that $\card{\mF_h} < \card{\mF_g}$. \medskip
    
    \noindent Since $\card{\mF_g} > \card{\mF}$, there exists $z_0 \in \mF_g \setminus \mF$. The $G$-orbit of $z_0$ is infinite, so we can find $k \in G$ such that $k(z_0) \in \mI_g$. The idea is to take $h=kg^n$ with $n$ large enough so that $g^n$ (and $kg^n$) moves points from $\mI_g$ so far to the right that their $\gen h$-orbits must be infinite, and $k$ inserts $z_0\in \mF_g$ into $\mI_g$. \smallskip
    
    
    \begin{adjustwidth}{2mm}{2mm}
        \textbf{Claim.} Consider $g\in\FSym(\Z)\rtimes\Z$ such that $\pi(g)\ge 1$. For any constant $C$, there exists an exponent $n=n(C)\ge 1$ such that $g^n(x)\ge x+C$ for all $x\in\mI_g$.
        
        \medskip
        \noindent$\triangleright$ Recall that each element $g\in\FSym(\Z)\rtimes\Z$ acts eventually like translation. Equivalently, there exists a finite set $F\subset\Z$ such that $g(x)=x+\pi(g)$ for all $x\in\Z\setminus F$. Let $m_g=\max_{x\in F}\bigl(x-g(x)\bigr)$ be the largest backward increment of $g$. We consider $x\in\mI_g$ and observe that there are at most at most $\abs F$ indices $i\in\Z$ such that $g^i(x)\in F$, since $\gen g\cdot x$ is infinite. Finally we have
        \[g^n(x)-x = \sum_{i=0}^{n-1}\bigl(g^{i+1}(x)-g^i(x)\bigr)\ge (n-\card F)\cdot \pi(g)-\abs F\cdot m_g\ge C.\]
        for all $x\in\mI_g$ and $n\ge C+(m_g+1)\card F$. \hfill $\triangleleft$
    \end{adjustwidth}
\medskip

\begin{adjustwidth}{3mm}{3mm}
    \textbf{Claim.} Consider $h\in\Sym(\Z)$. Suppose that there exists $F\subset \Z$ and $m\ge 1$ such that 
    \[ \forall x\in F,\; h(x)\ge x-m, \quad\text{and}\quad \forall x\in\Z\setminus F,\; h(x)> x+m\card F.\]
    Then all orbits $\gen h\cdot x$ which are not included in $F$ are infinite.

    \medskip
    \noindent $\triangleright$ Consider $x\in\Z$ such that $\gen h\cdot x\not\subset F$, and suppose on the contrary that there exists $t\ge 1$ such that $h^t(x) = x$. We consider the smallest such $t$, so that the elements $x, h(x), h^2(x), \ldots, h^{t-1}(x)$ are distinct. In particular, there are at most $\card F$ indices $i\in\{0,1,\ldots,t-1\}$ such that $h^i(x)\in F$, in which case $h^{i+1}(x)\ge h^i(x)-m$. Otherwise, we have $h^{i+1}(x)\ge h^i(x)+m\card F\ge h^i(x)$. Moreover, the later case happens at least once since $\gen h\cdot x\not\subset F$. We conclude that
    $$h^t(x) > x - \card{F} \cdot m + m\card{F} = x,$$
    a contradiction. \hfill $\triangleleft$
\end{adjustwidth}
\medskip
\noindent We apply this claim to $h=kg^n$, with the finite set $F=\mF_g$ and $m=\mathrm{diam}(\mF_g)$, where $n$ is taken as $n\bigl(m\abs F+\max_{x\in F}|x-k(x)|+1\bigr)$ from the first claim. 
The hypothesis $\gen h\cdot x\not\subseteq F$ applies to all $x\notin \mF_g$, as well as $x=g^{-n}(z_0)\in\mF_g$ (recall that $h(x)=k(z_0)\in\mI_g$), proving that $\mF_h \subsetneq\mF_g$. By induction on $\card{\mF_g}$, we can construct an element $\rho \in G$ such that $\mF_{\rho} = \mF$ and $\pi(\rho) > 0$. We ensure that $\rho(z) = z$ for all $z \in \mF$ by replacing $\rho$ with a suitable power $\rho^{\card{\mF}!}$.
\end{proof}

The following folklore result shows that decidability of Subgroup Membership is preserved by taking finite extension of groups.
\begin{lemma}[Nielsen-Schreier, c.f.~\cite{grunschlag1999algorithms}]\label{lem:finite_index_subgroup}
    Let $A$ be a finitely generated group, $G=\gen{g_1,\ldots,g_r}$ a finitely generated subgroup and $\varphi \colon G \to F$ be a homomorphism in a finite group $F$.
    Then $\ker(\varphi)$ is finitely generated and a set of generators is effectively computable.
    Furthermore, the membership problem in $G$ reduces to the membership problem in $\ker(\varphi)$.
\end{lemma}

We are now ready to ``eliminate'' the finite orbits.
Let $G$ be a finitely generated subgroup of $\FSym(\Z) \rtimes \Z$.
Suppose $G$ is infinite, then $\pi(G) = d \Z$ for some $d \geq 1$.
Let $\rho \in G$ be as in Lemma~\ref{lem:rho_inf_orbit}, then $\pi(\rho) = dn$ for some $n \geq 1$.
Let $\mF$ be the union of finite $G$-orbits.
For any element $g \in G$, denote by $\pi_{\mF}(g) \in \Sym(\mF)$ its induced action on $\mF$.
Consider the map 
\begin{align*}
\varphi \colon G &\rightarrow \Sym(\mF) \times \Z/n\Z,  \\
g &\mapsto \left(\pi_{\mF}(g), \; \frac{\pi(g)}{d} \bmod n\right).
\end{align*}
Its kernel $\widetilde{G} \coloneqq \ker(\varphi)$ is the set of elements $g \in G$ that act trivially on $\mF$ and such that $n \mid \frac{\pi(g)}{d}$. 
By Lemma~\ref{lem:finite_index_subgroup}, the membership problem in $G$ reduces to the membership problem in $\widetilde{G} \coloneqq \ker(\varphi)$.

\begin{lemma}\label{lem:standard_generating_set}
    Let $\widetilde{G} \coloneqq \ker(\varphi)$.
    Then we can compute elements $\tau_1, \ldots, \tau_s \in \widetilde{G}$ with $\pi(\tau_i) = 0$ for all $i$, such that the set $\{\tau_1, \ldots, \tau_s, \rho\}$ generates $\widetilde{G}$.
\end{lemma}
\begin{proof}
    By Lemma~\ref{lem:finite_index_subgroup} we can compute a set of generators $\{h_1, \ldots, h_s\}$ of $\widetilde{G}$.
    For $i = 1, \ldots, s$, let $\tau_i \coloneqq \rho^{-\frac{\pi(h_i)}{dn}} \cdot h_i$.
    Then $\{\tau_1, \ldots, \tau_s, \rho\}$ generates $\widetilde{G}$, and $\pi(\tau_i) = 0$ for all $i$.
\end{proof}

By the definition of $\varphi$, the group $\widetilde{G}$ stabilizes every element in $\mF$.
Since $\rho \in \widetilde{G}$ and every $\gen{\rho}$-orbit is either infinite or is in $\mF$, we obtain that every $\widetilde{G}$-orbit is either infinite or is in $\mF$.
We now without loss of generality replace $G$ by $\widetilde{G}$, and replace $\Z$ by $\Z \setminus \mF$.
Thus $G = \gen{\tau_1, \ldots, \tau_s, \rho}$ with $\pi(\tau_i) = 0$ for all $i$, and $\pi(\rho) > 0$, and all $\gen{\rho}$-orbits are now infinite. 
Let $\mI_1, \ldots, \mI_d$ denote the $\gen{\rho}$-orbits, so $\tau_1, \ldots, \tau_s \in \FSym(\mI_1 \sqcup \cdots \sqcup \mI_d)$.

For each orbit $\mI_j$, pick any $i \in \mI_j$, then $\mI_j = \{\rho^z(i) \mid z \in \Z\}$.
This allows us to identify $\mI_j$ with the set $\Z$, where $\rho$ acts as $z \mapsto z+1$.
Therefore, we identify
\begin{equation}\label{eq:identify_Ztimesd}
    \mI_1 \sqcup \cdots \sqcup \mI_d \cong \{1, \ldots, d\} \times \Z,
\end{equation}
where $\rho$ acts as incrementation of the second coordinate $(j, z) \mapsto (j, z+1)$.
The conjugation $g \mapsto \rho g \rho^{-1}$ defines an automorphism of $\FSym(\{1, \ldots, d\} \times \Z)$.
Recall that for a set $S \subseteq \FSym(\{1, \ldots, d\} \times \Z)$, the $\rho$-invariant group $\gen{S}_{\rho}$ is the group generated by $\rho^n s \rho^{-n}, n \in \Z, s \in S$.
This allows us to reduce Subgroup Membership in $\FSym(\Z) \rtimes \Z$ to the membership problem in $\rho$-invariant subgroups of $\FSym(\{1, \ldots, d\} \times \Z)$:

\begin{proposition}\label{prop:subgroup_to_invariant_subgroup}
    Subgroup Membership in $\FSym(\Z) \rtimes \Z$ reduces to the following problem: \smallskip
    \begin{adjustwidth}{3mm}{3mm}
        {\normalfont\textbf{Input:}} elements $\tau_1, \ldots, \tau_s, g \in \FSym(\{1, \ldots, d\} \times \Z)$. \smallskip\\
        {\normalfont\textbf{Question:}} whether $g \in \gen{\tau_1, \ldots, \tau_s}_{\rho}$.
    \end{adjustwidth}
\end{proposition}
\begin{proof}
    Let $G$ be a finitely generated subgroup of $\FSym(\Z) \rtimes \Z$ and $g \in \FSym(\Z) \rtimes \Z$: we need to decide whether $g \in G$. We separate two cases: \smallskip
    
    \noindent\textbullet\, If $\pi(G) = \{0\}$, then $G$ is finite and we can enumerate all its elements and compare them with $g$. \smallskip

    \noindent\textbullet\, If $\pi(G) \neq \{0\}$, by passing to a finite index subgroup (Lemma~\ref{lem:finite_index_subgroup} and \ref{lem:standard_generating_set}), we can suppose $G$ to be generated by the elements $\tau_1, \ldots, \tau_s, \rho,$ with $\tau_i \in \FSym(\Z)$ and $\pi(\rho) = d>0$.
    If $d \nmid \pi(g)$ or $g$ does not fix all elements of $\mF$, then obviously $g \notin G$. 
    Otherwise, we have $\pi(g) = ad$ for some $a \in \Z$, so $\rho^{-a}g \in \ker(\pi)$.
    Since $g \in G$ if and only if $\rho^{-a} g \in G$, we can without loss of generality replace $g$ with $\rho^{-a} g$ and suppose $g \in \ker(\pi) = \FSym(\mI_1 \sqcup \cdots \sqcup \mI_d)$.
    By Lemma~\ref{lem:ker_equals_inv} below, we have
    \[
    G \cap \ker(\pi) = \gen{\tau_1, \ldots, \tau_s}_{\rho}.
    \]
    Therefore $g \in G$ if and only if $g \in \gen{\tau_1, \ldots, \tau_s}_{\rho}$.
    We conclude the proof by identifying $\ker(\pi) = \FSym(\mI_1 \sqcup \cdots \sqcup \mI_d)$ with the group $\FSym(\{1, \ldots, d\} \times \Z)$ using the bijection~\eqref{eq:identify_Ztimesd}.
\end{proof}

\begin{lemma}\label{lem:ker_equals_inv}
    Consider a morphism $\pi\colon H\rightarrow \Z$ and a finitely generated subgroup $G=\gen{\tau_1,\ldots,\tau_s,\rho} \leq H$, with $\pi(\tau_1)= \cdots = \pi(\tau_s) = 0$ and $\pi(\rho) \neq 0$. Then $G \cap \ker(\pi) = \gen{\tau_1,\ldots,\tau_s}_\rho$.
\end{lemma}
\begin{proof}
    By restricting the domain of $\pi$ to $G$, it suffices to consider the case $H = G$ and prove $\ker(\pi) = \gen{\tau_1,\ldots,\tau_s}_\rho$.
    
    Recall that the group $\gen{\tau_1, \ldots, \tau_s}_{\rho}$ is generated by elements of the form $\rho^\ell \tau_i \rho^{-\ell}, \ell \in \Z, 1 \leq i \leq n$.
    Each of these generators satisfies $\pi(\rho^n \tau_i \rho^{-n}) = n \pi(\rho) + \pi(\tau_i) - n \pi(\rho) = 0$, and therefore belongs to $\ker(\pi)$.
    We thus have $\gen{\tau_1,\ldots,\tau_s}_\rho \leq \ker(\pi)$.
    
    For the other inclusion, consider any element of $\ker(\pi)$: it is of the form
    \begin{equation}\label{eq:element_ker}
    \rho^{k_1}\tau_{i_1}^{\varepsilon_1} \rho^{k_2}\tau_{i_2}^{\varepsilon_2} \cdots \rho^{k_m}\tau_{i_m}^{\varepsilon_m} \rho^{k_{m+1}},
    \end{equation}
    with $\varepsilon_i \in \{1, -1\}$ and $k_1 + \cdots + k_{m+1} = 0$.
    It can be rewritten as
    \[
    \rho^{\ell_1}\tau_{i_1}^{\varepsilon_1}\rho^{-\ell_1} \cdot \rho^{\ell_2}\tau_{i_2}^{\varepsilon_2} \rho^{-\ell_2} \cdot \ldots \cdot \rho^{\ell_m}\tau_{i_m}^{\varepsilon_m} \rho^{-\ell_m},
    \]
    where $\ell_i \coloneqq k_1 + \cdots + k_i, 1 \leq i \leq m$.
    Therefore the element~\eqref{eq:element_ker} is in $\gen{\tau_1, \ldots, \tau_s}_{\rho}$.
    We conclude that $\ker(\pi) = \gen{\tau_1, \ldots, \tau_s}_{\rho}$.
\end{proof}

The following subsections will focus on characterizing $\rho$-invariant subgroups of $\FSym(\{1, \ldots, d\} \times \Z)$, as well as deciding their membership problem.

\subsection{Orbits of $\rho$-invariant subgroups}\label{subsec:orbits}
We now consider the membership problem in the $\rho$-invariant subgroup $\gen{\tau_1, \ldots, \tau_s}_{\rho} \leq \FSym(\{1, \ldots, d\} \times \Z)$.
Let 
\[
G = \gen{\tau_1, \ldots, \tau_s}_{\rho}.
\]
We now classify the possible partitions of $\{1, \ldots, d\} \times \Z$ into $G$-orbits.
Let 
\[
\pi_{[d]} \colon \{1, \ldots, d\} \times \Z \rightarrow \{1, \ldots, d\}
\]
denote the projection onto the first coordinate.
The following lemma describes the possible forms of a $G$-orbit in $\{1, \ldots, d\} \times \Z$.
\begin{lemma}\label{lem:single_orbit}
    Let $\mO \subseteq \{1, \ldots, d\} \times \Z$ be a $G$-orbit.
    Let $\pi_{[d]}(\mO)=\{j_1,\ldots,j_k\}$.
    Then:
    \begin{enumerate}[nosep, label={\normalfont(\roman*)}]
        \item either $\mO$ is finite, and there exist $z_1, \ldots, z_k \in \Z$, such that 
        \begin{equation}\label{eq:orbit_finite}
            \mO = \{(j_1, z_1), \ldots, (j_k, z_k)\};
        \end{equation}
        \item or $\mO$ is infinite, and there exist $z_1, \ldots, z_k \in \Z$ and $\ell \geq 1$, such that 
        \begin{equation}\label{eq:orbit_infinite}
            \mO = \{(j_1, z_1 + \ell x) \mid x \in \Z\} \cup \cdots \cup \{(j_k, z_k + \ell x) \mid x \in \Z\}.
        \end{equation}
    \end{enumerate}
\end{lemma}
\begin{proof}
    Consider two cases.
    
    \textbf{Case 1:} for every $i = 1, \ldots, k$, the set $\mO \cap (\{j_i\} \times \Z)$ is a singleton.
    Then we have $\mO = \{(j_1, z_1), \ldots, (j_k, z_k)\}$ for some $z_1, \ldots, z_k \in \Z$ and $\mO$ is finite.

    \textbf{Case 2:} there is some $i$ such that $\mO \cap (\{j_i\} \times \Z)$ contains at least two elements. 
    We show that $(j, z), (j, z + \ell), (j', y) \in \mO \implies (j', y + \ell) \in \mO$. 
    Suppose $(j, z), (j, z + \ell), (j', y) \in \mO$.
    Take an $h \in G$ with $h (j, z) = (j', y)$.
    By the $\rho$-invariance of $G$, we have $\rho^{\ell} h \rho^{-\ell} \in G$, and
    \[
    \rho^{\ell} h \rho^{-\ell}(j, z + \ell) = \rho^{\ell} h (j, z) = \rho^{\ell} (j', y) = (j', y + \ell).
    \]
    Therefore $(j', y + \ell) \in \mO$.

    Take the smallest $\ell \geq 1$ such that there exist $j, z$ with $(j, z), (j, z + \ell) \in \mO$.
    Then we have $\mO = \{(j_1, z_1 + \ell x) \mid x \in \Z\} \cup \cdots \cup \{(j_k, z_k + \ell x) \mid x \in \Z\}$ for some $z_1, \ldots, z_k \in \Z$.
\end{proof}

The following proposition describes the possible partitions of $\{1, \ldots, d\} \times \Z$ into $G$-orbits.
See Figure~\ref{fig:orbits} for an example of the partition.

\begin{proposition}\label{prop:partition_orbits}
    Let $G$ be a $\rho$-invariant subgroup of $\FSym(\{1, \ldots, d\} \times \Z)$, generated by finitely many elements.
    Then one can effectively compute $z_1, \ldots, z_d \in \Z$, as well as a partition of the set $\{1, \ldots, d\}$ into a disjoint union $J_{\fin} \sqcup J_1 \sqcup \cdots \sqcup J_n$, such that:
    \begin{enumerate}[nosep, label={\normalfont(\roman*)}]
        \item For each $x \in \Z$, the set $\mA_x \coloneqq \{(j, z_j + x) \mid j \in J_{\fin}\}$ is a finite union of (finite) $G$-orbits.
        \item For each $J_i, i = 1, \ldots, n$, there is an effectively computable $\ell_i \geq 1$ such that the $\ell_i$ sets
        $$\{(j, z_j + \ell_i x+r) \mid j \in J_i, x \in \Z\}$$
        (for $r=0,1,\ldots,\ell_i-1$) are infinite $G$-orbits.
    \end{enumerate}
\end{proposition}
\begin{proof}
    Since $G$ is $\rho$-invariant, the partition of $\{1, \ldots, d\} \times \Z$ into $G$-orbits is also preserved $\rho$.
    Therefore if $\{(j, z_j + \ell x) \mid j \in J, x \in \Z\}$ is a $G$-orbit then $\{(j, z_j + \ell x + 1) \mid j \in J, x \in \Z\}$ is also a $G$-orbit.
    The statement of the proposition follows thus directly from Lemma~\ref{lem:single_orbit}, subject to effectivity.
    We now show that the partition into $G$-orbits can be effectively computed.

    Given $(j, y) \in \{1, \ldots, d\} \times \Z$, we will compute its $G$-orbit using a saturation process.
    Let $\tau_1, \ldots, \tau_s$ be the generators of $G$ as a a $\rho$-invariant subgroup.
    For a set $\mO$ of the form~\eqref{eq:orbit_finite}, we can check whether it is invariant under the action of $G$: it is invariant if and only if for each $z \in \Z, 1 \leq i \leq n$, such that $\supp(\rho^{z} \tau_i \rho^{-z}) \cap \mO \neq \emptyset$, we have $\rho^{z} \tau_i \rho^{-z}(\mO) \subseteq \mO$.
    This can be checked effectively since there are only finitely many $z \in \Z, 1 \leq i \leq n$ such that $\rho^{z}(\supp(\tau_i)) \cap \mO \neq \emptyset$.
    For a set $\mO$ of the form~\eqref{eq:orbit_infinite}, we can also check whether it is invariant under the action of $G$: it is invariant if and only if for each $0 \leq z \leq \ell-1, 1 \leq i \leq n$, we have $\rho^{z} \tau_i \rho^{-z}(\mO) \subseteq \mO$.
    Again this can be checked effectively since $\mO$ is periodic.
    Note that in both cases, if $\mO$ is not invariant under the action of $G$, then we can produce an element $(i, x) \in G \mO \setminus \mO$. 
    
    Therefore, to compute the $G$-orbit of $(j, y) \in \{1, \ldots, d\} \times \Z$, we start with the set $\mO \coloneqq \{(j, y)\}$, and iterate the following: at each step, check whether $\mO$ is invariant under the action of $G$.
    If $\mO$ is not invariant, let $(i, x) \in G \mO \setminus \mO$ and replace $\mO$ with the smallest set of the form~\eqref{eq:orbit_finite} or \eqref{eq:orbit_infinite} containing $\mO \cup \{(i, x)\}$.
    We repeat this step until $\mO$ is invariant under the action of $G$.
    It is easy to see that an ascending chain of sets of the forms~\eqref{eq:orbit_finite} and \eqref{eq:orbit_infinite} must stabilize, therefore the above procedure must terminate.
\end{proof}

\begin{figure}[h]
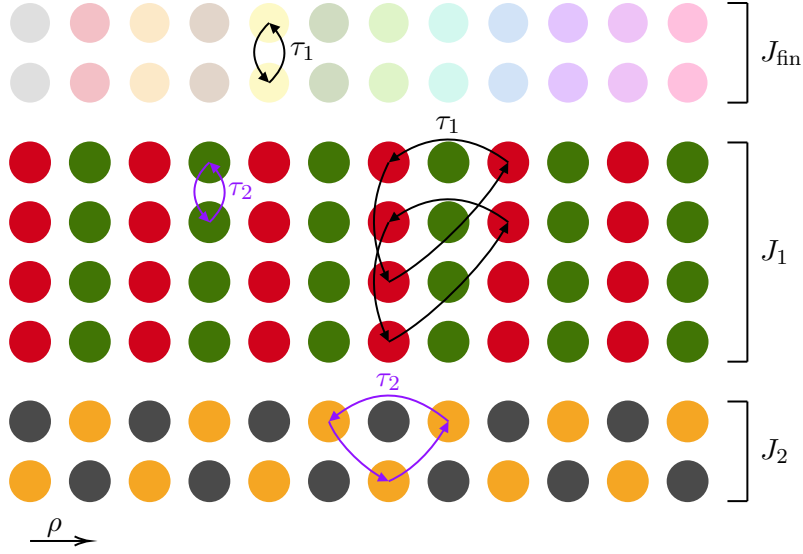

    \centering
    \figureorbits
    \caption{A possible partition of $\{1, \ldots, 8\} \times \Z$ into $G$-orbits. Here, $G = \gen{\tau_1, \tau_2}_{\rho}$ where the permutations $\tau_1, \tau_2$ are drawn in arrows. Different colours denote elements of different $G$-orbits.}
    \label{fig:orbits}
\end{figure}

Denote $J_{\infty} \coloneqq J_1 \sqcup \cdots \sqcup J_n$.
Now $J_{\fin} \times \Z = \bigcup_{x \in \Z} \mA_x$ is the union of finite $G$-orbits, and $J_{\infty} \times \Z$ is the union of infinite $G$-orbits.
Let $G_{\fin}$ denote the permutation group of $J_{\fin} \times \Z$ induced by the action of $G$ on the finite orbits,
and let $G_{\infty}$ denote the permutation group of $J_{\infty} \times \Z$ induced by the action of $G$ on the infinite orbits.
Then 
\[
G \leq G_{\fin} \times G_{\infty},
\]
where the projections $G \rightarrow G_{\fin}$ and $G \rightarrow G_{\infty}$ are surjective.

In the following subsections, we will characterize the groups $G_{\fin}$ and $G_{\infty}$, and then characterize $G \leq G_{\fin} \times G_{\infty}$ using Goursat's lemma (Lemma~\ref{lem:Goursat}).
To this end, we will also need to understand the normal subgroups of $G_{\fin}$ and $G_{\infty}$.
Note that $G_{\fin}$ and $G_{\infty}$ are $\rho$-invariant, and their generators (as $\rho$-invariant groups) are simply the projections of the generators $\tau_1, \ldots, \tau_s$ of $G$.

\subsection{Permutation group on the finite orbits}\label{subsec:finite_orbits}
In this subsection we will characterize the group $G_{\fin}$.
The group $G_{\fin}$ acts faithfully on $J_{\fin} \times \Z = \bigcup_{x \in \Z} \mA_x$.
By Proposition~\ref{prop:partition_orbits}(i), elements of $G_{\fin}$ stabilize the sets $\mA_x$ (for $x \in \Z$). 
Therefore 
\[
G_{\fin} \leq \bigoplus_{x \in \Z} \Sym(\mA_x).
\]
For each $x \in \Z$, under the bijection $\mA_x \rightarrow J_{\fin}\colon  (j, z) \mapsto j$, we have $\Sym(\mA_x) \simeq \Sym(J_{\fin})$.
Hence $G_{\fin}$ can be considered as a subgroup of
$
\bigoplus_{x \in \Z} \Sym(J_{\fin})
$.
By slight abuse of notation we again use $\rho$ to denote the action $(j, z) \mapsto (j, z+1)$ on $J_{\fin} \times \Z$.
Since $G_{\fin}$ is a $\rho$-invariant, we immediately obtain:
\begin{proposition}\label{prop:general_finite_orbit}
    $G_{\fin}$ is a $\rho$-invariant subgroup of $\bigoplus_{x \in \Z} H$ for the finite group $H = \Sym(J_{\fin})$, where $\rho$ acts on $\bigoplus_{x \in \Z} H$ as the shift $(h_k)_{k \in \Z} \mapsto (h_{k-1})_{k \in \Z}$.
\end{proposition}
This gives us part~(i) of Proposition~\ref{prop:structure_invariant_subgroup}.

\subsection{Permutation group on the infinite orbits}\label{subsec:infinite_orbits}
In this subsection we will characterize the group $G_{\infty}$.

\paragraph{Computing the $G_{\infty}$-orbits.}
Consider the union of $G_{\infty}$-orbits $J_{\infty} \times \Z$.
By slight abuse of notation we again use $\rho$ to denote the action $(j, z) \mapsto (j, z+1)$ on $J_{\infty} \times \Z$.
Let $\mO_1, \ldots, \mO_m$ respectively denote the $G_{\infty}$-orbits: they have been characterized by Proposition~\ref{prop:partition_orbits}.
For each $i = 1, \ldots, m$, let $G_i \leq \FSym(\mO_i)$ denote the group induced by the action of $G_{\infty}$ on $\mO_i$.
Then 
\[
G_{\infty} \leq G_1 \times \cdots \times G_m,
\]
such that each projection $G_{\infty} \rightarrow G_i, i = 1, \ldots, m$, is surjective.
The idea is to first compute each $G_i$, and then apply Goursat's lemma (Lemma~\ref{lem:Goursat}) iteratively to characterize $G_{\infty}$.

\paragraph{The transitive groups $G_i$.}

For $i = 1, \ldots, m$, we now characterize the group $G_i \leq \FSym(\mO_i)$.
Since there is only one orbit, $G_i$ acts transitively on $\mO_i$.
By Proposition~\ref{prop:partition_orbits}, we have $\rho^{\ell_i} \mO = \mO$ for some $\ell_i \geq 1$, and $G_i$ is $\rho^{\ell_i}$-invariant.
Let $\tau_1, \ldots, \tau_s$ be the generators of $G_{\infty}$ as a $\rho$-invariant group.
Then, the generators of $G_i$ as a $\rho^{\ell_i}$-invariant group are the permutations of $\mO_i$ induced by 
\[
\rho^{z}\tau_1\rho^{-z}, \ldots, \rho^{z}\tau_s\rho^{-z}, \quad 0 \leq z \leq \ell_i-1.
\]

We now fix an index $i$ and for brevity use $G, \mO, J, \ell,$ to denote $G_i, \mO_i, J_i, \ell_i$, so $G$ is a $\rho^{\ell}$-invariant subgroup of $\FSym(\mO)$.



\begin{lemma}\label{lem:effective_partition}
    We can effectively write $\mO$ as a partition $\Gamma \times \Delta \coloneqq \bigsqcup_{\delta \in \Delta} \Gamma \times \{\delta\}$, where the block $\Gamma$ is finite and the index set $\Delta$ is infinite, such that $G$ preserves this partition and induces a primitive action on $\Delta$.
    Furthermore, each block $\Gamma \times \{\delta\}$ contains at most one element in any $\gen{\rho^{\ell}}$-orbit of $\mO$.
\end{lemma}

\begin{proof}
    We first prove that blocks in non-trivial $G$-invariant partitions cannot be too large. 
    By re-indexing we can without loss of generality suppose
    \[
    \mO = \{(j, \ell x) \mid j \in J, x \in \Z\}
    \]
    for some finite index set $J$.
    Take any $\la\rho^\ell\ra$-invariant metric on $\mO$.
    For example one can take the metric defined by the distance 
    \begin{equation}\label{eq:distance}
        \dist((i, z_i), (j, z_j)) \coloneqq |z_i - z_j| + \frac{|i - j|}{L},
    \end{equation}
    where $L > \max_{i,j \in J}\{|i-j|\}$.
    For a set $S \subset \mO$, let $\diam(S)$ denote its diameter under this distance.

    Since $G$ is transitive, for each $j \in J$ we can effectively compute an element $h_j \in G$, such that $h_j(j, 0) = (j, \ell)$, and let 
    \[
    D=\max_{j\in J}\diam(\supp(h_j)).
    \]

\medskip

\begin{adjustwidth}{2mm}{2mm}
    \textbf{Claim.} Let $S$ be a block of a non-trivial $G$-invariant partition. Then $S$ is finite and $\diam S\le D$.

    \medskip

    \noindent$\triangleright$ Elements of $G$ have finite support, hence cannot permute infinite blocks. If $S$ were infinite, it would therefore be fixed by $G$. Using that $G$ is transitive, this would imply $S=\mO$, a contradiction.

    \medskip
    
    \noindent Fix $(j, z), (j', z') \in S$ with $\dist((j, z), (j', z'))$ maximal and $z\ge z'$. If $\dist((j, z), (j', z'))> D$, then the element $\rho^{z} h_j \rho^{-z} \in G$ maps $(j, z)$ to $(j, z+ \ell)$, but $\rho^{z} h_j \rho^{-z}(j', z') = (j', z')$ because the support of $\rho^{z} h_j \rho^{-z}$ has diameter at most $D$.
    Therefore the block $S$ must also contain $(j, z+ \ell)$, contradicting that $\dist((j,z),(j',z'))$ was maximal.
    \hfill$\triangleleft$
\end{adjustwidth}

\medskip

\noindent From the above claim, there is an upper bound $D$ on the size of blocks in non-trivial $G$-invariant partitions of $\mO$. Hence, there exists a coarsest\footnote{Here, $\rho^{\ell}$-invariance is essential for our proof. If the transitive group $G$ is not $\rho^{\ell}$-invariant, it might be \emph{totally imprimitive}~\cite{neumann1975lawlessness}, and the family of non-trivial partitions preserved by $G$ does not admit a coarsest element.} non-trivial partition (partition with largest blocks) of $\mO$ preserved by $G$.

\medskip

\begin{adjustwidth}{2mm}{2mm}
    \textbf{Claim.} The coarsest (proper) $G$-invariant partition is unique, and $\la\rho^\ell\ra$-invariant.
    
    \medskip

    \noindent $\triangleright$ Consider a coarsest partition $\mO=\bigsqcup_{\delta\in \Delta} \Gamma_\delta$ and a block $S$ in another non-trivial $G$-invariant partition.
    There exist three blocks $\Gamma_i,\Gamma_j,\Gamma_k$ with 
    \[ 
    \Gamma_i\cap S\ne\emptyset \quad\text{and}\quad \Gamma_j\cap S=\Gamma_k\cap S=\emptyset. 
    \]
    Note that $G$ acts primitively on $\Delta$, therefore one may apply \cite[Corollary~31.2]{neumann2023infinite}
    (or Lemma \ref{lem:transitive_finite} below) to deduce that there exists $g \in G$ such that $g(\Gamma_i)=\Gamma_j$, $g(\Gamma_j)=\Gamma_k$, $g(\Gamma_k)=\Gamma_i$, and $g$ fixes all other blocks $\Gamma_\delta$ pointwise.
    Since $S$ comes from another $G$-invariant partition, we conclude that $S\subseteq\Gamma_i$, hence the partition $\mO= \bigsqcup_{\delta\in\Delta} \Gamma_\delta$ is the \emph{unique} coarsest $G$-invariant partition. Since $\mO= \bigsqcup_{\delta\in\Delta} \rho^\ell(\Gamma_\delta)$ is another coarsest $G$-invariant partition, we deduce that the coarsest proper partition is $\la\rho^\ell\ra$-invariant. 
    \hfill $\triangleleft$
\end{adjustwidth}

\medskip

\noindent Consider a block $\Gamma$ of the coarsest $G$-invariant partition $\mO$. We observe that each $\la\rho^\ell\ra$-orbit intersects $\Gamma$ in at most one element. Indeed, if $o,\rho^{\ell d}(o)\in \Gamma$, then the pair $\rho^{n\ell d}(o),\rho^{(n+1)\ell d}(o)$ belongs to the same block of the partition for all $n\ge 0$, and this block is always $\Gamma$ by induction.

\medskip
Finally, for the effective search of the coarsest $G$-invariant partition, we just reduced ourself to $\la\rho^\ell\ra$-invariant partitions defined by a block $\Gamma$ of a diameter at most $D$. There are only finitely many such partitions, we can enumerate them, check for each one whether it is preserved by the generators of $G$, and select the coarsest one. This allows us to compute the partition $\mO = \Gamma \times \Delta$.
\end{proof}
\begin{figure}[h]
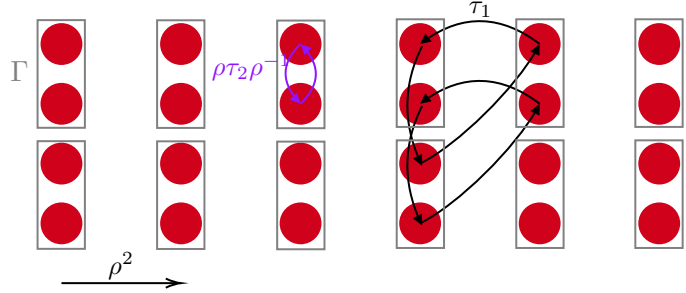

    \centering
    \figureinfiniteorbits
    \caption{Partition of an infinite orbit into blocks $\Gamma \times \Delta$. This orbit also appeared in Figure~\ref{fig:orbits}.}
    \label{fig:infinite_orbits}
\end{figure}

Since
\[
    \mO = \{(j, z_j + \ell x) \mid j \in J, x \in \Z\},
\]
and each block $\Gamma \times \{\delta\}$ contains at most one element in any $\gen{\rho^{\ell}}$-orbit of $\mO$, the index set $\Delta$ can be decomposed again as a product $S \times \ell\Z$ for some finite set $S = \{s_1, \ldots, s_t\}$.

Fix the element $(s_1, 0) \in \Delta$ and identify $\Gamma$ with its copy $\Gamma \times \{(s_1, 0)\}$.
Consider the set-wise stabilizer 
\[
G_{\{\Gamma\}} \coloneqq \{g \in G \mid g(\Gamma \times \{(s_1, 0)\}) = \Gamma \times \{(s_1, 0)\}\}.
\]
The group $G_{\{\Gamma\}}$ induces a permutation group on the set $\Gamma$, which we denote by $H \leq \Sym(\Gamma)$.
The following lemma shows that $G$ can be effectively embedded in a wreath product $W \coloneqq H \wr_{\Delta} \FSym(\Delta)$.

\begin{lemma}[{Generalization of~\cite[Lemma~2.6]{neumann1976structure}}]\label{lem:transitive_embed}
    Let $\mO = \Gamma \times \Delta$ be a partition of $\mO$ with $\Gamma$ finite and $\Delta$ infinite.
    Let $G$ be a subgroup of $\FSym(\mO)$ which preserves this partition and acts transitively on the set of blocks $\Delta$, and let $H \leq \Sym(\Gamma)$ be the permutation group induced on $\Gamma$ by its set-wise stabiliser.
    Then $G$ can be effectively embedded in the wreath product $W \coloneqq H \wr_{\Delta} \FSym(\Delta)$.
\end{lemma}
\begin{proof}
    We follow the proof of~\cite[Lemma~2.6]{neumann1976structure}, see Figure~\ref{fig:twist} for an illustration.
    Recall that the index set $\Delta$ can be decomposed as a product $S \times \ell\Z$ for some finite set $S = \{s_1, \ldots, s_t\}$.
    We identify $\Gamma$ with its copy $\Gamma \times \{(s_1, 0)\}$ in $\Gamma \times \Delta$.
    Since $G$ acts transitively on the set of blocks $\Delta$, we can compute elements $h, g_1, \ldots, g_t \in G$ such that $h$ sends the block $(s_1, 0)$ to $(s_1, \ell)$; and $g_i, i = 1, \ldots, t$ sends the block $(s_1, 0)$ to $(s_i, 0)$, with $g_1 = \id$.
    Then for $i = 1, \ldots, t$ and $x \geq 0$, the element
    \begin{equation}\label{eq:gix}
    g_{i, x} \coloneqq \rho^{x\ell}g_i\rho^{-x\ell} \cdot \rho^{(x-1)\ell}h\rho^{-(x-1)\ell} \cdot \cdots \cdot \rho^{\ell}h\rho^{-\ell} \cdot h = \rho^{x\ell}g_i \cdot (\rho^{-\ell} h)^x
    \end{equation}
    sends the block $(s_1, 0)$ to $(s_i, x\ell)$.
    A similar element can be defined for $x < 0$.
    Then, the set-wise stabilizer $G_{\{\Gamma \times \{s_i, x \ell\}\}}$ is conjugate to $G_{\{\Gamma \times \{s_1, 0\}\}}$ via $g_{i, x}$.

    We then identify the copy $\Gamma \times \{(s_i, x\ell)\}$ with $\Gamma$ through the map 
    \[
    g_{i, x}^{-1} \colon \Gamma \times \{(s_i, x\ell)\} \rightarrow \Gamma \times \{(s_1, 0)\} = \Gamma, \quad \gamma \mapsto g_{i, x}^{-1}(\gamma).
    \]
    This yields a map $\phi \colon \mO = \bigsqcup_{\delta \in \Delta} \Gamma \times \{\delta\} \rightarrow \Gamma \times \Delta,\; (\gamma, (s_i, x\ell)) \mapsto (g_{i, x}^{-1}(\gamma), (s_i, x \ell))$.
    Let $\Omega = \Gamma \times \Delta$ denote the image of $\phi$.
    Note that the identification of $\Omega$ with $\Gamma \times \Delta$ is different from the original product $\mO = \Gamma \times \Delta$.
    Under conjugation by $\phi$, the group $G$ becomes a partition-preserving permutation group of $\Omega = \Gamma \times \Delta$, which yields an embedding
    \[
    c_{\phi} \colon G \hookrightarrow \Sym(\Gamma) \wr_{\Delta} \FSym(\Delta).
    \]
    Under this embedding, the set-wise stabilizers $c_{\phi}(G)_{\{\Gamma \times \{s_i, x \ell\}\}}$ all induce the same permutation group $H \leq \Sym(\Gamma)$.
    Furthermore, for each $i, x$, the element $c_{\phi}(g_{i,x})$ sends $\Gamma \times \{(s_i, x\ell)\}$ to $\Gamma \times \{(s_1, 0)\}$ while preserving the order of the elements, i.e., $c_{\phi}(g_{i,x}) = ((h_{\delta})_{\delta \in \Delta}, g)$ with $h_{g(s_i, x\ell)} = \id$.
    
    We now show that
    \[
    c_{\phi}(G) \leq H \wr_{\Delta} \FSym(\Delta).
    \]
    Indeed, let $((h_{\delta})_{\delta \in \Delta}, g) \in c_{\phi}(G)$ and take any $(s_i, x \ell) \in \Delta$.
    Write $g(s_i, x \ell) = (s_{i'}, x' \ell)$, then the product $c_{\phi}(g_{i, x}) \cdot c_{\phi}(g_{i', x'})^{-1} \cdot ((h_{\delta})_{\delta \in \Delta}, g) \in c_{\phi}(G)$ stabilizes $\Gamma \times \{(s_i, x \ell)\}$ as a set, on which it induces a permutation $h_{g(s_i, x \ell)}$.
    Therefore $h_{g(s_i, x \ell)} \in H$.
    Applying this to all $(s_i, x \ell) \in \Delta$, we obtain that $((h_{\delta})_{\delta \in \Delta}, g) \leq H \wr_{\Delta} \FSym(\Delta)$.
    
    One may note that, since $G$ is $\rho^{\ell}$-invariant, its image $c_{\phi}(G)$ is invariant under the conjugated shift map $\phi \circ \rho^{\ell} \circ \phi^{-1}$.
    Furthermore, in Definition~\eqref{eq:gix}, the term $\rho^{-\ell} h$ permutes the set $\Gamma \times \{(s_1, 0)\}$.
    Therefore for $r = \card{\Gamma}!$, we have $g_{i,x+r}\vert_{\Gamma \times \{(s_1, 0)\}} = \rho^{r \ell} g_{i,x}\vert_{\Gamma \times \{(s_1, 0)\}}$ for all $1 \leq i \leq t, x \in \Z$.
    This $(\card{\Gamma}!)$-periodicity gives an effective representation of $\phi$ and $\phi \circ \rho^{\ell} \circ \phi^{-1}$, and shows that $(\phi \circ \rho^{\ell} \circ \phi^{-1})^{\card{\Gamma}!}$ is the same as $\rho^{\ell \card{\Gamma}!}$.

    Note that $H$ can be effectively computed. Let $f_1, \ldots, f_n \in \Sym(\Gamma) \wr_{\Delta} \FSym(\Delta)$ be the generators of $G$ as a $\rho^{\ell \card{\Gamma}!}$-invariant group, where $f_i = ((h_{i\delta})_{\delta \in \Delta}, g_i)$ with $h_{i\delta} \in \Sym(\Gamma)$.
    Since $f_i \in H \wr_{\Delta} \FSym(\Delta)$, the group $H$ is generated by the (finitely many non-zero) elements $h_{i\delta} \in \Sym(\Gamma)$.
    Altogether, we obtain an effective embedding of $G$ into $W \coloneqq H \wr_{\Delta} \FSym(\Delta)$, as well as an effective interpretation of the action of $\rho^{\ell}$ on $W$.
\end{proof}

    \begin{figure}[h!]
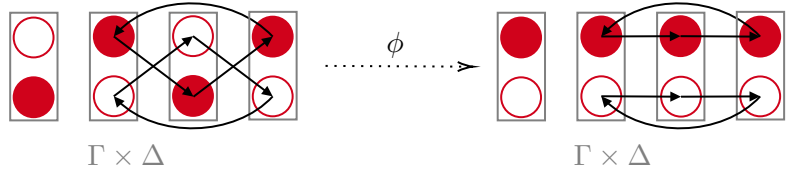

        \centering
        \figuretwist
        \caption{Re-identifying $\mO = \Gamma \times \Delta$ with $\Omega = \Gamma \times \Delta$ under the map $\phi$. Here, $H = \{\id\}$.}
        \label{fig:twist}
    \end{figure}

From now on we consider $G$ as a permutation group of $\Omega = \Gamma \times \Delta$ under the embedding $G \hookrightarrow H \wr_{\Delta} \FSym(\Delta)$ from Lemma~\ref{lem:transitive_embed}.
The group $G$ can now be represented as a $\rho^{\ell'}$-invariant group for some $\ell' > 0$ (for example, one can take $\ell' = \ell \card{\Gamma}!$ in the proof of Lemma~\ref{lem:transitive_embed}), given by a finite set of generators.

Recall that, for a group $K$, we denote by $K'$ its commutator subgroup, generated by the elements $[g, h]$ (for $g, h \in K$).
In particular $K' \unlhd K$ and the quotient $K/K'$ is abelian.

\begin{lemma}[{\cite[Theorem~3]{neumann1976structure}, see note below}]\label{lem:transitive_finite}
    Let $\Omega = \Gamma \times \Delta$ with $\Gamma$ finite and $\Delta$ infinite.
    Let $H$ be a group acting on $\Gamma$, and $W = H \wr_{\Delta} \FSym(\Delta)$.
    Then $W/W' \cong H/H' \times C_2$ where $C_2$ is the cyclic group of order two, and $((f_{\delta})_{\delta \in \Delta}, b) \in W'$ if and only if $b \in \FAlt(\Delta)$ and $\prod_{\delta \in \Delta} f_{\delta} \in H'$.
    Let $G$ be a subgroup of $W$ such that 
    \begin{enumerate}[nosep, label = {\normalfont(\roman*)}]
        \item $G$ acts primitively on the blocks $\Delta$,
        \item the permutation group on $\Gamma$ induced by the stabilizer $G_{\{\Gamma\}}$ is equal to $H$.
    \end{enumerate}
    Then we have $W' \leq G \leq W$.
\end{lemma}

Note that the original theorem of Neumann~\cite[Theorem~3]{neumann1976structure} is stated in the case where $G$ is transitive on $\Omega$ (i.e.\ $H$ transitive on $\Gamma$).
However, the proof in~\cite[Section~5]{neumann1976structure} does not use the transitivity of $G$ or $H$.
The statement $W/W' \cong H/H' \times C_2$ comes from~\cite[Lemma~2.7]{neumann1976structure}.
In particular, $GW'/W'$ is a subgroup of the abelian group $W/W' = H/H' \times C_2$.
Below is a simple example.

\begin{example}\label{example:derived_wreath}
    Let $\card{\Gamma} = 2$, $\Delta = \Z$, and $H$ be the cyclic group $C_2 = \{1, -1\}$ of two elements acting on $\Gamma$.
    In particular $H' = \{\id\}$.
    Then the group $W$ is the wreath product $C_2 \wr_{\Z} \FSym(\Z)$.
    An element $(f, b) \in W$ is in $W'$ if and only if $b \in \FAlt(\Z)$ and $f = (f_{\delta})_{\delta \in \Z}$ where $\prod_{\delta \in \Delta} f_{\delta} = 1$, that is, the number of $\delta \in \Z$ with $f_{\delta} = -1$ is even.
    The quotient $W/W'$ is isomorphic to $C_2 \times C_2$, where the projection map $W \rightarrow W/W' = C_2 \times C_2$ is $(f, b) \mapsto (\prod_{\delta \in \Delta} f_{\delta},\; \sgn(b))$.
    If $G \leq W$ satisfies the two conditions in Lemma~\ref{lem:transitive_finite}, then $G = AW'$ where $A$ is a subgroup of $C_2 \times C_2$.
\end{example}

\paragraph{Normal subgroups of $G_i$.}

We now characterize normal subgroups of $G$: this will be important for subsequent application of Goursat's lemma.
Let $H$ be a finite group and $\Delta$ be an infinite set.
Denote the projection 
\[
\pi_{\Delta} \colon H \wr_{\Delta} \FSym(\Delta) \rightarrow \FSym(\Delta)
\]
and let $W = H \wr_{\Delta} \FSym(\Delta)$.

\begin{proposition}\label{cor:N_contains_Wprime}
    Let $W' \leq G \leq W$ and $N$ be a normal subgroup of $G$. Then either $\pi_{\Delta}(N) = \{\id\}$ or $N\ge W'$.
    In the first case $\card{G/N}$ is infinite, and in the second case $\card{G/N} < \infty$.
\end{proposition}
\begin{proof}
The first statement of the proposition is a combination of the following two claims. The first claim is a variation on the \say{double commutator lemma}, which goes back to Higman \cite{Higman}.
\medskip

\begin{adjustwidth}{2mm}{2mm}
    \textbf{Claim.} Let $W' \le G \le W$. Suppose that $N \unlhd G$, then either $\pi_{\Delta}(N) = \{\id\}$ or $N \ge G'$.
    \medskip
    
    \noindent$\triangleright$ We first observe that $\FAlt(\Delta)\le\pi_{\Delta}(G)\le\FSym(\Delta)$ since $W'\le G \le W$. Since $\pi_{\Delta}(N)$ is a normal subgroup of $\pi_{\Delta}(G)$, the intersection $\pi_\Delta(N)\cap\FAlt(\Delta)$ is either $\{\id\}$ or $\FAlt(\Delta)$. In the former case, we have $\pi_{\Delta}(N) = \{\id\}$. Suppose that we are in the later, i.e., $\pi_{\Delta}(N) \ge \FAlt(\Delta)$.
    \medskip
    
    \noindent Given $g_1,g_2 \in G$, we prove that $[g_1,g_2] \in N$. Since $\pi_{\Delta}(N) \ge \FAlt(\Delta)$, there is an element $n \in N$, such that the support $\supp(ng_1^{-1}n^{-1})=n(\supp(g_1^{-1}))$ is disjoint from $\supp(g_2)$ and hence $[ng_1^{-1}n^{-1},g_2]= \id$. Since $[g,n]=gng^{-1}\cdot n^{-1}\in N$ for all $g\in G$ and $n\in N$, we have
    \[ [g_1,g_2] = g_1[ng_1^{-1}n^{-1},g_2]g_2g_1^{-1}g_2^{-1}=\bigl[[g_1,n],g_2\bigr]\in N.\]
    Alternatively $[g_1, g_2]\equiv [ng_1n^{-1}, g_2]=\id \pmod N$.
    We conclude that $N \geq G'$. \hfill$\triangleleft$
\end{adjustwidth}
\medskip
\noindent Since $W' \le G$, the case $N \ge G'$ already yields $N \ge G' \ge W''$. It remains to prove $W''=W'$.
\medskip
\begin{adjustwidth}{2mm}{2mm}
    \textbf{Claim.} We have $W'' = W'$.
    \medskip
    
    \noindent$\triangleright$ We exhibit different elements that belong to $W''$. 
    For $h \in H$ and $\delta \in \Delta$, we let $h@\delta$ denote the element in $\bigoplus_{\Delta} H$ whose $\delta$-coordinate is $h$ and whose all other coordinates are $\id$.
    \begin{itemize}[leftmargin=6mm]
        \item First $W''\ge \FAlt(\Delta)''=\FAlt(\Delta)$ since $W\ge\FAlt(\Delta)$ and $\FAlt(\Delta)' = \FAlt(\Delta)$ (Lemma~\ref{lem:transitive_finite}).
        \item For all $h_1,h_2\in H$ and $\delta\in \Delta$, we have $[h_1,h_2]@\delta\in W''$. Indeed, by Lemma~\ref{lem:transitive_finite}, we have
        \[ [h_1,h_2]@\delta = \bigl[ (h_1@\delta)\cdot (h_1^{-1}@\delta_1),\; (h_2@\delta)\cdot (h_2^{-1}@\delta_2)\bigr] \in [W',W']=W'' \]
        for any $\delta_1,\delta_2\in\Delta$ distinct from each other and distinct from $\delta$.
        \item For all $h\in H$ and distinct $\delta_1,\delta_2\in \Delta$, we have $(h@\delta_1)\cdot (h^{-1}@\delta_2) \in W''$. Indeed, we have
        \[ (h@\delta_1)\cdot (h^{-1}@\delta_2) = \bigl[ (h@\delta_1)\cdot (h^{-1}@\delta_3),\, (\delta_1,\delta_2,\delta_4)\bigr]\in [W',\FAlt(\Delta)]\le W''\]
        for any $\delta_3,\delta_4\in\Delta$ distinct from each other and distinct from $\delta_1,\delta_2$.
    \end{itemize}
    Finally, it is easy to see that $W'$ is generated by these three types of elements. (The first two types of elements generate the subgroup $H'\wr_\Delta\FAlt(\Delta)\ge \bigoplus_\Delta H'$, then we can argue in the quotient $W/\bigoplus_\Delta H'\simeq (H/H')\wr_\Delta\FAlt(\Delta)$.) We conclude that $W''\ge W'$ hence $W''=W'$. \hfill$\triangleleft$
\end{adjustwidth}
\medskip
    
    \noindent 
    In case $\pi_{\Delta}(N) = \{\id\}$, the quotient $G/N$ induces a permutation group of $\Delta$ that contains $\FAlt(\Delta)$, therefore $\card{G/N}$ is infinite.
    In case $N \geq W'$, we have $\card{G/N} \leq \card{W/W'} < \infty$ by Lemma~\ref{lem:transitive_finite}.
\end{proof}

\paragraph{The subdirect product $G_{\infty} \leq G_1 \times \cdots \times G_m$.}

We now compute the group 
\[
G_{\infty} \leq G_1 \times \cdots \times G_m
\]
Recall that the projections $G_{\infty} \rightarrow G_i, i = 1, \ldots, m$, are surjective, and we have already characterized the groups $G_i$ as well as their normal subgroups.
Here, each $G_i$ is a $\rho^{\ell_i}$-invariant group for some $\ell_i > 0$.
To compute $G_{\infty}$, we will iteratively apply Goursat's lemma.
First, we need a generalization of a classic result by H\"{o}lder~\cite{holder1895bildung}:

\begin{lemma}[{\cite[Section~11.4.6]{scott1987group}}]\label{lem:no_outer}
    Let $\Delta, \Delta'$ be two countably infinite sets, and $\varphi \colon \FAlt(\Delta) \rightarrow \FAlt(\Delta')$ or $\varphi \colon \FSym(\Delta) \rightarrow \FSym(\Delta')$ be an isomorphism.
    Then there exists an isomorphism $\xi \colon \Delta \rightarrow \Delta'$ such that $\varphi(g) = \xi g \xi^{-1}$ for all $g$.\footnote{Lemma~\ref{lem:no_outer} also holds for finite sets $\Delta, \Delta'$ with $\card{\Delta} = \card{\Delta'} \neq 6$.
    Remarkably, when $\card{\Delta} = 6$, there are automorphisms of $\Sym(\Delta)$ that are not conjugations~\cite{holder1895bildung}.}
\end{lemma}

For any index $i = 1, \ldots, m$, let 
\[
W_i = H_i \wr_{\Delta_i} \FSym(\Delta_i)
\]
be a wreath product where $\Delta_i$ is countably infinite and $H_i \leq \Sym(\Gamma_i)$ with finite $\Gamma_i$.
By Lemma~\ref{lem:transitive_finite}, the quotients $W_i/W'_i$ is finite.

\begin{lemma}\label{lem:diagonal_or_fi_two}
    Let $G$ be a subgroup of $W_i \times W_j$, and let $G_i, G_j$ denote respectively the projection of $G$ on the components $W_i$ and $W_j$.
    Suppose $G_i \geq W'_i$ and $G_j \geq W'_j$.
    Then $G$ falls into one of the two following cases:
    \begin{enumerate}[nosep, label={\normalfont(\roman*)}]
        \item either $G \geq W'_i \times W'_j$, 
        \item or there exists a bijection $\xi \colon \Delta_i \rightarrow \Delta_j$, such that $G$ acts on the blocks $\Delta_i \sqcup \Delta_j$ as the group $\{(g, \xi g \xi^{-1}) \mid g \in \FAlt(\Delta_i)\}$ or $\{(g, \xi g \xi^{-1}) \mid g \in \FSym(\Delta_i)\}$.
    \end{enumerate}
    In the second case, under the bijection $\xi^{-1} \colon \Delta_j \rightarrow \Delta_i$, the set $(\Gamma_i \times \Delta_i) \sqcup (\Gamma_j \times \Delta_j)$ can be identified with $(\Gamma_i \sqcup \Gamma_j) \times \Delta_i$, on which $G$ acts faithfully as a subgroup of $\Sym(\Gamma_i \sqcup \Gamma_j) \wr_{\Delta_i} \FSym(\Delta_i)$.
\end{lemma}
\begin{proof}
    See Figures~\ref{fig:mix} and \ref{fig:mixed} for an illustration of case~(i); see Figures~\ref{fig:merge} and \ref{fig:merged} for an illustration of case~(ii).

\begin{figure}[h]
    \centering
    \begin{minipage}[t]{.45\textwidth}
        \centering
        \figuremix
        \caption{Case~(i) of Lemma~\ref{lem:diagonal_or_fi_two} (see Figure~\ref{fig:mixed} for justification): $G$ is the group $\gen{\tau}_{\rho}$, where $\tau$ acts as a 2-cycle on $\Delta_i$ and a 3-cycle on $\Delta_j$.}
        \label{fig:mix}
    \end{minipage}
    \hfill
    \begin{minipage}[t]{.45\textwidth}
        \centering
        \figuremixed
        \caption{$\tau^3$ acts as a 2-cycle on $\Delta_i$ and trivially on $\Delta_j$, so $\gen{\tau^3}_{\rho} = \FSym(\Delta_i) \times \{\id\}$. Also, $\tau^2$ acts trivially on $\Delta_i$ and as a 3-cycle on $\Delta_j$, so $\gen{\tau^2}_{\rho} = \{\id\} \times \FAlt(\Delta_j)$. Therefore $\gen{\tau}_{\rho} \geq W'_i \times W'_j$. \vspace{5mm}}
        \label{fig:mixed}
    \end{minipage}
    \centering
    \centering
    \begin{minipage}[t]{.45\textwidth}
        \centering
        \figuremerge
        \caption{Case~(ii) of Lemma~\ref{lem:diagonal_or_fi_two}: $G$ acts diagonally on the pair $(\Delta_i, \Delta_j)$}
        \label{fig:merge}
    \end{minipage}
    \hfill
    \begin{minipage}[t]{.45\textwidth}
        \centering
        \figuremerged
        \caption{Identification with a subgroup of $\Sym(\Gamma_i \sqcup \Gamma_j) \wr_{\Delta_i} \FSym(\Delta_i)$.}
        \label{fig:merged}
    \end{minipage}
\end{figure}
    
    Consider the surjective projections 
    \[
    p_i \colon G \rightarrow G_i, \;\; p_j \colon G \rightarrow G_j,
    \]and denote 
    \[
    N_i \coloneqq p_i(\ker(p_j)) \unlhd G_i, \;\; N_j \coloneqq p_j(\ker(p_i)) \unlhd G_j.
    \]
    By Goursat's Lemma (Lemma~\ref{lem:Goursat}), there is an isomorphism
    \[
    \varphi \colon G_i/N_i \xrightarrow{\sim} G_j/N_j,
    \]
    such that $G = \{(g_i, g_j) \mid \varphi(g_iN_i) = g_j N_j\}$.
    By Proposition~\ref{cor:N_contains_Wprime}, we have either $N_i \geq W'_i$, or $\pi_{\Delta_i}(N_i) = \{\id\}$.
    
    If $N_i \geq W'_i$, then the quotient $G_i/N_i$ is finite since $W_i/W'_i$ is finite. 
    Therefore $G_j/N_j \cong G_i/N_i$ is also finite. So by Proposition~\ref{cor:N_contains_Wprime} we must have $N_j \geq W'_j$.
    This yields 
    \[
    G \geq \ker(p_i) = \{0\} \times N_j \geq \{0\} \times W'_j, \;\; \text{ and } \;\; G \geq \ker(p_j) = N_i \times \{0\} \geq W'_i \times \{0\}.
    \]
    Therefore $G \geq W'_i \times W'_j$ and we obtain case~(i).

    If $\pi_{\Delta_i}(N_i) = \{\id\}$, then $N_i$ is a subgroup of $\bigoplus_{\Delta_i}H_i \unlhd W_i$.
    Since $G_i/N_i$ and $G_j/N_j$ are infinite, $N_j$ cannot contain $W'_j$.
    Apply Proposition~\ref{cor:N_contains_Wprime} to $N_j \unlhd G_j$, we obtain $\pi_{\Delta_j}(N_j) = \{\id\}$.
    Consider the image of the normal subgroup of infinite index $(G_i \cap \bigoplus_{\Delta_i}H_i)/N_i \unlhd G_i/N_i$ under the map $\varphi$: it is a normal subgroup of $G_j/N_j$ of infinite index.
    So by Proposition~\ref{cor:N_contains_Wprime}, we have $\pi_{\Delta_j}\Big(\varphi\left(\big(G_i \cap \bigoplus_{\Delta_i}H_i\big)/N_i\right) \cdot N_j\Big) = \{\id\}$, that is,
    \[
    \textstyle
    \varphi\left(\big(G_i \cap \bigoplus_{\Delta_i}H_i\big)/N_i\right) \leq \left(G_j \cap \bigoplus_{\Delta_j}H_j\right)/N_j.
    \]
    Similarly, we have 
    $
    \textstyle
    \varphi^{-1}\left(\big(G_j \cap \bigoplus_{\Delta_j}H_j\big)/N_j\right) \leq \big(G_i \cap \bigoplus_{\Delta_i}H_i\big)/N_i.
    $
    Therefore 
    \[
    \textstyle
    \varphi\left(\big(G_i \cap \bigoplus_{\Delta_i}H_i\big)/N_i\right) = \left(G_j \cap \bigoplus_{\Delta_j}H_j\right)/N_j.
    \]
    Hence, $\varphi$ induces an isomorphism
    \[
    \textstyle
    \overline\varphi \colon G_i/\left(G_i \cap \bigoplus_{\Delta_i}H_i\right) \xrightarrow{\sim} G_j/\left(G_j \cap \bigoplus_{\Delta_j}H_j\right).
    \]
    Note that $G_i/(G_i \cap \bigoplus_{\Delta_i}H_i)$ acts faithfully and finitarily on the blocks $\Delta_i$, so it can be considered as a subgroup of $\FSym(\Delta_i)$.
    Furthermore, since $G_i \geq W'_i$ we have $G_i/(G_i \cap \bigoplus_{\Delta_i}H_i) \geq \FAlt(\Delta_i)$.
    Similarly $G_j/(G_j \cap \bigoplus_{\Delta_j}H_j)$ can be considered as a subgroup of $\FSym(\Delta_j)$ containing $\FAlt(\Delta_j)$.
    By the isomorphism $\overline\varphi$, the two groups are either both $\FSym$ or both $\FAlt$.
    By Lemma~\ref{lem:no_outer}, we have $\overline\varphi(g) = \xi g \xi^{-1}$ for some bijection $\xi \colon \Delta_i \xrightarrow{\sim} \Delta_j$.
    Now, the group $G = \{(g_i, g_j) \mid \varphi(g_iN_i) = g_j N_j\}$ acts on the blocks $\Delta_i \sqcup \Delta_j$ as the group 
    \[
    \textstyle
    \left\{(g_i, g_j) \;\middle|\; \overline\varphi(g_i) = g_j, g_i \in G_i/(G_i \cap \bigoplus_{\Delta_i}H_i) \right\} = \left\{(g, \xi g \xi^{-1}) \;\middle|\; g \in G_i/(G_i \cap \bigoplus_{\Delta_i}H_i) \right\},
    \]
    where $G_i/(G_i \cap \bigoplus_{\Delta_i}H_i) = \FAlt(\Delta_i)$ or $\FSym(\Delta_i)$.
    We obtain case~(ii).

    In case~(ii), the group $G$ acts on $(\Gamma_i \times \Delta_i) \sqcup (\Gamma_j \times \Delta_j)$ in such a way that, if an element induces a permutation $\tau \in \FSym(\Delta_i)$ of the blocks $\Delta_i$, then it must induce the permutation $\xi\tau \xi^{-1}$ on the blocks $\Delta_j$.
    Therefore, we can identify $(\Gamma_i \times \Delta_i) \sqcup (\Gamma_j \times \Delta_j)$ with $(\Gamma_i \sqcup \Gamma_j) \times \Delta_i$ under the map
    \begin{align*}
    (\gamma, \delta) \mapsto
        \begin{cases}
            (\gamma, \delta) \quad & (\gamma, \delta) \in \Gamma_i \times \Delta_i \\ 
            (\gamma, \xi^{-1}(\delta)) \quad & (\gamma, \delta) \in \Gamma_j \times \Delta_j.
        \end{cases}
    \end{align*}
    Then elements of $G$ acts on $(\Gamma_i \sqcup \Gamma_j) \times \Delta_i$ while preserving the partition.
    By Lemma~\ref{lem:transitive_embed}, this gives an embedding of $G$ as a subgroup of $\Sym(\Gamma_i \sqcup \Gamma_j) \wr_{\Delta_i} \FSym(\Delta_i)$.
\end{proof}

In the second case of Lemma~\ref{lem:diagonal_or_fi_two}, we say that $G \leq W_i \times W_j$ acts \emph{diagonally} on the pair $(\Delta_i, \Delta_j)$.
In general, let $G$ be a subgroup of $W_1 \times \cdots \times W_k$.
We say that $G$ acts \emph{diagonally} on the pair $(\Delta_i, \Delta_j)$ if the projection of $G$ on the coordinates $i, j$ acts diagonally on $(\Delta_i, \Delta_j)$.
In Proposition~\ref{prop:general_infinite_orbits} we will provide an effective way of determining whether $G$ acts diagonally on $(\Delta_i, \Delta_j)$ and, if so, compute the identification map $\xi \colon \Delta_i \rightarrow \Delta_j$.
Before that, let us first generalize Lemma~\ref{lem:diagonal_or_fi_two} to more than two indices:

\begin{lemma}\label{lem:diagonal_or_fi_k}
    Let $k \geq 1$, $G$ be a subgroup of $W_1 \times \cdots \times W_k$ with $W_i = H_i \wr_{\Delta_i} \FSym(\Delta_i), i = 1, \ldots, k$, and let $G_i$ denote the projection of $G$ on the component $W_i$.
    Suppose $G_i \geq W'_i$ for all $1 \leq i \leq k$.
    Then either there exist two indices $i, j$ such that $G$ acts diagonally on the pair $(\Delta_i,\Delta_j)$, or $G \geq W'_1 \times \cdots \times W'_k$.
\end{lemma}
\begin{proof}    
    We prove the lemma by induction on $k$.
    For $k = 1$ the claim is obvious, and for $k = 2$ the claim is simply Lemma~\ref{lem:diagonal_or_fi_two}.
    For the induction step let $k \geq 3$.
    Take $G \leq W_1 \times \cdots \times W_k$ whose projections $G_i \leq W_i$ satisfy $G_i \geq W'_i$.
    Consider the projection maps 
    \[
    p \colon G \rightarrow W_1 \times \cdots \times W_{k-1}, \;\; \text{ and } \;\; p' \colon G \rightarrow G_k.
    \]
    Consider $p'(\ker(p)) \unlhd G_k$.
    By Goursat's Lemma, we have
    \[
    p(G)/p(\ker(p')) \cong G_k/p'(\ker(p)).
    \]
    Here, $p(G)$ is a subgroup of $W_1 \times \cdots \times W_{k-1}$, whose projection on $W_i$ is equal to $G_i$ for $i = 1, \ldots, k-1$.
    If $p(G)$ is diagonal on some pair $(\Delta_i, \Delta_j), 1 \leq i < j \leq k-1$, then $G$ is also diagonal on the same pair.
    Otherwise by the induction hypothesis we have 
    \[
    p(G) \geq W'_1 \times \cdots \times W'_{k-1},
    \]
    which yields finiteness of the index 
    \[
    [W_1 \times \cdots \times W_{k-1} \colon p(G)] \leq [W_1 \times \cdots \times W_{k-1} \colon W'_1 \times \cdots \times W'_{k-1}] < \infty.
    \]
    By Proposition~\ref{cor:N_contains_Wprime}, there are two cases: either $p'(\ker(p)) \geq W'_k$, or $\pi_{\Delta_k}\left(p'(\ker(p))\right) = \{\id\}$.

    \textbf{Case one.} If $p'(\ker(p)) \geq W'_k$ then the quotients
    $
    p(G)/p(\ker(p')) \cong G_k/p'(\ker(p))
    $
    are finite. 
    The projection of $p(\ker(p')) \unlhd p(G)$ on each component $W_i, 1 \leq i \leq k-1$, must be a finite index normal subgroup of $G_i$ and hence contains $W'_i$ (by Proposition~\ref{cor:N_contains_Wprime}).
    We can therefore apply the induction hypothesis on the subgroup $p(\ker(p'))$ of $W_1 \times \cdots \times W_{k-1}$ and obtain that either $p(\ker(p'))$ acts diagonally on some pair $(\Delta_i, \Delta_j), 1 \leq i < j \leq k-1$, or $p(\ker(p')) \geq W'_1 \times \cdots \times W'_{k-1}$.
    But since the index 
    \[
    [W_1 \times \cdots \times W_{k-1} \colon p(\ker(p'))] = [W_1 \times \cdots \times W_{k-1} \colon p(G)] [p(G) \colon p(\ker(p'))]
    \]
    is finite, the group $p(\ker(p'))$ cannot act diagonally on any pair $(\Delta_i, \Delta_j), 1 \leq i < j \leq k-1$.
    This shows that $p(\ker(p')) \geq W'_1 \times \cdots \times W'_{k-1}$.
    Therefore, $G$ contains $\ker(p) \geq \{\id\} \times  W'_k$ as well as $\ker(p') \geq W'_1 \times \cdots \times W'_{k-1} \times \{\id\}$, so $G$ contains $W'_1 \times \cdots \times W'_{k}$.

    \textbf{Case two.} If $\pi_{\Delta_k}\left(p'(\ker(p))\right) = \{\id\}$.
    Consider the normal subgroup $p(\ker(p')) \unlhd p(G)$.
    For $i = 1, \ldots, k-1$, the projection $\left(p(\ker(p'))\right)_{i}$ of $p(\ker(p'))$ on the component $G_i$ is a normal subgroup, so either $\pi_{\Delta_i}\left(p(\ker(p'))\right) = \{\id\}$ or $\left(p(\ker(p'))\right)_{i} \geq W'_i$ (by Proposition~\ref{cor:N_contains_Wprime}).
    If $\pi_{\Delta_i}\left(p(\ker(p'))\right) = \{\id\}$ for some $1 \leq i \leq k-1$, then $G$ acts diagonally on the pair $(\Delta_i, \Delta_k)$.
    Otherwise $\left(p(\ker(p'))\right)_{i} \geq W'_i$ for all $1 \leq i \leq k-1$ and we can apply the induction hypothesis on the subgroup $p(\ker(p'))$ of $W_1 \times \cdots \times W_{k-1}$.
    Since the quotient $p(G)/p(\ker(p')) \cong G_k/p'(\ker(p)$ is infinite, we cannot have $p(\ker(p') \geq W'_1 \times \cdots \times W'_{k-1}$.
    Therefore $p(\ker(p'))$ must act diagonally on the pair $(\Delta_i, \Delta_j)$ for some $1 \leq i < j \leq k-1$.
    We will show that this is impossible since $p(\ker(p'))$ is normal in $p(G)$.

    Take any $a \in p(\ker(p'))$ such that $\pi_{\Delta_i}(a) \neq \id$.
    Since $p(G) \geq W'_1 \times \cdots \times W'_{k-1}$, we can pick an element $h \in p(G)$ such that $\pi_{\Delta_i}(h)$ does not commute with $\pi_{\Delta_i}(a)$ but $\pi_{\Delta_j}(h) = \id$.
    Since $p(\ker(p'))$ is normal in $p(G)$, we have $hah^{-1} \in p(\ker(p'))$, and $hah^{-1}$ acts on $\Delta_i$ as $\pi_{\Delta_i}(h)\pi_{\Delta_i}(a)\pi_{\Delta_i}(h)^{-1} \neq \pi_{\Delta_i}(a)$, and acts on $\Delta_j$ as $\pi_{\Delta_j}(a)$.
    This means that $a, hah^{-1} \in p(\ker(p'))$ act differently on $\Delta_i$ but the same on $\Delta_j$, contradicting the claim that $p(\ker(p'))$ acts diagonally on the pair $(\Delta_i, \Delta_j)$.

    We have thus showed that in all cases, either there exist two indices $i, j$ such that $G$ acts diagonally on the pair $(\Delta_i,\Delta_j)$, or $G \geq W'_1 \times \cdots \times W'_k$.
    This concludes the induction step.
\end{proof}

The above lemmas allow us to obtain an effective characterization of $G_{\infty}$:

\begin{proposition}\label{prop:general_infinite_orbits}
    We can effectively compute a decomposition $J_{\infty} \times \Z = \bigsqcup_{i = 1}^k (\Gamma_i \times \Delta_i)$ where each $\Gamma_i$ is finite and $\Delta_i$ is infinite, as well as $H_i \leq \Sym(\Gamma_i)$,
    such that $W'_1 \times \cdots \times W'_k \leq G_{\infty} \leq W_1 \times \cdots \times W_k$, where $W_i = H_i \wr_{\Delta_i} \FSym(\Delta_i)$.
\end{proposition}
\begin{proof}
Recall that $G_{\infty}$ acts on the set $J_{\infty} \times \Z$ in such a way that all the $G_{\infty}$-orbits $\mO_1, \ldots, \mO_m$ are infinite.
By Lemma~\ref{lem:effective_partition}, the orbit $\mO_i$ can be effectively written as a product $\Gamma_i \times \Delta_i$ for finite $\Gamma_i$ and infinite $\Delta_i$, such that $G$ acts primitively (and hence transitively) on the set of blocks $\Delta_i$.
Under this identification, the action of $G_{\infty}$ on $\mO_i$ induces a group $G_i \leq W_i \coloneq H_i \wr_{\Delta_i} \FSym(\Delta_i)$, so $G_{\infty} \leq W_1 \times \cdots \times W_m$.
By Lemma~\ref{lem:transitive_finite}, we have $G_i \geq W'_i$.

We first give a procedure to decide whether $G_{\infty}$ acts diagonally on a pair $(\Delta_i, \Delta_j)$.
Suppose $G_{\infty}$ acts diagonally on the pair $(\Delta_i, \Delta_j)$.
By Lemma~\ref{lem:diagonal_or_fi_two} we can identify the projection $G_{\{i, j\}}$ of $G_{\infty}$ on the component $W_i \times W_j$ as a subgroup of $\Sym(\Gamma_i \sqcup \Gamma_j) \wr_{\Delta_i} \FSym(\Delta_i)$, under some bijection $\xi \colon \Delta_i \rightarrow \Delta_j$. 
Let $G_i, G_j$ be respectively the projection of $G_{\{i, j\}}$ on $W_i, W_j$.
Recall that $G_i, G_j$ are respectively $\rho^{\ell_i}$ and $\rho^{\ell_j}$-invariant.
Let $\ell = \ell_i \ell_j$, so $G_{\infty}, \Delta_i, \Delta_j$ are $\rho^{\ell}$-invariant.

Since $G_{\infty}$ acts on the pair $(\Delta_i, \Delta_j)$ as $(g, \xi g \xi^{-1}), g \in \FAlt(\Delta_i)$ or $\FSym(\Delta_i)$, by $\rho^{\ell}$-invariance it also acts as $(\rho^{\ell}g\rho^{-\ell}, \rho^{\ell} \xi g \xi^{-1} \rho^{-\ell})$.
Therefore $(\rho^{\ell}g\rho^{-\ell}, \rho^{\ell} \xi g \xi^{-1} \rho^{-\ell}) = (\rho^{\ell}g\rho^{-\ell}, \xi \rho^{\ell} g \rho^{-\ell} \xi^{-1})$ for all $g \in \FAlt(\Delta_i)$ or $\FSym(\Delta_i)$.
Consequently, the identification map $\xi \colon \Delta_i \rightarrow \Delta_j$ is $\rho^{\ell}$-invariant (i.e., $\rho^{\ell} \xi = \xi \rho^{\ell}$).
Given any $\rho^{\ell}$-invariant metric of $\bigsqcup_{i=1}^k \Delta_i$ (such as the one defined by Equation~\eqref{eq:distance}), the distances $\dist(\delta_i, \xi(\delta_i)), \delta_i \in \Delta_i$, are bounded by the maximal diameter of the generators of $G_{\infty}$.
Note that there are only finitely many $\rho^{\ell}$-invariant bijections $\xi \colon \Delta_i \rightarrow \Delta_j$ with a given bound on $\sup_{\delta_i \in \Delta_i}\dist(\delta_i, \xi(\delta_i))$.
Therefore in order to decide whether $G_{\infty}$ acts diagonally on the pair $(\Delta_i, \Delta_j)$, it suffices to verify whether any of these bijections produces $G_{\{i, j\}}$.

If $G_{\infty}$ acts diagonally on some pair $(\Delta_i, \Delta_j), 1\leq i < j \leq k$, then by Lemmas~\ref{lem:transitive_embed} and \ref{lem:transitive_finite}, the group $G_{\{i, j\}}$ can be effectively embedded in a wreath product $W_{\{i,j\}} \coloneqq H \wr_{\Delta_i} \FSym(\Delta_i)$ for some effectively computable $H \leq \Sym(\Gamma_i \sqcup \Gamma_j)$, such that $G_{\{i, j\}} \geq W'_{\{i,j\}}$.
Therefore, we have $G_{\infty} \leq W_1 \times \cdots \times W_{\{i,j\}} \times \cdots \times W_m$, so we can replace $\Gamma_i$ with $\Gamma_i \sqcup \Gamma_j$ and remove the index $j$.

We can repeat the above procedure until for any two indices $i \neq j$, the group $G$ does not act diagonally on $(\Delta_i, \Delta_j)$.
Up to renaming the indices we can write $G_{\infty} \leq W_1 \times \cdots \times W_k$.
Then by Lemma~\ref{lem:diagonal_or_fi_k}, we have $G_{\infty} \geq W'_1 \times \cdots \times W'_k$.
\end{proof}

This gives us part~(ii) of Proposition~\ref{prop:structure_invariant_subgroup}.

\subsection{Membership Problem in $\rho$-invariant subgroups}\label{subsec:membership}

In this subsection we give a complete characterization for $\rho$-invariant subgroups of $\FSym(\{1, \ldots, d\} \times \Z)$.
We then give a procedure to decide their membership problem by reducing to Subgroup Membership in wreath products $H \wr \Z$.
The complete characterization for $\rho$-invariant subgroups of $\FSym(\{1, \ldots, d\} \times \Z)$ is given by Proposition~\ref{prop:structure_invariant_subgroup}, which we can now prove using Propositions~\ref{prop:general_finite_orbit} and~\ref{prop:general_infinite_orbits}:

\propstructure*
\begin{proof}
    Recall that $G_{\fin}$ denotes the group induced by $G$ on the union of finite orbits, and $G_{\infty}$ denotes the group induced by $G$ on the union of infinite orbits.
    Then $G$ is a subgroup of $G_{\fin} \times G_{\infty}$, and both projection maps $G \rightarrow G_{\fin}$ and $G \rightarrow G_{\infty}$ are surjective.
    
    By Proposition~\ref{prop:general_finite_orbit} we can effectively write $G_{\fin}$ as a $\rho$-invariant subgroup of $\bigoplus_{\Z} H$ for some finite $H$, where $\rho_{\fin}$ acts as the shift $(h_k)_{k \in \Z} \mapsto (h_{k-1})_{k \in \Z}$.
    
    If $G$ has no infinite orbit then $G_{\infty}$ is trivial.
    Otherwise by Proposition~\ref{prop:general_infinite_orbits}, we have $W'_1 \times \cdots \times W'_k \leq G_{\infty} \leq W_1 \times \cdots \times W_k$ with $W_i$ defined as in the proposition.
    Let $W = W_1 \times \cdots \times W_k$, one can easily see that $W' = W'_1 \times \cdots \times W'_k$.
    
    We now use Goursat's Lemma on $G \leq G_{\fin} \times G_{\infty}$.
    This gives us normal subgroups $N_{\fin} \unlhd G_{\fin}$, $N_{\infty} \unlhd G_{\infty}$, and an isomorphism $\varphi \colon G_{\fin}/N_{\fin} \xrightarrow{\sim} G_{\infty}/N_{\infty}$, such that 
    \[
    G = \{(g_{\fin}, g_{\infty}) \mid \varphi(g_{\fin}N_{\fin}) = g_{\infty} N_{\infty}\}.
    \]
    In particular, $G \geq \{\id\} \times N_{\infty}$.
    We now show that $N_{\infty} \geq W'$.

    The quotient $G_{\fin}/N_{\fin}$ is $\card{H}$-torsion since $G_{\fin}$ is $\card{H}$-torsion.
    For $i = 1, \ldots, k$, denote by $p_i$ the projection $W \rightarrow W_i$, then $p_i(G_{\infty}) \geq W'_i$.
    Consider the group $p_i(N_{\infty})$. 
    \medskip
    
    \begin{adjustwidth}{3mm}{3mm}
        \textbf{Claim.} We have $p_i(N_{\infty}) \geq W'_i$ for all $1 \leq i \leq k$.
        
        \medskip
        \noindent$\triangleright$ Since $p_i(N_{\infty})$ is a normal subgroup of $p_i(G_{\infty})$, so by Proposition~\ref{cor:N_contains_Wprime} it either contains $W'_i$ or $\pi_{\Delta_i}(p_i(N_{\infty})) = \{\id\}$.
        However, if $\pi_{\Delta_i}(p_i(N_{\infty})) = \{\id\}$ then all elements of $N_{\infty}$ induce trivial permutations the blocks $\Delta_i$.
        Since $G_{\infty} \ge W'$ we can take an element $g \in G_{\infty}$ that induces a cycle of length $2|H|+1$ on the blocks $\Delta_i$, so $g N_{\infty} \in G_{\infty}/N_{\infty}$ is not $\card{H}$-torsion.
        This contradicts $G_{\infty}/N_{\infty} \cong G_{\fin}/N_{\fin}$ which is $\card{H}$-torsion. \hfill $\triangleleft$
    \end{adjustwidth}
    \medskip

    \begin{adjustwidth}{3mm}{3mm}
        \textbf{Claim.} We have $N_{\infty} \geq W'$.
        
        \medskip
        \noindent$\triangleright$ By the previous claim we have $p_i(N_{\infty}) \geq W'_i$ for all $1 \leq i \leq k$. 
        Therefore we can apply Lemma~\ref{lem:diagonal_or_fi_k} on the subgroup $N_{\infty}$ of $W_1 \times \cdots \times W_k$.
        We obtain that either $N_{\infty}$ acts diagonally on some pair $(\Delta_i, \Delta_j)$, or $N_{\infty} \geq W'$.
        However, since $N_{\infty}$ is normal in $G_{\infty}$, it cannot act diagonally on some pair $(\Delta_i, \Delta_j)$.
        This can be shown using arguments similar to case two in the proof of Lemma~\ref{lem:diagonal_or_fi_k}: we can take $a \in N_{\infty}$ and $h \in W' \leq G_{\infty}$ such that $a$ and $hah^{-1}$ act the same on $\Delta_i$ but differently on $\Delta_j$.
        We therefore obtain $N_{\infty} \geq W'$.
        \hfill $\triangleleft$
    \end{adjustwidth}
    We conclude that $G \geq \{\id\} \times N_{\infty} \geq \{\id\} \times W'$.
\end{proof}

Using Proposition~\ref{prop:structure_invariant_subgroup} we can decide the membership problem in $\rho$-invariant subgroups of $\FSym(\{1, \ldots, d\} \times \Z)$:

\begin{proposition}\label{prop:decidable_invariant_membership}
    It is decidable, given elements $\tau_1, \ldots, \tau_s, g \in \FSym(\{1, \ldots, d\} \times \Z)$, whether $g \in \gen{\tau_1, \ldots, \tau_s}_{\rho}$.
\end{proposition}
\begin{proof}
    Let $G = \gen{\tau_1, \ldots, \tau_s}_{\rho}$.
    By Proposition~\ref{prop:structure_invariant_subgroup}, $G \leq G_{\fin} \times G_{\infty}$, with 
    \[
    \textstyle
    G_{\fin} \leq \bigoplus_{\Z} H, \;\; \text{ and } \;\; G_{\infty} \leq W,
    \]
    such that $G \geq \{\id\} \times W'$.
    First we check whether $g \in \left(\bigoplus_{\Z} H\right) \times W$: this amounts to checking preservation of the block partitions and finitely many membership problems in $H$, and can hence be done effectively.
    If $g \notin \left(\bigoplus_{\Z} H\right) \times W$, then we already conclude that $g \notin G$; otherwise, we proceed as follows to check whether $g \in G$.
    Since
    \[
    g \in G \iff g \left(\{\id\} \times W'\right) \in G/\left(\{\id\} \times W'\right),
    \]
    it suffices to decide the membership problem in the $\rho$-invariant subgroup 
    \[
    \textstyle
    G/\left(\{\id\} \times W'\right) \leq (\bigoplus_{\Z} H) \times (W/W').
    \]
    Let $\pi_{W/W'}$ denote the projection onto $W/W'$.

    Write 
    \[
    G/\left(\{\id\} \times W'\right) = \gen{\tau_1, \ldots, \tau_s}_{\rho}
    \]
    with $\tau_1, \ldots, \tau_s \in (\bigoplus_{\Z} H) \times (W/W')$. 
    Since $W/W'$ is finite, there exists $\ell \geq 1$ such that $\rho^{\ell}$ induces the trivial automorphism on $W/W'$ (for example take $\ell = \card{W/W'}!$).
    Letting $\sigma \coloneqq \rho^{\ell}$, Then 
    \[
    G/\left(\{\id\} \times W'\right) = \gen{\tau_1, \ldots, \tau_s, \rho\tau_1\rho^{-1}, \ldots, \rho\tau_s\rho^{-1}, \ldots, \rho^{\ell-1}\tau_1\rho^{-(\ell-1)}, \ldots, \rho^{\ell-1}\tau_s\rho^{-(\ell-1)} }_{\sigma}
    \]

    We now consider $\bigoplus_{\Z} H$ as the direct sum $\bigoplus_{\Z} H^{\ell}$ by grouping $\ell$ consecutive components:
    \[
    \bigoplus_{\Z} H \xrightarrow{\sim} \bigoplus_{\Z} H^{\ell}, \quad (h_k)_{k \in \Z} \mapsto \left(h_{\ell k}, h_{\ell k + 1}, \ldots, h_{\ell k + (\ell-1)} \right)_{k \in \Z}.
    \]
    On $\bigoplus_{\Z} H^{\ell}$, conjugation by $\rho^{\ell}$ induces the one-step shift $(a_k)_{k \in \Z} \mapsto (a_{k-1})_{k \in \Z}$.  
    Let $S$ denote the set $\{\tau_1, \ldots, \tau_s, \rho\tau_1\rho^{-1}, \ldots, \rho\tau_s\rho^{-1}, \ldots, \rho^{\ell-1}\tau_1\rho^{-(\ell-1)}, \ldots, \rho^{\ell-1}\tau_s\rho^{-(\ell-1)}\}$, then it suffices to decide the membership problem in 
    \[
    \textstyle
    \gen{S}_{\sigma} \leq (\bigoplus_{\Z} H^{\ell}) \times (W/W'),
    \]
    where $\sigma$ acts trivially on $W/W'$.

    We now embed $\sigma$ and $(\bigoplus_{\Z} H^{\ell}) \times (W/W')$ in $(H^{\ell} \wr \Z) \times (W/W')$, where $\sigma$ acts as element $1 \in \Z$ in $H^{\ell} \wr \Z$ and trivially on $W/W'$.
    Consider the projection
    \[
    \pi \colon (H^{\ell} \wr \Z) \times (W/W') \rightarrow \Z \times \{\id\}.
    \]
    By Lemma~\ref{lem:ker_equals_inv}, $\gen{S}_{\sigma} = \gen{S \cup \{\sigma\}} \cap \ker(\pi)$.
    Therefore given any $\overline{g} \in (\bigoplus_{\Z} H^{\ell}) \times (W/W')$, we have
    \[
    \overline{g} \in \gen{S}_{\sigma} \iff \overline{g} \in \gen{S \cup \{\sigma\}}.
    \]
    Hence the membership problem in $\gen{S}_{\sigma}$ reduces to the Subgroup Membership problem in $(H^{\ell} \wr \Z) \times (W/W')$.
    Since $H^{\ell} \wr \Z$ has decidable Subgroup Membership (Lemma~\ref{lem:decidable_H_wreath_Z}), and $(H^{\ell} \wr \Z) \times (W/W')$ is a finite extension of $H^{\ell} \wr \Z$, we conclude that $(H^{\ell} \wr \Z) \times (W/W')$ also has decidable Subgroup Membership (Lemma~\ref{lem:finite_index_subgroup}).
\end{proof}

Finally, Proposition~\ref{prop:decidable_invariant_membership} gives us decidability of Submonoid Membership in $\FSym(\Z) \rtimes \Z$:

\thmmain*

\begin{proof}
    By Lemma~\ref{lem:submonoid_to_subgroup}, Submonoid Membership in $\FSym(\Z) \rtimes \Z$ reduces to Subgroup Membership.
    By Proposition~\ref{prop:subgroup_to_invariant_subgroup}, Subgroup Membership in $\FSym(\Z) \rtimes \Z$ reduces to the membership problem in $\rho$-invariant subgroups of $\FSym(\{1, \ldots, d\} \times \Z)$, which is decidable by Proposition~\ref{prop:decidable_invariant_membership}.
\end{proof}

\bibliography{bibliography}

@article{metric_in_Houghton,
    title={Commensurations and metric properties of {H}oughton’s groups},
    author={Burillo, Jos{\'e} and Cleary, Sean and Martino, Armando and R{\"o}ver, Claas},
    journal={Pacific Journal of Mathematics},
    volume={285},
    number={2},
    pages={289--301},
    year={2016}}

@article{conjugacy_in_Houghton,
    author = {Antol{\'{\i}}n, Yago and Burillo, Jos{\'e} and Martino, Armando},
    title = {Conjugacy in {Houghton}'s groups.},
    journal = {Publicacions Matem{\`a}tiques},
    volume = {59},
    number = {1},
    pages = {3--16},
    year = {2015}}

@misc{neumann2023infinite,
  title={Infinite Permutation Groups},
  author={Neumann, Peter M.},
  note={\url{https://arxiv.org/abs/2307.11564}},
  year={2023}
}

@article{neumann1976structure,
  title={The structure of finitary permutation groups},
  author={Neumann, Peter M.},
  journal={Archiv der Mathematik},
  volume={27},
  number={1},
  pages={3--17},
  year={1976},
  publisher={Springer}
}

@article{neumann1975lawlessness,
  title={The lawlessness of groups of finitary permutations},
  author={Neumann, Peter M.},
  journal={Archiv der Mathematik},
  volume={26},
  number={1},
  pages={561--566},
  year={1975},
  publisher={Springer}
}

@book{scott1987group,
  title={Group theory},
  author={Scott, William R.},
  year={1987},
  publisher={Courier Corporation}
}

@article{holder1895bildung,
  title={Bildung zusammengesetzter {G}ruppen},
  author={H{\"o}lder, Otto},
  journal={Mathematische Annalen},
  volume={46},
  number={3},
  pages={321--422},
  year={1895},
  publisher={Springer}
}

@article{lohrey2015rational,
  title={Rational subsets and submonoids of wreath products},
  author={Lohrey, Markus and Steinberg, Benjamin and Zetzsche, Georg},
  journal={Information and Computation},
  volume={243},
  pages={191--204},
  year={2015},
  publisher={Elsevier}
}

@book{grunschlag1999algorithms,
  title={Algorithms in geometric group theory},
  author={Grunschlag, Zeph},
  year={1999},
  publisher={University of California, Berkeley}
}

@misc{bodart2025,
  title={On finite extensions of lamplighter groups},
  author={Bodart, Corentin},
  year={2025},
  note={\url{https://arxiv.org/abs/2507.13203}}
}

@inproceedings{dong2025submonoid,
  author       = {Ruiwen Dong},
  title        = {Submonoid {M}embership in $n$-Dimensional Lamplighter Groups and {S}-Unit
                  Equations},
  booktitle    = {52nd International Colloquium on Automata, Languages, and Programming (ICALP)},
  series       = {LIPIcs},
  volume       = {334},
  pages        = {154:1--154:20},
  publisher    = {Schloss Dagstuhl - Leibniz-Zentrum f{\"{u}}r Informatik},
  year         = {2025}
}

@inproceedings{lohrey2024membership,
  title={Membership problems in infinite groups},
  author={Lohrey, Markus},
  booktitle={Conference on Computability in Europe (CiE)},
  pages={44--59},
  year={2024},
  organization={Springer}
}

@inproceedings{cadilhac2020rational,
  author    = {Micha{\"{e}}l Cadilhac and
               Dmitry Chistikov and
               Georg Zetzsche},
  title     = {Rational Subsets of {B}aumslag-{S}olitar Groups},
  booktitle = {47th International Colloquium on Automata, Languages, and Programming,
               (ICALP)},
  series    = {LIPIcs},
  volume    = {168},
  pages     = {116:1--116:16},
  publisher = {Schloss Dagstuhl - Leibniz-Zentrum f{\"{u}}r Informatik},
  year      = {2020},
  doi={10.4230/LIPIcs.ICALP.2020.116}
}

@inproceedings{colcombet2019reachability,
  author    = {Thomas Colcombet and
               Jo{\"{e}}l Ouaknine and
               Pavel Semukhin and
               James Worrell},
  title     = {On Reachability Problems for Low-Dimensional Matrix Semigroups},
  booktitle = {46th International Colloquium on Automata, Languages, and Programming
               (ICALP)},
  series    = {LIPIcs},
  volume    = {132},
  pages     = {44:1--44:15},
  publisher = {Schloss Dagstuhl - Leibniz-Zentrum f{\"{u}}r Informatik},
  year      = {2019}
}

@inproceedings{potapov2017decidability,
  title={Decidability of the {M}embership {P}roblem for $2 \times 2$ integer matrices},
  author={Potapov, Igor and Semukhin, Pavel},
  booktitle={Proceedings of the Twenty-Eighth Annual ACM-SIAM Symposium on Discrete Algorithms (SODA)},
  pages={170--186},
  year={2017},
  organization={SIAM}
}

@inproceedings{babai1996multiplicative,
  title={Multiplicative equations over commuting matrices},
  author={Babai, L{\'a}szl{\'o} and Beals, Robert and Cai, Jin-yi and Ivanyos, G{\'a}bor and Luks, Eugene M.},
  booktitle={Proceedings of the Seventh Annual ACM-SIAM Symposium on Discrete Algorithms (SODA)},
  pages={498--507},
  year={1996}
}

@article{bodart2025membership,
  title={Membership problems in nilpotent groups},
  author={Bodart, Corentin},
  journal={Journal of Groups, Complexity, Cryptology},
  volume = {Volume 17, Issue 1},
  eid = 3,
  pages={1--25},
  year={2025},
  publisher={Episciences.org}
}

@article{roman2022undecidability,
	title={Undecidability of the submonoid membership problem for a sufficiently large finite direct power of the {H}eisenberg group},
	author={Roman'kov, Vitaly},
	journal={Siberian Electronic Mathematical Reports},
	volume={20},
	number={1},
	pages={293--305},
	year={2023}}

@misc{shafrir2024decidability,
  title={Is decidability of the {S}ubmonoid {M}embership {P}roblem closed under finite extensions?},
  author={Shafrir, Doron},
  note={\url{https://arxiv.org/abs/2405.12921}},
  year={2024}
}

@article{lohrey2011tilings,
  title={Tilings and submonoids of metabelian groups},
  author={Lohrey, Markus and Steinberg, Benjamin},
  journal={Theory of Computing Systems},
  volume={48},
  number={2},
  pages={411--427},
  year={2011},
  publisher={Springer}
}

@article{markov1951certain,
  title={On an unsolvable problem concerning matrices},
  author={Markov, A. A.},
  journal={Doklady Akad. Nauk SSSR},
  volume={78},
  number={6},
  pages={1089–1092},
  year={1951}
}

@article{mikhailova1958occurrence,
  title={The entrance problem for direct unions of groups},
  author={Mikhailova, K. A.},
  journal={Doklady Akad. Nauk SSSR},
  volume={119},
  number={6},
  pages={1103–1105},
  year={1958}
}

@inproceedings{lohrey2023complexity,
  title={On the complexity of diameter and related problems in permutation groups},
  author={Lohrey, Markus and Rosowski, Andreas},
  booktitle={50th International Colloquium on Automata, Languages, and Programming (ICALP)},
  pages={134--1},
  year={2023},
  organization={Schloss Dagstuhl--Leibniz-Zentrum f{\"u}r Informatik}
}

@inproceedings{sims1970computational,
  title={Computational methods in the study of permutation groups},
  author={Sims, Charles C.},
  booktitle={Computational problems in abstract algebra},
  pages={169--183},
  year={1970},
  organization={Elsevier}
}

@article{knuth1991efficient,
  title={Efficient representation of perm groups},
  author={Knuth, Donald E.},
  journal={Combinatorica},
  volume={11},
  number={1},
  pages={33--43},
  year={1991},
  publisher={Springer}
}

@inproceedings{dahan2025group,
  author       = {Anatole Dahan},
  title        = {Group Order Logic},
  booktitle    = {40th Annual {ACM/IEEE} Symposium on Logic in Computer Science (LICS)},
  pages        = {236--248},
  publisher    = {{IEEE}},
  year         = {2025}
}

@inproceedings{sims1971computation,
  title={Computation with permutation groups},
  author={Sims, Charles C.},
  booktitle={Proceedings of the second ACM symposium on Symbolic and algebraic manipulation},
  pages={23--28},
  year={1971}
}

@inproceedings{ghosh2023orbit,
  title={Orbit-finite linear programming},
  author={Ghosh, Arka and Hofman, Piotr and Lasota, S{\l}awomir},
  booktitle={38th Annual ACM/IEEE Symposium on Logic in Computer Science (LICS)},
  pages={1--14},
  year={2023},
  organization={IEEE}
}

@article{cameron1990oligomorphic,
  title={Oligomorphic permutation groups},
  author={Cameron, Peter J.},
  journal={Perspectives in mathematical sciences. II},
  volume={8},
  pages={37--61},
  year={1990},
  publisher={World Scientific}
}

@inproceedings{bojanczyk2011automata,
  title={Automata with group actions},
  author={Boja{\'n}czyk, Miko{\l}aj and Klin, Bartek and Lasota, S{\l}awomir},
  booktitle={26th Annual IEEE Symposium on Logic in Computer Science (LICS)},
  pages={355--364},
  year={2011},
  organization={IEEE}
}

@article{houghton1978first,
  title={The first cohomology of a group with permutation module coefficients},
  author={Houghton, C. H.},
  journal={Archiv der Mathematik},
  volume={31},
  number={1},
  pages={254--258},
  year={1978},
  publisher={Springer}
}

@article{brown1987finiteness,
  title={Finiteness properties of groups},
  author={Brown, Kenneth S.},
  journal={Journal of Pure and Applied Algebra},
  volume={44},
  number={1-3},
  pages={45--75},
  year={1987},
  publisher={North-Holland}
}

@article{wiegold1977transitive,
  title={Transitive groups with fixed-point-free permutations {II}},
  author={Wiegold, James},
  journal={Archiv der Mathematik},
  volume={29},
  number={1},
  pages={571--573},
  year={1977},
  publisher={Springer}
}

@article{bojanczyk2024orbit,
  title={Orbit-finite-dimensional vector spaces and weighted register automata},
  author={Boja{\'n}czyk, Miko{\l}aj and Fijalkow, Joanna and Klin, Bartek and Moerman, Joshua},
  journal={TheoretiCS},
  volume={3},
  year={2024},
  publisher={Episciences.org}
}

@article{ackerman2016invariant,
  author  = {Ackerman, Nathanael and Freer, Cameron and Patel, Rehana},
  title   = {Invariant Measures Concentrated on Countable Structures},
  journal = {Forum of Mathematics, Sigma},
  volume  = {4},
  pages   = {e17},
  year    = {2016},
  doi     = {10.1017/fms.2016.15},
  publisher = {Cambridge University Press}
}

@article{cox2025finiteness,
  title={Finiteness properties of subgroups of Houghton groups of full Hirsch length},
  author={Cox, Charles Garnet and Kropholler, Peter H. and Martino, Armando},
  journal={Groups, Geometry, and Dynamics},
  year={2026}
}

@inproceedings{haase2019presburger,
  title={Presburger arithmetic with stars, rational subsets of graph groups, and nested zero tests},
  author={Haase, Christoph and Zetzsche, Georg},
  booktitle={34th Annual ACM/IEEE Symposium on Logic in Computer Science (LICS)},
  pages={1--14},
  year={2019},
  organization={IEEE}
}

@article{choffrut2005some,
  title={Some decision problems on integer matrices},
  author={Choffrut, Christian and Karhum{\"a}ki, Juhani},
  journal={RAIRO-Theoretical Informatics and Applications-Informatique Th{\'e}orique et Applications},
  volume={39},
  number={1},
  pages={125--131},
  year={2005}
}

@inproceedings{hrushovski2018polynomial,
  title={Polynomial invariants for affine programs},
  author={Hrushovski, Ehud and Ouaknine, Jo{\"e}l and Pouly, Amaury and Worrell, James},
  booktitle={33rd Annual ACM/IEEE Symposium on Logic in Computer Science (LICS)},
  pages={530--539},
  year={2018}
}

@article{blondel2000survey,
  title={A survey of computational complexity results in systems and control},
  author={Blondel, Vincent D. and Tsitsiklis, John N.},
  journal={Automatica},
  volume={36},
  number={9},
  pages={1249--1274},
  year={2000},
  publisher={Elsevier}
}

@article{derksen2005quantum,
  title={Quantum automata and algebraic groups},
  author={Derksen, Harm and Jeandel, Emmanuel and Koiran, Pascal},
  journal={Journal of Symbolic Computation},
  volume={39},
  number={3-4},
  pages={357--371},
  year={2005},
  publisher={Elsevier}
}

@article{Higman,
    author = {Graham Higman},
    title = {On infinite simple permutation groups},
    journal={Publicationes Mathematicae Debrecen},
    year = {1954},
    volume = {3},
    pages = {221-226}}

@article{silva2024poisson,
  title={The {P}oisson boundary of lampshuffler groups},
  author={Silva, Eduardo},
  journal={Mathematische Zeitschrift},
  volume={308},
  number={3},
  pages={46},
  year={2024},
  publisher={Springer}
}

@misc{genevois2024lamplighter,
  title={Lamplighter-like geometry of groups},
  author={Genevois, Anthony and Tessera, Romain},
  note={\url{https://arxiv.org/abs/2401.13520}},
  year={2024}
}

@misc{bonnet2022twin,
  title={Twin-width VII: groups},
  author={Bonnet, {\'E}douard and Geniet, Colin and Tessera, Romain and Thomass{\'e}, St{\'e}phan},
  note={\url{https://arxiv.org/abs/2204.12330}},
  year={2022}
}

@misc{gheysens2022between,
  title={Between free and direct products of groups},
  author={Gheysens, Maxime and Monod, Nicolas},
  note={\url{https://arxiv.org/abs/2201.03625}},
  year={2022}
}

@article{Gray2020OneRelatorInverseMonoids,
  author       = {Robert D. Gray},
  title        = {Undecidability of the Word Problem for One-Relator Inverse Monoids via Right-Angled {A}rtin Subgroups of One-Relator Groups},
  journal      = {Inventiones Mathematicae},
  volume       = {219},
  number       = {3},
  pages        = {987--1008},
  year         = {2020},
  doi          = {10.1007/s00222-019-00920-2},
  publisher    = {Springer}
}

@article{foniqi2026membership,
  title={Membership problems for positive one-relator groups and one-relation monoids},
  author={Foniqi, Islam and Gray, Robert D. and Nyberg-Brodda, Carl-Fredrik},
  journal={Canadian Journal of Mathematics},
  volume={78},
  number={3},
  pages={802--841},
  year={2026},
  publisher={Canadian Mathematical Society}
}

@article{romanovskii1974some,
  title={Some algorithmic problems for solvable groups},
  author={Romanovskii, N. S.},
  journal={Algebra and Logic},
  volume={13},
  number={1},
  pages={13--16},
  year={1974},
  publisher={Springer},
  doi={10.1007/BF01462922}
}

@misc{dong2026standard,
  title={Standard bases for shift-stable groups and {S}ubgroup {M}embership in wreath products},
  author={Dong, Ruiwen},
  year={2026},
  note={\url{https://arxiv.org/abs/2608.24620}}
}

@article{GrigorchukWilson2003,
  author  = {Grigorchuk, R. I. and Wilson, J. S.},
  title   = {A structural property concerning abstract commensurability of subgroups},
  journal = {Journal of the London Mathematical Society},
  series  = {2},
  volume  = {68},
  number  = {3},
  pages   = {671--682},
  year    = {2003},
  doi     = {10.1112/S0024610703004745}
}

@article{kopytov1968solvability,
  title={Solvability of the problem of occurrence in finitely generated soluble groups of matrices over the field of algebraic numbers},
  author={Kopytov, V. M.},
  journal={Algebra and Logic},
  volume={7},
  number={6},
  pages={388--393},
  year={1968},
  publisher={Springer}
}

\end{document}